\documentclass[leqno]{article}

\usepackage{amsmath,amsthm}
\usepackage[psamsfonts]{amssymb}

\newtheorem{theorem}{Theorem}[section]
\newtheorem{proposition}[theorem]{Proposition}
\newtheorem{lemma}[theorem]{Lemma}
\newtheorem{corollary}[theorem]{Corollary}

\theoremstyle{definition}

\theoremstyle{remark}
\newtheorem{remark}[theorem]{Remark}

\newtheorem{definition}[theorem]{Definition}
\newtheorem{example}[theorem]{Example}

\begin{document}

\title{\textbf{Geometric structures \\
associated with the Chern connection \\
attached to a SODE}}

\author{\textsc{J. Mu\~{n}oz Masqu\'e}\\
{\small Instituto de F\'{\i}sica Aplicada, CSIC}\\
{\small C/ Serrano 144, 28006-Madrid, Spain} \\
{\small \emph{E-mail:\/} jaime@iec.csic.es}
\and
\textsc{M. Eugenia Rosado Mar\'{\i}a}\\
{\small Departamento de Matem\'{a}tica Aplicada}\\
{\small Escuela T\'ecnica Superior de Arquitectura, UPM}\\
{\small Avda.\ Juan de Herrera 4, 28040-Madrid, Spain}\\
{\small \emph{E-mail:\/} eugenia.rosado@upm.es}}

\date{}

\maketitle

\begin{abstract}
\noindent To each second-order ordinary differential
equation $\sigma $ on a smooth manifold $M$
a $G$-structure $P^\sigma $ on $J^1(\mathbb{R},M)$
is associated and the Chern connection $\nabla ^\sigma $
attached to $\sigma $ is proved to be reducible to
$P^\sigma $; in fact, $P^\sigma $ coincides generically
with the holonomy bundle of $\nabla ^\sigma $.
The cases of unimodular and orthogonal holonomy
are also dealt with. Two characterizations of the Chern
connection are given: The first one in terms
of the corresponding covariant derivative and
the second one as the only principal connection
on $P^\sigma $ with prescribed torsion tensor field.
The properties of the curvature tensor field
of $\nabla ^\sigma $ in relationship to the existence
of special coordinate systems for $\sigma $ are
studied. Moreover, all the odd-degree characterictic
classes on $P^\sigma $ are seen to be exact
and the usual characteristic classes induced
by $\nabla ^\sigma $ determine the Chern classes
of $M$. The maximal group of automorphisms
of the projection
$p\colon \mathbb{R}\times M\to \mathbb{R}$ with respect
to which $\nabla ^\sigma $ has a functorial behaviour,
is proved to be the group of $p$-vertical automorphisms.
The notion of a differential invariant under such a group
is defined and stated that second-order differential
invariants factor through the curvature mapping;
a structure is thus established for KCC theory.
\end{abstract}


\smallskip


\smallskip

\noindent \emph{Acknowledgements:\/}
Supported by Ministerio of Ciencia
e Innovaci\'on of Spain, under grant
\#MTM2008--01386.

\section{Introduction}

S.-S. Chern introduced several
connections---currently referred
to as `Chern connections'---in
different geometric settings; notably,
in: i) the theory of second-order
ordinary differential equations
\cite{Chern} in the context
of the correspondence between Cartan
and Kosambi (cf.\ \cite{Cartan},
\cite{Kosambi}), ii) Finsler geometry
(see \cite{Chern3} and the references
therein), iii) $3$-web geometry
\cite{Chern2}, and iv) almost Hermitian
manifolds \cite{Chern2.5} (the second
canonical connection, cf.\ \cite{Lich}).
These connections usually play a role
in modern research on such topics; e.g.,
see \cite{Anastasiei}, \cite{AB},
\cite{BaoChernShen}, \cite{DGM},
\cite{Muzsnay}, \cite{Nakai},
among others.

Due to their importance, Chern classes
have eclipsed in part the rest of the works
of Chern (cf.\ \cite{Wu}); in particular,
this happens with the connections above.

In this paper, we consider exclusively
the Chern connection of item i).

Let $M$ be a connected $C^\infty $ manifold
with natural projections
$p\colon \mathbb{R}\times M\to \mathbb{R}$
and $p^\prime \colon \mathbb{R}\times M\to M$.
The bundle of $r$-jets of smooth maps from
$\mathbb{R}$ into $M$ is denoted by
$p^r\colon M^r=J^r(\mathbb{R},M)\to \mathbb{R}$,
with natural projections $p^{rs}\colon M^r\to M^s$
for $r>s$, and the $r$-jet prolongation of the curve
$\gamma \colon \mathbb{R}\to M$ is denoted by
$j^r\gamma \colon \mathbb{R}\to M^r$. Every coordinate
system $(x^i)$, $1\leq i\leq n=\dim  M$, on $M$ induces
a coordinate system $(t,x^i;x^j_k)$, $i,j=1,\dotsc,n$,
$0\leq k\leq r$ on $M^r$ as follows:
\[
x^j_k(j^r_{t_0}\gamma )
=\frac{d^k(x^j\circ \gamma )}{dt^k}(t_0),
\quad x^j_0=x^j.
\]
Below, however, for first and second orders
we use the more usual `dot' notation; namely,
$\dot{x}^j=x^j_1$, and $\ddot{x}^j=x^j_2$.

A second-order ordinary differential equation
\begin{equation}
\label{SODE}
\ddot{x}^i=F^i
\left(
t,x^i,\dot{x}^i
\right) ,
\quad
F^i\in C^\infty (M^1),
\; 1\leq i\leq n,
\end{equation}
can be better understood as a section
$\sigma \colon M^1\to M^2$
of the map $p^{21}\colon M^2\to M^1$
by simply setting
$\ddot{x}^i\circ \sigma =F^i$.
The correspondence
$\sigma \leftrightarrow (F^i)_{i=1}^n$
is natural and bijective.

\begin{remark}
\label{remark1}
The second-order ordinary differential
equation considered in the original paper
by Chern \cite{Chern} (also see \cite{Cartan},
\cite{Kosambi}) is
$\ddot{x}^i+F^i(t,x^i,\dot{x}^i)=0$,
$1\leq i\leq n$, instead of \eqref{SODE}.
Hence, in all the formulas below, $F^i$
should be replaced by $-F^i$ in order to compare
with the formulas in \cite{Chern}.
\end{remark}

Since its introduction in
\cite{Cartan, Chern, Kosambi},
the Chern connection associated to a SODE on $M$
has been studied by several authors; e.g., see
\cite{Byrnes1, Byrnes2, CrampinMartinezSarlet, MassaPagani}.

In the sections \ref{section_G-structure},
\ref{section_characterizations} below,
the Chern connection is presented in a similar
way as the Levi-Civita connection is introduced
in Riemannian Geometry; namely,
\begin{enumerate}
\item
The notion of a linear frame of $M^1$ `adapted'
to a SODE $\sigma $ is defined (corresponding
to the notion of an orthonormal frame
in the Riemannian case).
\item
The structure of the set of such frames,
is analysed and proved to be a $G$-structure
$P^\sigma $ of the linear frames of $M^1$;
the first prolongation (cf.\ \cite{Guillemin, LL, L})
of the Lie algebra of the Lie group of $P^\sigma $
vanishes and the adjoint bundle of $P^\sigma $
is described in terms of the geometry of $\sigma $.
The Chern connection $\nabla ^\sigma $ is then
proved to be reducible to this $G$-structure.
\item
Two caracterizations of the Chern connection
are provided: The first one characterizes
the covariant derivative $\nabla ^\sigma $
as a derivation law on the tangent bundle
of $M^1$ and the second one as a principal
connection on $P^\sigma $ with prescribed
torsion tensor field. (Observe also the analogy
with the Riemannian case.)
\end{enumerate}

In the section \ref{section_curvature},
given a SODE $\sigma $ on $M$, the existence
of a fibred coordinate system $(t,x^1,\dotsc,x^n)$
for the projection $p\colon M^0\to \mathbb{R}$
such that,
\[
\ddot{x}^i\circ \sigma
=f_0^i(t,x^1,\dotsc,x^n)
+f_j^i(t,x^1,\dotsc,x^n)\dot{x}^j,
\quad 1\leq i\leq n,
\]
is given in terms of the curvature tensor
of $\nabla ^\sigma $. The particular cases
$f^i_0\equiv 0$ and $f^i_0\equiv 0$,
$f^i_j\equiv 0$ are specially considered.

The odd-degree characteristic classes on
$P^\sigma $---as defined in \cite{AM}---are
seen to be exact and the standard
characteristic classes induced by
$\nabla ^\sigma $ are shown to determine
the Chern classes of the ground manifold
$M$ by means of the natural isomorphism
$H^\bullet (M^1;\mathbb{R})
\cong H^\bullet (M;\mathbb{R})$.

In the section \ref{section_naturality}
the functorial character of the Chern connection
is studied. We confine ourselves to consider
the group of $p$-vertical automorphisms
of the submersion $p$, denoted by $\mathrm{Aut}^v(p)$
(see its precise definition at the beginning
of this section), as this subgroup is the largest
group of transformations for which the functoriality
holds. This corresponds to the ``problem (B)''
in the terminology introduced by Chern in \cite{Chern}.
Chern solves the problem of functoriality
of his connection by simply saying (cf.\
\cite[p.\ 208]{Chern}):
``Il en r\'{e}sulte que la choix $\beta ^i_k
=\frac{1}{2}\frac{\partial F^i}{\partial y^k}$
a un caract\`{e}re intrins\`{e}que.''

Finally, the section \ref{section_invariants}
is devoted to introduce the notion of a differential
invariant for SODEs with respect to the group
$\mathrm{Aut}^v(p)$. The main result states
that invariant functions factor through
the curvature mapping induced by the curvature
$K^\sigma $ of the splitting $H^\sigma $ attached
to each SODE $\sigma $ (see the formulas
\eqref{Hsigma} and \eqref{K_sigma} below),
which almost coincides with the torsion tensor
field of the Chern connection $\nabla ^\sigma $,
see the formulas \eqref{torsion}, \eqref{Torsion}.
This explains the role of the Chern connection
in KCC theory. We also remark the similarity
between the aforementioned result and the geometric
version of the Utiyama theorem (e.g., see
\cite[10.2]{Bleecker}) in gauge theories.

\section{Preliminaries}

\subsection{Dynamical flows}

Every SODE $\sigma $ defines a vector field
$X^\sigma \in \mathfrak{X}(M^1)$, called `the dynamical
flow' associated to $\sigma $ (cf.\ \cite{MassaPagani}),
as follows:
$(X^\sigma )_\xi =(j^1\gamma )_\ast (d/dt)_{t_0}$,
$\forall \xi \in (p^1)^{-1}(t_0)$, where
$\gamma ^i=x^i\circ \gamma $, $1\leq i\leq n$,
is the only solution to \eqref{SODE} satisfying
the initial conditions
$\gamma ^i(t_0)=x^i(\xi )$ and
$d\gamma ^i/dt(t_0)=\dot{x}^i(\xi )$.
The local expression of the dynamical flow is
$X^\sigma =\partial/\partial t
+\dot{x}^i\partial/\partial x^i
+F^i\partial/\partial \dot{x}^i$.

\subsection{The splitting induced
by a SODE\label{splitting}}

As is known, $p^{10}\colon M^1\to M^0
=\mathbb{R}\times M$ is an affine bundle
modelled over $p^{\prime \ast }TM$;
in fact, given $v\in T_{x_0}M$ and
$j_{t_0}^1\gamma \in M^1$,
with $\gamma (t_0)=x_0$, then
$v+j_{t_0}^1\gamma
=j_{t_0}^1\gamma ^\prime $ is defined
as follows: 1) $\gamma ^\prime (t_0)
=x_0$, and 2) $\gamma _\ast ^\prime
(d/dt)_{t_0}=v+\gamma _\ast (d/dt)_{t_0}$.
Hence, the following exact sequence
of vector bundles over $M^1$ holds:
\begin{equation}
\label{exactseq}
0\to
\left(
p^\prime \circ p^{10}
\right) ^\ast
TM\overset{\varepsilon}{\longrightarrow}V
\left( p^{10}\right) \to TM^1
\overset{(p^{10})_\ast }{-\!\!\!\!\longrightarrow }
(p^{10})^\ast TM^0\to 0,
\end{equation}
where $V\left( p^{10}\right) $ denotes the vector
subbundle of $p^{10}$-vertical vectors
and $\varepsilon $ is defined by the directional
derivative, namely
\begin{equation}
\label{epsilon}
\varepsilon
\left(
j_{t_0}^1\gamma ,v
\right) (f)
=\lim _{t\to 0}
\frac{f(tv+j_{t_0}^1\gamma )
-f(j_{t_0}^1\gamma )}{t},
\quad
v\in T_{\gamma (t_0)}M,\;f\in C^\infty (M^1).
\end{equation}
In local coordinates $\varepsilon $
is determined by
$\varepsilon (\partial/\partial x^i)
=\partial/\partial \dot{x}^i$.

Given a SODE $\sigma $, the Lie derivative
of the fundamental tensor
\begin{equation}
\label{J}
J=\omega ^i
\otimes \frac{\partial }{\partial \dot{x}^i},
\quad
\omega ^i=dx^i-\dot{x}^idt
\end{equation}
of $M^1$ (e.g., see
\cite[formula (1.13)]{MassaPagani})
along $X^\sigma $ is
\begin{align}
L_{X^\sigma }J
& =-\left(
dx^i-\dot{x}^idt
\right)
\otimes \frac{\partial }{\partial x^i}
\label{LXJ} \\
& \quad
+\left\{
\left(
\dot{x}^i
\frac{\partial F^j}{\partial \dot{x}^i}-F^j
\right)
dt-\frac{\partial F^j}
{\partial \dot{x}^i}dx^i+d\dot{x}^j
\right\}
\otimes \frac{\partial }{\partial \dot{x}^j},
\nonumber
\end{align}
and it is readily checked that $L_{X^\sigma }J$
is diagonalizable with eigenvalues $0$, $+1$, $-1$,
and multiplicities $1$, $n$, $n$, respectively
(cf.\ \cite[p. 6620]{Byrnes1}). If $T^0(M^1)$,
$T^+(M^1)$, $T^-(M^1)$ are the corresponding
vector subbundles of eigenvectors, then
\begin{align}
T(M^1) & =T^0(M^1)\oplus T^-(M^1)\oplus T^+(M^1),
\label{decomposition}\\
T^0(M^1) &
=\left\langle X^\sigma
\right\rangle ,\label{T^0}\\
T^-(M^1) &
=\left\langle X_i^\sigma
\right\rangle ,\;X_i^\sigma
=\frac{\partial }{\partial x^i}
+\tfrac{1}{2}\frac{\partial F^j}
{\partial \dot{x}^i}
\frac{\partial }{\partial \dot{x}^j},
\label{T^-}\\
T^+(M^1) &
=V\left(
p^{10}
\right)
=\left\langle
\frac{\partial }{\partial \dot{x}^1},\dotsc,
\frac{\partial }{\partial \dot{x}^n}
\right\rangle .
\label{T^+}
\end{align}
Hence, the epimorphism $(p^{10})_\ast $
in \eqref{exactseq} induces an isomorphism
\begin{equation}
(p^{10})_\ast \colon T^0(M^1)\oplus T^-(M^1)
\overset{\cong }{\longrightarrow }(p^{10})^\ast
TM^0, \label{isomorphism}
\end{equation}
whose inverse mapping determines a section
$H^\sigma \colon(p^{10})^\ast TM^0\to TM^1$
of $(p^{10})_\ast $ given by
\begin{equation}
H^\sigma =dt\otimes X^\sigma
+\omega ^i\otimes X_i^\sigma ,\label{Hsigma}
\end{equation}
and consequently, the exact sequence
\eqref{exactseq} splits; i.e.,
every tangent vector $X$ in $TM^1$
can uniquely be written as $X=X^v+X^h$,
where
\[
X^h=H^\sigma
\left(
(p^{10})_\ast X
\right) \in T^0(M^1)
\oplus T^-(M^1),
\quad
X^v=X-X^h\in V(p^{10}).
\]
In coordinates,{\small
\begin{equation}
\begin{array}
[c]{ll}
\left(
\dfrac{\partial }{\partial t}
\right) ^v
=\left(
\tfrac{1}{2}
\dot{x}^i\dfrac{\partial F^j}
{\partial \dot{x}^i}-F^j
\right)
\dfrac{\partial }{\partial \dot{x}^j},
&
\left(
\dfrac{\partial }{\partial t}
\right) ^h
=\dfrac{\partial }{\partial t}
+\left(
F^j
-\tfrac{1}{2}
\dot{x}^i
\dfrac{\partial F^j}{\partial \dot{x}^i}
\right)
\dfrac{\partial }{\partial \dot{x}^j},
\smallskip \\
\left(
\dfrac{\partial }{\partial x^i}
\right) ^v
=-\tfrac{1}{2}
\dfrac{\partial F^j}{\partial \dot{x}^i}
\dfrac{\partial }{\partial \dot{x}^j},
& \left(
\dfrac{\partial }{\partial x^i}
\right) ^h
=X_i^\sigma ,\smallskip\\
\left(
\dfrac{\partial }{\partial \dot{x}^i}
\right) ^v
=\dfrac{\partial }{\partial \dot{x}^i},
& \left(
\dfrac{\partial }{\partial \dot{x}^i}
\right) ^h=0.
\end{array}
\label{vH}
\end{equation}
}

Below we need the explicit expression
for the curvature form of the splitting
$H^\sigma $; i.e.,
\begin{equation}
\begin{array}
[c]{l}
K^\sigma \in
\bigwedge ^2
T^\ast M^1\otimes V(p^{10}),\\
K^\sigma (X,Y)=\left[ X^h,Y^h\right] ^v,
\quad \forall X,Y\in \mathfrak{X}(M^1).
\end{array}
\label{curvature_form}
\end{equation}

\begin{proposition}
\label{Prop1}The curvature form
of the splitting \emph{\eqref{Hsigma}}
is given as follows:
\begin{equation}
K^\sigma
=-\Bigl(
P_j^hdt\wedge \omega ^j
+\sum _{i<j}T_{ij}^h\omega ^i\wedge \omega ^j
\Bigr)
\otimes \frac{\partial }{\partial \dot{x}^h},
\label{K_sigma}
\end{equation}
where the $1$-form $\omega ^j$ is
introduced in the formula \emph{\eqref{J}}
and
\begin{align}
T_{ij}^k & =\tfrac{1}{2}
\left(
\frac{\partial^2F^k}
{\partial x^i\partial \dot{x}^j}
-\frac{\partial^2F^k}
{\partial x^j\partial \dot{x}^i}
+\tfrac{1}{2}
\left(
\frac{\partial F^h}{\partial \dot{x}^i}
\frac{\partial^2F^k}
{\partial \dot{x}^h\partial \dot{x}^j}
-\frac{\partial F^h}{\partial \dot{x}^j}
\frac{\partial^2F^k}
{\partial \dot{x}^h\partial \dot{x}^i}
\right)
\right) ,\label{T's}\\
P_j^i & =\tfrac{1}{2}X^\sigma
\left(
\frac{\partial F^i}{\partial \dot{x}^j}
\right)
-\frac{\partial F^i}{\partial x^j}
-\tfrac{1}{4}
\frac{\partial F^k}{\partial \dot{x}^j}
\dfrac{\partial F^i}{\partial \dot{x}^k}.
\label{P's}
\end{align}
\end{proposition}

\begin{remark}
\label{remark2}
The formulas \eqref{T's}, \eqref{P's}
coincide with those in \cite[formula (17)]{Chern}
after replacing $F^i$ by $-F^i$, as explained
in Remark \ref{remark1}.
\end{remark}

\begin{proof}
As a simple computation shows, the following
formulas hold:
\[
\begin{array}
[c]{ll}
\left[
X^\sigma ,X_j^\sigma
\right]
=P_j^i\dfrac{\partial }{\partial \dot{x}^i}
-\tfrac{1}{2}
\dfrac{\partial F^k}{\partial \dot{x}^j}
X_k^\sigma ,
&
\left[
X^\sigma ,
\dfrac{\partial }{\partial \dot{x}^j}
\right]
=-X_j^\sigma
-\tfrac{1}{2}
\dfrac{\partial F^i}{\partial \dot{x}^j}
\dfrac{\partial }{\partial \dot{x}^i},
\smallskip \\
\left[  X_i^\sigma ,X_j^\sigma
\right]
=T_{ij}^k
\dfrac{\partial }{\partial \dot{x}^k},
& \left[
X_i^\sigma ,
\dfrac{\partial }{\partial \dot{x}^j}
\right]
=-\tfrac{1}{2}
\dfrac{\partial ^2F^h}
{\partial \dot{x}^j\partial \dot{x}^i}
\dfrac{\partial }{\partial \dot{x}^h},
\end{array}
\]
and we can conclude by using Table \ref{vH}.
\end{proof}

\section{$G$-structure attached to a SODE}
\label{section_G-structure}

\begin{definition}
\label{definition1}
A linear frame
$(X_0,X_1,\dotsc,X_{2n})\in F_\xi (M^1)$,
$\xi\in M^1$, is said to be \emph{adapted}
to a SODE $\sigma $ on $M$
if it satisfies the following three
conditions:

\begin{enumerate}
\item[(i)]
$X_0=(X^\sigma )_\xi $.

\item[(ii)] The tangent vector $X_{n+i}$
is $p^{10}$-vertical for $1\leq i\leq n$.

\item[(iii)]
$X_i=\left(
H^\sigma \circ\varepsilon^{-1}
\right)
\left(
X_{n+i}
\right) $
for $1\leq i\leq n$.
\end{enumerate}
\end{definition}

\begin{proposition}
\label{Prop2}
Let $G$ be the image of the Lie-group monomorphism
\[
\begin{array}
[c]{l}
\iota \colon Gl(n,\mathbb{R})
\to Gl(2n+1,\mathbb{R}),\\
\iota (\Lambda )
=\left(
\begin{array}
[c]{ccc}
1 & 0 & 0\\
0 & \Lambda & 0\\
0 & 0 & \Lambda
\end{array}
\right) ,
\quad
\forall \Lambda \in Gl(n,\mathbb{R}).
\end{array}
\]
Let $\pi \colon F(M^1)\to M^1$ be the bundle
of linear frames of the manifold $M^1$.
The bundle $\pi\colon P^\sigma \to M^1$
of all linear frames adapted to a given
SODE $\sigma $ on $M$ is a $G$-structure.
Moreover, if $(U;x^i)$ is a coordinate
open domain in $M$, then the linear frame
\begin{equation}
\begin{array}
[c]{l}
s\colon J^1(\mathbb{R},U)\to  F(M^1),\\
s(\xi )
=\left(
\left(
X^\sigma
\right) _\xi ,
\left(
X_i^\sigma \right) _\xi ,
\left(
\frac{\partial }{\partial \dot{x}^i}
\right) _\xi
\right) ,
\quad
1\leq i\leq n,\;\xi\in J^1(\mathbb{R},U),
\end{array}
\label{reference}
\end{equation}
defines a section of $P^\sigma $,
with dual coframe
$(dt,\omega ^i,\varpi ^i)$,
$1\leq i\leq n$, where
\begin{equation}
\varpi ^i=d\dot{x}^i-F^idt
-\tfrac{1}{2}
\frac{\partial F^i}{\partial \dot{x}^j}
\left(
dx^j-\dot{x}^jdt
\right) .\label{varomega}
\end{equation}

\end{proposition}

\begin{proof}
First of all we prove the last part
of the statement, i.e., the section
$s$ in the formula \eqref{reference}
takes values in $P^\sigma $.
In Definition \ref{definition1},
the items (i) and (ii) are obvious,
and the item (iii) follows directly
from the definitions of $\varepsilon $
and $H^\sigma $ in the formulas
\eqref{epsilon} and \eqref{Hsigma},
respectively. By again taking the items
(i)-(iii) in Definition \ref{definition1}
into account, it follows that every linear
frame $(X_\alpha )_{\alpha =0}^{2n}
\in (P^\sigma )_\xi $ can be written as
$X_0=(X^\sigma )_\xi $,
$X_j=\lambda _j^i(X_i^\sigma )_\xi $,
$X_{n+j}=\lambda _j^i
\left(
\partial /\partial \dot{x}^i
\right) _\xi $,
$\;1\leq j\leq n$, with
$\Lambda =(\lambda _j^i)\in Gl(n,\mathbb{R})$,
or equivalently,
$(X_0,X_1,\dotsc,X_{2n})
=s(\xi )\cdot \iota (\Lambda )$,
$\iota (\Lambda )$ being the matrix
defined in the statement, and the result follows.
\end{proof}

\begin{remark}
Each $G$-structure $P\subset F(M^1)$ determines
a vector field $X_P$ on $M^1$ by $(X_P)_\xi
=X_0\in T_\xi M^1$ for every $\xi \in M^1$,
where $u=(X_0,X_1,\dotsc,X_{2n})$ is an arbitrary
linear frame in $P$ at $\xi $. The definition
makes sense as $X_0$ is kept invariant under
all the elements in $G$. A $G$-structure $P$
is the $G$-structure associated with a SODE
if and only if $X_P$ is a dynamical flow, i.e.,
$X_P(t)=1$ and $i_{X_P}(\mathcal{C}_{M^1})=0$,
where $\mathcal{C}_{M^1}$ is the contact
differential system on $TM^1$, locally spanned
by the $1$-forms $\omega ^i$, $1\leq i\leq n$,
defined in the formula \eqref{J}.
\end{remark}

\begin{proposition}
\label{Prop2.1}
The first prolongation of the Lie algebra
$\mathfrak{g}
=\iota _\ast \mathfrak{gl}(n,\mathbb{R})$
of $G$ in \emph{Proposition \ref{Prop2}},
vanishes; namely, $\mathfrak{g}^{(1)}
=\left(
S^2V^\ast \otimes V
\right)
\cap
\left(
V^\ast \otimes\mathfrak{g}
\right)
=\{ 0\} $,
where $V=\mathbb{R}^{2n+1}$.
\end{proposition}

\begin{proof}
Let $(v_\alpha )_{\alpha =0}^{2n}$ be the standard basis
for $V$ with dual basis $(v^\alpha )_{\alpha =0}^{2n}$.
If $t=t_{\alpha \beta }^\gamma v^\alpha \otimes v^\beta
\otimes v_\gamma \in \mathfrak{g}^{(1)}$, then, once
an index $0\leq \alpha \leq 2n$ is fixed, the endomorphism
$t_{\alpha \beta }^\gamma v^\beta \otimes v_\gamma $
belongs to $\mathfrak{g}$. Taking the defintion
of the subalgebra
$\mathfrak{g}\subset \mathfrak{gl}(2n+1,\mathbb{R})$
in Proposition \ref{Prop2} into account, we obtain
$t_{\alpha \beta }^\gamma =0$, if $\beta =0$
or $\gamma =0$, and
\begin{equation}
\label{c4}
t_{\alpha \beta }^\gamma =0,
\;\text{if }1\leq \beta \leq n,
\; n+1\leq \gamma \leq 2n\text{ or }
n+1\leq \beta \leq 2n,
\; 1\leq \gamma \leq n,
\end{equation}
\begin{equation}
\label{c5}
t_{\alpha \beta }^\gamma
=t_{\alpha ,\beta +n}^{\gamma +n},
\;\beta ,\gamma =1,\dotsc,n.
\end{equation}
For $1\leq \alpha \leq n$ and
$\beta ,\gamma =1,\dotsc,n$,
taking $t_{\alpha \beta }^\gamma
=t_{\beta \alpha }^\gamma $,
\eqref{c5}, and \eqref{c4} into account,
we obtain
\begin{equation}
\label{c1}
t_{\alpha \beta}^\gamma
\overset{\mathrm{\eqref{c5}}}{=}
t_{\alpha ,\beta +n}^{\gamma +n}
=t_{\beta +n,\alpha }^{\gamma +n}
\overset{\mathrm{\eqref{c4}}}{=}0.
\end{equation}
For $n+1\leq \alpha \leq 2n$
and $\beta,\gamma =1,\dotsc,n$,
taking
$t_{\alpha \beta}^\gamma
=t_{\beta \alpha }^\gamma $
and \eqref{c4} into account, we obtain
\begin{equation}
\label{c2}
t_{\alpha \beta }^\gamma
=t_{\beta \alpha }^\gamma
\overset{\mathrm{\eqref{c4}}}{=}0.
\end{equation}
Finally, from \eqref{c5}, \eqref{c1},
and \eqref{c2}, we conclude
\[
t_{\alpha ,\beta +n}^{\gamma +n}
=t_{\alpha \beta}^\gamma =0,
\text{ for }1\leq \alpha \leq 2n
\text{ and }\beta,\gamma =1,\dotsc,n.
\]
\end{proof}

\begin{remark}
The previous proof is avoidable by simply looking
at the table of Lie algebras with non-trivial first
prolongation, e.g., see \cite[Tables 1 \& 2]{L}.
\end{remark}

\begin{proposition}
\label{Prop2.2}The adjoint bundle of the $G$-structure
$\pi \colon P^\sigma \to M^1$ corresponding to a SODE
$\sigma $ on $M$ can be identified to the vector subbundle
\[
\mathrm{ad}P^\sigma
\subseteq T^\ast (M^1)\otimes T(M^1)
\]
of all endomorphisms
$E\colon T(M^1)\to T(M^1)$ such that,
\[
E(T^0(M^1))=\{ 0\} ,
\quad E(T^-(M^1))\subseteq T^-(M^1),
\quad E(T^+(M^1))\subseteq T^+(M^1),
\]
and the composition map
$H^\sigma \circ \varepsilon ^{-1}$
conjugates
$\left. E\right\vert _{T^-(M^1)}$
and
$\left. E\right\vert _{T^+(M^1)}$,
i.e.,
\[
\left.
E\right\vert _{T^-(M^1)}\circ H^\sigma
\circ \varepsilon ^{-1}
=H^\sigma \circ\varepsilon^{-1}\circ
\left.
E\right\vert _{T^+(M^1)},
\]
where $\varepsilon $, $T^0(M^1)$, $T^-(M^1)$,
$T^+(M^1)$, and $H^\sigma $ are given
in the formulas \emph{\eqref{epsilon}},
\emph{\eqref{T^0}}, \emph{\eqref{T^-}},
\emph{\eqref{T^+}}, and \emph{\eqref{Hsigma}},
respectively.
\end{proposition}

\begin{proof}
By its very definition, the adjoint bundle
of $\pi \colon P^\sigma \to M^1$ is the bundle associated
to $P^\sigma $ with respect to the adjoint representation
of $G$ on its Lie algebra $\mathfrak{g}$, as defined
in Propositions \ref{Prop2} and \ref{Prop2.1}, respectively.
The result then follows readily from the conditions
(i)-(iii) in Definition \ref{definition1}.
\end{proof}

\begin{proposition}
\label{Prop3}The Chern connection as defined
in \emph{\cite{Chern}} and \emph{\cite{MassaPagani}}
is reducible (cf.\ \emph{\cite[p. 81]{KN}})
to the $G$-structure $P^\sigma $.
\end{proposition}

\begin{proof}
As is well known (see \cite[Corollary 2.1]{MassaPagani}),
the Chern connection $\nabla ^\sigma $ attached to a SODE
$\sigma $ is locally given in the frame
$\left( X_{\alpha }\right) _{\alpha =0}^{2n}
=\left(
X^\sigma ,X_i^\sigma ,\partial/\partial \dot{x}^j
\right) $,
$i,j=1,\dotsc,n$, by the following
formulas:
{\small
\begin{equation}
\begin{array}
[c]{lll}
\nabla  _{X^\sigma }^\sigma X^\sigma
=0,
& \nabla _{X^\sigma }^\sigma X_i^\sigma
=-\frac{1}{2}
\tfrac{\partial F^j}{\partial \dot{x}^i}
X_j^\sigma ,
& \nabla _{X^\sigma }^\sigma
\tfrac{\partial }{\partial \dot{x}^i}
=-\frac{1}{2}
\tfrac{\partial F^j}{\partial \dot{x}^i}
\tfrac{\partial }{\partial \dot{x}^j},
\medskip \\
\nabla _{X_i^\sigma }^\sigma X^\sigma =0,
& \nabla _{X_j^\sigma }^\sigma X_i^\sigma
=-\frac{1}{2}
\tfrac{\partial^2F^k}
{\partial \dot{x}^i\partial \dot{x}^j}
X_k^\sigma ,
& \nabla _{X_i^\sigma }^\sigma
\tfrac{\partial }{\partial \dot{x}^j}
=-\frac{1}{2}
\tfrac{\partial^2F^k}
{\partial \dot{x}^i\partial \dot{x}^j}
\tfrac{\partial }{\partial \dot{x}^k},
\medskip \\
\nabla _{\frac{\partial }
{\partial \dot{x}^i}}^\sigma X^\sigma
=0, & \nabla _{\frac{\partial }
{\partial \dot{x}^i}}^\sigma X_j^\sigma
=0, & \nabla _{\frac{\partial }
{\partial \dot{x}^i}}^\sigma
\tfrac{\partial }{\partial \dot{x}^j}=0.
\end{array}
\label{Tabla}
\end{equation}
}We also set
$\left( \theta ^\alpha \right) _{\alpha =0}^{2n}
=\left( dt,\omega ^i,\varpi ^i\right) $,
$1\leq i\leq n$, where $\omega ^i$ (resp.\
$\varpi ^i$) is defined in the formula
\eqref{J} (resp.\ \eqref{varomega}). Moreover,
the equations of the subalgebra
$\mathfrak{g}
\subset \mathfrak{gl}(2n+1,\mathbb{R})$
(cf.\ Proposition \ref{Prop2}) are
\begin{align*}
a_0^\alpha
& =0,
\; 0\leq\alpha \leq 2n,\\
a_\alpha ^0
& =0,
\; 1\leq\alpha \leq 2n,\\
a_\beta ^\alpha
& =0,
\; \text{if }1\leq \alpha \leq n,
\; n+1\leq\beta
\leq 2n \text{ or }n+1\leq \alpha \leq 2n,
\; 1\leq \beta \leq n,\\
a_\beta ^\alpha
& =a_{n+\beta }^{n+\alpha },
\; \alpha ,\beta =1,\dotsc,n,
\end{align*}
where
$A=(a_\beta ^\alpha )_{\alpha ,\beta =0}^{2n}
\in \mathfrak{gl}(2n+1,\mathbb{R})$. According
to (\cite[Proposition 5.4]{Fujimoto}) Chern's
connection is reducible to $P^\sigma $ if
and only if for every $X\in \mathfrak{X}(M^1)$,
the following equations hold:
\begin{equation}
\begin{array}
[c]{l}
\theta ^0
\left(
\nabla _X^\sigma X_0
\right)
=dt\left(
\nabla _X^\sigma X^\sigma
\right)
=0,\\
\theta ^i
\left(
\nabla _X^\sigma X_0
\right)
=\omega ^i
\left(
\nabla _X^\sigma X^\sigma
\right)
=0,\\
\theta ^{n+i}
\left(
\nabla _X^\sigma X_0
\right)
=\varpi^i
\left(
\nabla _X^\sigma X^\sigma
\right)
=0,\\
\theta ^0
\left(
\nabla _X^\sigma X_i
\right)
=dt\left(
\nabla _X^\sigma X_i^\sigma
\right)
=0,\\
\theta ^0
\left(
\nabla _X^\sigma X_{n+i}
\right)
=dt\left(
\nabla _X^\sigma
\dfrac{\partial }{\partial \dot{x}^i}
\right)
=0,\\
\theta ^{n+i}
\left(
\nabla _X^\sigma X_j
\right)
=\varpi^i
\left(
\nabla _X^\sigma X_j^\sigma
\right)
=0,\\
\theta ^i
\left(
\nabla _X^\sigma X_{n+j}
\right)
=\omega ^i
\left(
\nabla _X^\sigma
\dfrac{\partial }{\partial \dot{x}^j}
\right)
=0,\\
\theta ^j
\left(
\nabla _X^\sigma X_i
\right)
=\theta ^{n+j}
\left(
\nabla _X^\sigma X_{n+i}
\right)  .
\end{array}
\label{G-structure}
\end{equation}
The equations \eqref{G-structure}
are an easy consequence of the formulas
in Table \eqref{Tabla}. In fact,
the three first conditions are straightforward
taking account of the fact that $X^\sigma $
is parallel. Furthermore, from the equations
in Table \eqref{Tabla} we obtain
$dt(\nabla _X^\sigma X_i^\sigma )=0$,
$\varpi ^i(\nabla _X^\sigma X_j^\sigma )=0$
(resp.\
$dt(\nabla _X^\sigma \partial /\partial \dot{x}^i)
=0$, $\omega ^i
(\nabla _X^\sigma \partial /\partial \dot{x}^j)=0$).

Finally, by setting $X=aX^\sigma +a^iX_i^\sigma
+b^i\partial /\partial \dot{x}^i$, we obtain
\begin{align*}
\theta ^j
\left(
\nabla _X^\sigma X_i
\right)
& =\omega ^j
\left(
\nabla _X^\sigma X_i^\sigma
\right) \\
& =-\tfrac{1}{2}
\left(
a\frac{\partial F^j}{\partial \dot{x}^i}
+a^h
\frac{\partial^2F^j}
{\partial \dot{x}^h\partial \dot{x}^i}
\right)  ,\\
\theta ^{n+j}
\left(
\nabla _X^\sigma X_{n+i}
\right)
& =\varpi^j
\left(
\nabla _X^\sigma
\dfrac{\partial }{\partial \dot{x}^i}
\right) \\
& =-\tfrac{1}{2}
\left(
a\frac{\partial F^j}{\partial \dot{x}^i}
+a^h\frac{\partial^2F^j}
{\partial \dot{x}^h\partial \dot{x}^i}
\right)  ,
\end{align*}
for $i,j=1,\dotsc,n$, and the last equation
in \eqref{G-structure} follows.
\end{proof}

\section{Characterizations of the Chern connection}
\label{section_characterizations}

\subsection{First characterization}

\begin{theorem}
Given a SODE $\sigma$ on $M$, there is a unique
linear connection $\nabla ^\sigma $ on $M^1$
such that,

\begin{enumerate}
\item[\emph{(1)}]
$\nabla ^\sigma X^\sigma =0$.

\item[\emph{(2)}]
$\nabla ^\sigma L_{X^\sigma }J=0$.

\item[\emph{(3)}]
If $E^\sigma \colon T(M^1)\to T(M^1)$ is
$E^\sigma =J+H^\sigma \circ\varepsilon^{-1}\circ
\left(
L_{X^\sigma }J
\right) ^v$,
then $\nabla ^\sigma E^\sigma =0$.

\item[\emph{(4)}]
The torsion of $\nabla ^\sigma $ is the tensor field
$T^\sigma $ given by,
\begin{equation}
\label{torsion}
T^\sigma =K^\sigma +dt\wedge
\left(
H^\sigma \circ \varepsilon ^{-1}\circ
\left(
L_{X^\sigma }J
\right) ^v
\right) ,
\end{equation}
where $K^\sigma $ is the curvature form of the splitting
$H^\sigma $ induced by $\sigma $ as defined in the formula
\emph{\eqref{curvature_form}}.

This connection coincides with that defined in
\emph{\cite{Chern}} and \emph{\cite{MassaPagani}}.
\end{enumerate}
\end{theorem}

\begin{proof}
By expressing the equation $\nabla ^\sigma X^\sigma =0$
in the frame \eqref{reference} we obtain
\begin{equation}
\left(
\nabla ^\sigma
\right) _{X^\sigma }X^\sigma =0,
\quad
\left(
\nabla ^\sigma
\right) _{X_i^\sigma }X^\sigma =0,
\quad
\left(
\nabla ^\sigma
\right) _{\partial /\partial \dot{x}^i}
X^\sigma =0.
\label{cc1}
\end{equation}
From item (2), taking the formula \eqref{LXJ}
for $L_{X^\sigma }J$ and the conditions \eqref{cc1}
into account, we obtain
\begin{align*}
0  &
=\left(
\nabla ^\sigma L_{X^\sigma }J
\right)
(X^\sigma ,X^\sigma )
=\nabla _{X^\sigma }^\sigma
\left(
\left(
L_{X^\sigma }J
\right)
(X^\sigma )
\right)
-\left(
L_{X^\sigma }J
\right)
(\nabla _{X^\sigma }^\sigma X^\sigma ),\\
0  &
=\left(
\nabla ^\sigma L_{X^\sigma }J
\right)
\left(
\frac{\partial }{\partial \dot{x}^i},X^\sigma
\right)
=\nabla _{X^\sigma }^\sigma
\frac{\partial }{\partial \dot{x}^i}
-\left(
L_{X^\sigma }J
\right)
\left(
\nabla _{X^\sigma }^\sigma
\frac{\partial }{\partial \dot{x}^i}
\right) ,\\
0  &
=\left(
\nabla ^\sigma L_{X^\sigma }J
\right)
\left(  X_i^\sigma ,X^\sigma
\right)
=-\nabla _{X^\sigma }^\sigma X_i^\sigma
-\left(
L_{X^\sigma }J
\right)
\left(
\nabla _{X^\sigma }^\sigma X_i^\sigma
\right) ,\\
0  &
=\left(
\nabla ^\sigma L_{X^\sigma }J
\right)
\left(
\frac{\partial }{\partial \dot{x}^i},X_j^\sigma
\right)
=\nabla _{X_j^\sigma }^\sigma
\frac{\partial }{\partial \dot{x}^i}
-\left(
L_{X^\sigma }J
\right)
\left(
\nabla _{X_j^\sigma }^\sigma
\frac{\partial }{\partial \dot{x}^i}
\right)  ,\\
0  &
=\left(
\nabla ^\sigma L_{X^\sigma }J
\right)
\left(
X_i^\sigma ,X_j^\sigma
\right)
=-\nabla _{X_j^\sigma }^\sigma X_i^\sigma
-\left(
L_{X^\sigma }J
\right)
\left(
\nabla _{X_j^\sigma }^\sigma X_i^\sigma
\right) ,\\
0  &
=\left(
\nabla ^\sigma L_{X^\sigma }J
\right)
\left(
X_i^\sigma ,\frac{\partial }{\partial \dot{x}^j}
\right)
=-\nabla _{\partial/\partial \dot{x}^j}^\sigma
X_i^\sigma
-\left(
 L_{X^\sigma }J
\right)
\left(
\nabla _{\partial /\partial \dot{x}^j}^\sigma
X_i^\sigma
\right) ,\\
0  &
=\left(
\nabla ^\sigma L_{X^\sigma }J
\right)
\left(
\frac{\partial }{\partial \dot{x}^i},
\frac{\partial }{\partial \dot{x}^j}
\right)
=\nabla _{\partial /\partial \dot{x}^j}^\sigma
\frac{\partial }{\partial \dot{x}^i}
-\left(
L_{X^\sigma }J
\right)
\left(
\nabla _{\partial/\partial \dot{x}^j}^\sigma
\frac{\partial }{\partial \dot{x}^i}
\right) .
\end{align*}
Recalling that $T^+(M^1)=\ker(L_{X^\sigma }J-I)$
and $T^-(M^1)=\ker(L_{X^\sigma }J+I)$,
from the previous formulas we deduce
$\nabla _{X^\sigma }^\sigma
\partial /\partial \dot{x}^i$,
$\nabla _{X_j^\sigma }^\sigma
\partial/\partial \dot{x}^i$,
$\nabla _{\partial/\partial \dot{x}^j}^\sigma
\partial /\partial \dot{x}^i$
(resp.\ $\nabla _{X^\sigma }^\sigma X_i^\sigma $,
$\nabla _{X_j^\sigma }^\sigma X_i^\sigma $,
$\nabla _{\partial/\partial \dot{x}^j}^\sigma X_i^\sigma $)
belong to $T^+(M^1)$ (resp.\ $T^-(M^1)$)
and therefore
{\small
\begin{equation}
\begin{array}
[c]{ccc}
\nabla _{X^\sigma }^\sigma
\partial /\partial \dot{x}^i
=A_i^k
\frac{\partial }{\partial \dot{x}^k},
& \nabla _{X_j^\sigma }^{\sigma }
\partial / \partial \dot{x}^i
=A_{ij}^k\partial/\partial \dot{x}^k, &
\nabla _{\partial/\partial \dot{x}^j}^\sigma
\partial /\partial \dot{x}^i
=B_{ij}^k\partial /\partial \dot{x}^k,\\
\nabla _{X^\sigma }^\sigma X_i^\sigma
=C_i^kX_k^\sigma ,
& \nabla _{X_j^\sigma }^\sigma X_i^\sigma
=C_{ij}^kX_k^\sigma ,
& \nabla _{\partial /\partial \dot{x}^j}^\sigma X_i^\sigma
=D_{ij}^kX_k^\sigma .
\end{array}
\label{c3}
\end{equation}
}
Furthermore, according to the definition of $E^\sigma $
in the statement, its local expression is seen to be
\begin{equation}
\label{E}
E^\sigma =\omega ^i\otimes
\frac{\partial }{\partial \dot{x}^i}
+\varpi ^i\otimes X_i^\sigma ,
\end{equation}
where $\omega ^i$, $\varpi^i$, and $X_i^\sigma $
are given by the formulas \eqref{J}, \eqref{varomega},
and \eqref{T^-}, respectively. Hence, from the item
(3) and taking the formulas \eqref{c3} and \eqref{E}
into account, we obtain
\begin{align*}
0  &
=\left(
\nabla ^\sigma E^\sigma
\right)
(X^\sigma ,X^\sigma )
=\nabla _{X^\sigma }^\sigma
\left(
E^\sigma (X^\sigma )
\right)
-E^\sigma (\nabla _{X^\sigma }^\sigma X^\sigma ),\\
0 &
=\left(
\nabla ^\sigma E^\sigma
\right)
\left(
\frac{\partial }{\partial \dot{x}^i},X^\sigma
\right)
=\left(
C_i^k-A_i^k
\right)
X_k^\sigma ,\\
0  &
=\left(
\nabla ^\sigma E^\sigma
\right)
\left(
X_i^\sigma ,X^\sigma
\right)
=\left(
A_i^k-C_i^k
\right)
\frac{\partial }{\partial \dot{x}^k},\\
0  &
=\left(
\nabla ^\sigma E^\sigma
\right)
\left(
\frac{\partial }{\partial \dot{x}^i},X_j^\sigma
\right)
=\left(
C_{ij}^k-A_{ij}^k
\right)
X_k^\sigma ,\\
0  &
=\left(
\nabla ^\sigma E^\sigma
\right)
\left(  X_i^\sigma ,X_j^\sigma
\right)
=\left(
A_{ij}^k-C_{ij}^k
\right)
\frac{\partial }{\partial \dot{x}^k},\\
0  &
=\left(
\nabla ^\sigma E^\sigma
\right)
\left(
X_i^\sigma ,\frac{\partial }{\partial \dot{x}^j}
\right)
=\left(
B_{ij}^k-D_{ij}^k
\right)
\frac{\partial }{\partial \dot{x}^k},\\
0  &
=\left(
\nabla ^\sigma E^\sigma
\right)
\left(
\frac{\partial }{\partial \dot{x}^i},
\frac{\partial }{\partial \dot{x}^j}
\right)
=\left(
D_{ij}^k-B_{ij}^k
\right)
X_k^\sigma .
\end{align*}
Hence $C_i^k=A_i^k$,
$C_{ij}^k=A_{ij}^k$,
$D_{ij}^k=B_{ij}^k$.
Therefore,
\begin{equation}
\begin{array}
[c]{ccc}
\nabla _{X^\sigma }^\sigma
\frac{\partial }{\partial \dot{x}^i}
=A_i^k\frac{\partial }{\partial \dot{x}^k},
& \nabla _{X_j^\sigma }^\sigma
\frac{\partial }{\partial \dot{x}^i}
=A_{ij}^k
\frac{\partial }{\partial \dot{x}^k},
& \nabla _{\partial/\partial \dot{x}^j}^\sigma
\frac{\partial }{\partial \dot{x}^i}
=B_{ij}^k
\frac{\partial }{\partial \dot{x}^k},\\
\nabla _{X^\sigma }^\sigma X_i^\sigma
=A_i^kX_k^\sigma ,
& \nabla _{X_j^\sigma }^\sigma X_i^\sigma
=A_{ij}^kX_k^\sigma ,
& \nabla _{\partial /\partial \dot{x}^j}^\sigma
X_i^\sigma
=B_{ij}^kX_k^\sigma .
\end{array}
\label{cc4}
\end{equation}
Let us finally impose the condition (4).
Taking the expression for $K^\sigma $
in \eqref{K_sigma} and the definitions
of $H^\sigma $, $\varepsilon^{-1}$
and $\left( L_{X^\sigma }J\right) ^v$
in \eqref{Hsigma}, \eqref{epsilon}
and \eqref{LXJ}, \eqref{vH} respectively,
the tensor field $T^\sigma $ defined
in \eqref{torsion} is written as:
\begin{equation}
T^\sigma =-P_j^idt\wedge \omega ^j\otimes
\frac{\partial }{\partial\dot {x}^i}
-\sum _{i<j}T_{ij}^k\omega ^i\wedge \omega ^j
\otimes\frac{\partial }{\partial \dot{x}^k}
+dt\wedge \varpi ^i\otimes X_i^\sigma .
\label{Torsion}
\end{equation}
Taking the equations \eqref{Torsion}
and \eqref{cc1} into account we have
\begin{align*}
X_i^\sigma
& =T^\sigma
\left(
X^\sigma ,
\frac{\partial }{\partial \dot{x}^i}
\right) \\
& =\nabla _{X^\sigma }^\sigma
\frac{\partial }{\partial \dot{x}^i}
-\left(
-X_i^\sigma
-\tfrac{1}{2}
\frac{\partial F^j}{\partial \dot{x}^i}
\frac{\partial }{\partial \dot{x}^j}
\right)  .
\end{align*}
Hence
\begin{equation}
\label{cc2}
\nabla _{X^\sigma }^\sigma
\frac{\partial }{\partial \dot{x}^i}
=-\tfrac{1}{2}
\frac{\partial F^j}{\partial \dot{x}^i}
\frac{\partial }{\partial \dot{x}^j}.
\end{equation}
Finally, taking \eqref{cc4} into account,
from $T^\sigma
\left(
X_i^\sigma ,\partial /\partial \dot{x}^j
\right)
=0$ we have
\begin{align*}
0  & =T^\sigma
\left(
X_i^\sigma ,
\frac{\partial }{\partial \dot{x}^j}
\right) \\
& =\nabla _{X_i^\sigma }^\sigma
\frac{\partial }{\partial \dot{x}^j}
-\nabla _{\partial /\partial \dot{x}^j}^\sigma
X_i^\sigma
+\tfrac{1}{2}
\frac{\partial^2F^k}
{\partial \dot{x}^j\partial \dot{x}^i}
\frac{\partial }{\partial \dot{x}^k}\\
& =\left(
A_{ji}^k
+\tfrac{1}{2}
\frac{\partial^2F^k}
{\partial \dot{x}^j\partial \dot{x}^i}
\right)
\frac{\partial }{\partial \dot{x}^k}
-B_{ij}^kX_k^\sigma ,
\end{align*}
and therefore
\begin{equation}
A_{ji}^k
=-\tfrac{1}{2}
\frac{\partial ^2F^k}{\partial \dot{x}^j
\partial \dot{x}^i},
\quad B_{ij}^k=0.
\label{c6}
\end{equation}
From \eqref{cc2}, \eqref{cc4}, and \eqref{c6}
we thus obtain
\begin{align*}
\nabla _{X^\sigma }^\sigma
\frac{\partial }{\partial \dot{x}^i}
& =-\tfrac{1}{2}
\frac{\partial F^j}{\partial \dot{x}^i}
\frac{\partial }
{\partial \dot{x}^k},
\quad
\nabla _{X_j^\sigma }^\sigma
\frac{\partial }{\partial \dot{x}^i}
=-\tfrac{1}{2}
\frac{\partial^2F^k}
{\partial \dot{x}^j\partial \dot{x}^i}
\frac{\partial }{\partial \dot{x}^k},
\quad
\nabla _{\partial /\partial \dot{x}^j}^\sigma
\frac{\partial }{\partial \dot{x}^i}=0,\\
\nabla _{X^\sigma }^\sigma X_i^\sigma
& =-\tfrac{1}{2}
\frac{\partial F^j}
{\partial \dot{x}^i}X_k^\sigma ,
\quad
\nabla _{X_j^\sigma }^\sigma X_i^\sigma
=-\tfrac{1}{2}\frac{\partial ^2F^k}
{\partial \dot{x}^j\partial \dot{x}^i}
X_k^\sigma ,
\quad
\nabla _{\partial /\partial \dot{x}^j}^\sigma
X_i^\sigma =0.
\end{align*}
These formulas together with \eqref{cc1}
are exactly the same as those in \eqref{Tabla}.
\end{proof}

\begin{remark}
From the characterization of the adjoint bundle
of the $G$-structure $\pi \colon P^\sigma \to  M^1$
given in Proposition \ref{Prop2.2} it follows that
the first structure tensor of $P^\sigma $, i.e.,
$T^\sigma \operatorname{mod}\mathrm{alt}
\left(
T^\ast M^1\otimes \mathrm{ad}P^\sigma
\right) $
in $\bigwedge ^2T^\ast M^1\otimes TM^1/\mathrm{alt}
\left(
T^\ast M^1\otimes\mathrm{ad}P^\sigma
\right) $
(e.g.\ see \cite{Guillemin}, \cite{LL}) never vanishes
and that $P^\sigma $ is not $1$-integrable.
\end{remark}

\begin{remark}
The item (4) in the theorem can be replaced
by the following weaker conditions:
$\mathrm{Tor}\nabla ^\sigma |_{T^-(M^1)\times T^+(M^1)}=0$,
$i_{X^\sigma }\mathrm{Tor}\nabla ^\sigma |_{T^+(M^1)}
=H^\sigma \circ \varepsilon ^{-1}$.
\end{remark}

\subsection{Second characterization}

\begin{theorem}
\label{thm1}
The Chern connection $\nabla ^\sigma $ attached
to a SODE $\sigma $ on $M$ is the only linear connection
on $M^1$ reducible to the $G$-structure
$\pi \colon P^\sigma \to  M^1$ introduced in
\emph{Proposition \ref{Prop2}} whose torsion tensor field
is given in the formula \emph{\eqref{torsion}}.
\end{theorem}

\begin{proof}
According to Proposition \ref{Prop3} the Chern connection
is reducible to $P^\sigma $. Furthermore, from the formulas
in \eqref{Tabla} the torsion of $\nabla ^\sigma $ is readily
computed, namely,
\[
\begin{array}
[c]{ll}
\mathrm{Tor}\nabla ^\sigma
\left(  X^\sigma ,X_j^\sigma
\right)
=-P_j^i\frac{\partial }{\partial \dot{x}^i},
& \mathrm{Tor}\nabla ^\sigma
\left(
X^\sigma ,\frac{\partial }{\partial \dot{x}^i}
\right)
=X_i^\sigma ,\smallskip \\
\mathrm{Tor}\nabla ^\sigma
\left(
X_i^\sigma ,X_j^\sigma
\right)
=-T_{ij}^k\frac{\partial }{\partial \dot{x}^k},
& \mathrm{Tor}\nabla ^\sigma
\left(
X_i^\sigma ,\frac{\partial }{\partial \dot{x}^j}
\right)
=0,\smallskip \\
\mathrm{Tor}\nabla ^\sigma
\left(
\frac{\partial }{\partial \dot{x}^i},
\frac{\partial }{\partial \dot{x}^j}
\right) =0,
&
\end{array}
\]
where the functions $P_j^i$, $T_{ij}^k$
are defined in the formulas \eqref{P's},
\eqref{T's}, respectively. Hence,
from the above equations and the formula
\eqref{Torsion}, one obtains
$\mathrm{Tor}\nabla ^\sigma =T^\sigma $.

If $\nabla $ is the covariant derivative
of another linear connection on $M^1$
reducible to $P^\sigma $, then
(e.g., see \cite[Proposition I.1]{LL})
a unique section $h$ of the vector subbundle
$T^\ast (M^1)\otimes \mathrm{ad}P^\sigma
\subseteq T^\ast (M^1)\otimes T^\ast (M^1)
\otimes T(M^1)$ exists such that
$\nabla _XY=\nabla _X^\sigma Y+h(X,Y)$,
$\forall X,Y\in \mathfrak{X}(M^1)$.
If the torsion tensor field of $\nabla $
coincides with that of $\nabla ^\sigma $,
then $h$ takes values in
\[
(S^2T^\ast M^1\otimes TM^1)
\cap (T^\ast M^1\otimes \mathrm{ad}P^\sigma ).
\]
But this vector bundle vanishes by virtue
of Proposition \ref{Prop2.1}.
\end{proof}

\section{Curvature of $\nabla ^\sigma $ and characteristic classes}
\label{section_curvature}

\begin{lemma}
\label{lemma}Let $R^\sigma $ be the curvature
tensor of the Chern connection $\nabla ^\sigma $
of a SODE $\sigma$ on $M$.

\begin{enumerate}
\item[\emph{(i)}]
The condition
$R^\sigma \left( X^\sigma ,Y\right) Z=0$,
$\forall Y,Z\in T^-(M^1)$ is equivalent
to the condition
$R^\sigma \left( X^\sigma ,Y\right) Z=0$,
$\forall Y\in T^-(M^1)$, $\forall Z\in T^+(M^1)$.

\item[\emph{(ii)}]
The condition $R^\sigma (X,Y)Z=0$,
$\forall X,Y,Z\in T^-(M^1)$
is equivalent to the condition $R^\sigma (X,Y)Z=0$,
$\forall X,Y\in T^-(M^1)$, $\forall Z\in T^+(M^1)$.

\item[\emph{(iii)}]
The condition $R^\sigma (X,Y)Z=0$,
$\forall X,Z\in T^-(M^1)$,
$\forall Y\in T^+(M^1)$ is equivalent
to the condition
$R^\sigma (X,Y)Z=0$, $\forall X\in T^-(M^1)$,
$\forall Y,Z\in T^+(M^1)$.
\end{enumerate}
\end{lemma}

\begin{proof}
The result immediately follows from
the following formulas:
\begin{equation}
\begin{array}
[c]{ll}
R^\sigma
\left(
X^\sigma ,X_j^\sigma
\right)
X_i^\sigma =A_{ij}^hX_h^\sigma ,
& R^\sigma
\left(  X^\sigma ,X_j^\sigma
\right)
\dfrac{\partial }{\partial \dot{x}^i}
=A_{ij}^h
\dfrac{\partial }{\partial \dot{x}^h},
\smallskip\\
R^\sigma
\left(
X_i^\sigma ,X_j^\sigma
\right)
X_k^\sigma =B_{ijk}^hX_h^\sigma ,
& R^\sigma
\left(
X_i^\sigma ,X_j^\sigma
\right)
\dfrac{\partial }{\partial \dot{x}^k}
=B_{ijk}^h
\dfrac{\partial }{\partial \dot{x}^h}
\; (i<j), \smallskip\\
R^\sigma
\left(
X_i^\sigma ,\dfrac{\partial }{\partial \dot{x}^j}
\right)
X_k^\sigma =R_{ijk}^hX_h^\sigma ,
& R^\sigma
\left(
X_i^\sigma ,\dfrac{\partial }{\partial \dot{x}^j}
\right)
\dfrac{\partial }{\partial \dot{x}^k}
=R_{ijk}^h\dfrac{\partial }{\partial \dot{x}^h},
\end{array}
\label{R_sigma}
\end{equation}
where
\begin{align}
2A_{kj}^h & =T_{jk}^h
-\frac{\partial P_k^h}{\partial \dot{x}^j}
-\frac{\partial P_j^h}{\partial \dot{x}^k},
\label{A's}\\
B_{ijk}^h  &
=-\frac{\partial T_{ij}^h}{\partial \dot{x}^k}\;(i<j),
\label{B's}\\
R_{ijk}^h & =\tfrac{1}{2}
\dfrac{\partial ^3F^h}
{\partial \dot{x}^i\partial \dot{x}^j\partial \dot{x}^k},
\label{R's}
\end{align}
the functions $T_{jk}^h$, $P_k^h$ being respectively
defined in \eqref{T's}, \eqref{P's}, and the rest
of components of the curvature tensor vanishes.
\end{proof}

\begin{remark}
The items (i), (ii), and (iii) above are a simple
consequence of the following formula:
\[
\left.
R^\sigma
\left(
X,Y
\right)
\right\vert _{T^-(M^1)}
=H^\sigma \circ \varepsilon ^{-1}\circ
\left.
R^\sigma
\left(
X,Y
\right)
\right\vert _{T^+(M^1)},
\;\forall X,Y\in T(M^1),
\]
which is also deduced from \eqref{R_sigma}.
\end{remark}

\begin{theorem}
\label{CurvatureTheorem}
Let $R^\sigma $ be the curvature tensor of the Chern
connection $\nabla ^\sigma $ of a SODE $\sigma$ on $M$.

\begin{enumerate}
\item[\emph{(a)}]
If one of the equivalent conditions in items \emph{(ii)}
and \emph{(iii) of Lemma \ref{lemma}} hold on a neighbourhood
of a point $x\in M$, then there exists a fibred coordinate
system $(t,x^{\prime i})$ centred at $x$ on $M^0$ such that
$\ddot{x}^{\prime i}\circ \sigma$ is polynomial of first
degree in $\dot{x}^{\prime 1},\dotsc,\dot{x}^{\prime n}$
for every $1\leq i\leq n$.

\item[\emph{(b)}]
If one of the equivalent conditions in items \emph{(i)}
and \emph{(iii) of Lemma \ref{lemma}} hold on a neighbourhood
of a point $x\in M$, then there exists a fibred coordinate
system $(t,x^{\prime i})$ centred at $x$ on $M^0$ such that
$\ddot{x}^{\prime i}\circ \sigma \in C^\infty (M^0)$
for every $1\leq i\leq n$ \emph{(cf.\ \cite[p.\ 621]{Cartan},
\cite[Theorem 7]{CrampinMartinezSarlet})}.

\item[\emph{(c)}] If one of the equivalent conditions in item
\emph{(iii) of Lemma \ref{lemma}} holds and $K^\sigma $ vanishes
on a neighbourhood of a point $x\in M$, then there exists
a fibred coordinate system $(t,x^{\prime i})$ centred at $x$
on $M^0$ such that $\ddot{x}^{\prime i}\circ \sigma =0$ for
every $1\leq i\leq n$ \emph{(cf.\ \cite[p. 621]{Cartan},
\cite[Theorem 6]{CrampinMartinezSarlet})}.
\end{enumerate}
\end{theorem}

\begin{proof}
According to the formulas \eqref{R_sigma},
the hypothesis in item (a) means $B_{ijk}^h=0$
and $R_{ijk}^h=0$. Taking the equations
\eqref{R's} into account, we obtain
\begin{equation}
\label{quadratic}
F^h=F_{ij}^h\dot{x}^i\dot{x}^j
+F_i^h\dot{x}^i+F_0^h,
\quad
F_{ij}^h=F_{ji}^h,
\;F_0^h,F_i^h,F_{ij}^h\in C^\infty (M^0).
\end{equation}
Substituting \eqref{quadratic} into the expression
for $B_{ijr}^k$ in \eqref{B's} and taking the formula
\eqref{T's} into account, we have
\begin{equation}
\label{Bkijr}
B_{ijr}^k=\frac{\partial F_{jr}^k}{\partial x^i}
-\frac{\partial F_{ir}^k}{\partial x^j}
+F_{ir}^hF_{hj}^k-F_{jr}^hF_{hi}^k.
\end{equation}
For every fixed value $t\in \mathbb{R}$, a symmetric linear
connection $\nabla ^t$ on $M$ can be defined by imposing
that its Christoffel symbols in the coordinate system
$(x^i)_{i=1}^{n}$ are $\Gamma _{ij}^h=-F_{ij}^h$,
and from the formulas \eqref{Bkijr} we conclude
that all the components of the curvature tensor
of $\nabla ^t$ vanish. Hence there exists a coordinate
system $(x^{\prime i})_{i=1}^n$ depending smoothly on $t$,
such that $F_{ij}^{\prime h}=0$ and the result follows.

Similarly, the hypothesis in item (b) means $A_{ij}^h=0$
and $R_{ijk}^h=0$. Hence the formulas \eqref{quadratic}
for $F^1,\dotsc,F^{n}$ also hold in this case and substituting
them into \eqref{A's} recalling the expressions \eqref{T's}
and \eqref{P's}, we obtain
\begin{equation*}
2A_{kj}^h=2B_{jka}^h\dot{x}^a
-2\frac{\partial F_{jk}^h}{\partial t}
+\frac{\partial F_k^h}{\partial x^j}
+F_r^hF_{jk}^r-F_k^rF_{rj}^h.
\end{equation*}

Hence the equations \eqref{Bkijr} again hold and,
in addition, we have
\begin{equation}
\label{A2}
0=-2\frac{\partial F_{jk}^h}{\partial t}
+\frac{\partial F_j^h}{\partial x^k}
+F_r^hF_{jk}^r-F_k^rF_{rj}^h.
\end{equation}

By virtue of (a) and making a change of coordinates,
we can further assume $F_{ij}^h=0$; hence
$F^h=F_i^h\dot{x}^i+F_0^h$, and the equations \eqref{A2}
simply mean that the functions $F_j^h$ are independent
of the coordinates $x^1,\dotsc,x^n$, i.e., they depend
on $t$ only. If we look for a fibred coordinate system
$(x^{\prime i})_{i=1}^n$ centred at $x$ on $M^0$ such that
$\ddot{x}^{\prime i}\circ\sigma\in C^\infty (M^0)$,
for $1\leq i\leq n$, then we obtain
\[
x_{x^ax^b}^{\prime i}=0,
\quad
2x_{tx^b}^{\prime i}+F_b^jx_{x^j}^{\prime i}=0.
\]
The first group of equations above is equivalent to saying
$x^{\prime i}=u_j^i(t)\dot{x}^j+u_0^i(t)$, and from the second
group we obtain $\dot{u}_b^i=-\frac{1}{2}F_b^ju_j^i$,
which is a system of ordinary differential equations
on the unknown functions $u_j^i$, thus proving (b).

Finally, the tensors $K^\sigma $ vanishes if, and only if,
$P_j^i=0$ and $T_{ij}^k=0$, as follows from the expression
of $K^\sigma $ in \eqref{K_sigma}. From $R_{ijk}^h=0$
we again deduce the equations \eqref{quadratic}
and substituting them into the formulas \eqref{T's},
\eqref{P's}, and letting $T_{ij}^k=0$, $P_j^i=0$, we obtain
\[
0=\left(
\!
\frac{\partial F_{jr}^k}{\partial x^i}
-\frac{\partial F_{ir}^k}{\partial x^j}
+F_{ir}^hF_{hj}^k-F_{jr}^hF_{hi}^k
\! \right)
\dot{x}^r
+\tfrac{1}{2}
\left(
\!
\frac{\partial F_j^k}{\partial x^i}
-\frac{\partial F_i^k}{\partial x^j}
+F_i^hF_{hj}^k-F_j^hF_{hi}^k
\! \right) \! ,
\]
\begin{align*}
0 &
=\left(
\frac{\partial F_{ja}^i}{\partial x^b}
-\frac{\partial F_{ab}^i}{\partial x^j}
-F_{ja}^kF_{kb}^i+F_{ab}^rF_{jr}^i
\right)
\dot{x}^a\dot{x}^b\\
& +\left(
\frac{\partial F_{aj}^i}{\partial t}
+\tfrac{1}{2}\frac{\partial F_j^i}
{\partial x^a}
-\frac{\partial F_a^i}{\partial x^j}
-\tfrac{1}{2}F_k^iF_{aj}^k
-\tfrac{1}{2}F_j^kF_{ak}^i
+F_a^hF_{hj}^i
\right)
\dot{x}^a\\
& +\tfrac{1}{2}
\frac{\partial F_j^i}{\partial t}
-\frac{\partial F_0^i}{\partial x^j}
-\tfrac{1}{4}F_j^kF_k^i+F_0^hF_{hj}^i.
\end{align*}
Hence,
\begin{align}
0 & =\frac{\partial F_{jr}^k}{\partial x^i}
-\frac{\partial F_{ir}^k}{\partial x^j}
+F_{ir}^hF_{hj}^k-F_{jr}^hF_{hi}^k,
\label{e_1}\\
0 & =\frac{\partial F_{aj}^i}{\partial t}
+\tfrac{1}{2}\frac{\partial F_j^i}{\partial x^a}
-\frac{\partial F_a^i}{\partial x^j}
-\tfrac{1}{2}F_k^iF_{aj}^k
-\tfrac{1}{2}F_j^kF_{ak}^i+F_a^hF_{hj}^i,
\label{e_2}\\
0 & =\tfrac{1}{2}
\frac{\partial F_j^i}{\partial t}
-\frac{\partial F_0^i}{\partial x^j}
-\tfrac{1}{4}F_j^kF_k^i+F_0^rF_{jr}^i,
\label{e_3}\\
0 & =\frac{\partial F_j^k}{\partial x^i}
-\frac{\partial F_i^k}{\partial x^j}
+F_i^hF_{hj}^k-F_j^hF_{hi}^k.\label{e_4}
\end{align}
Letting $x^0=t$, an auxiliary symmetric linear
connection $\nabla $ can be defined on $M^0$
by giving its Christoffel symbols as follows:
\[
\Gamma _{00}^0
=\Gamma _{0i}^0
=\Gamma _{ij}^0=0,
\;\Gamma _{00}^h=-F_0^h,
\;\Gamma _{i0}^h=-\tfrac{1}{2}F_i^h,
\;\Gamma _{ij}^h=-F_{ij}^h,
\]
and, as a computation shows, we obtain
\[
R_{0kl}^0=0,
\;R_{jkl}^0=0,
\;R_{j00}^0=0,
\;R_{j0l}^0=0,
\]
\begin{align*}
R_{0kl}^i
& =-\tfrac{1}{2}
\frac{\partial F_{l}^i}{\partial x^k}
+\tfrac{1}{2}
\frac{\partial F_k^i}{\partial x^l}
+\tfrac{1}{2}F_l^hF_{hk}^i
-\tfrac{1}{2}F_k^hF_{hl}^i
\underset{\text{\eqref{e_4}}}{=}0,\\
R_{jkl}^i
& =\frac{\partial F_{kj}^i}{\partial x^l}
-\frac{\partial F_{lj}^i}{\partial x^k}
+F_{lj}^hF_{kh}^i-F_{kj}^hF_{lh}^i
\underset{\text{\eqref{e_1}}}{=}0,\\
R_{j00}^i
& =-\tfrac{1}{2}
\frac{\partial F_j^i}{\partial t}
+\frac{\partial F_0^i}{\partial x^l}
+\tfrac{1}{4}F_j^hF_h^i-F_0^hF_{hl}^i
\underset{\text{\eqref{e_3}}}{=}0,\\
R_{j0l}^i
& =-\frac{\partial F_{lj}^i}{\partial t}
+\tfrac{1}{2}
\frac{\partial F_j^i}{\partial x^{l}}
+\tfrac{1}{2}F_{lj}^hF_h^i
-\tfrac{1}{2}F_j^hF_{hl}^i
\underset{\text{\eqref{e_4}--\eqref{e_2}}}{=}0.
\end{align*}
Hence $\nabla $ is flat. Consequently, there exists
a coordinate system
$(x^{\prime 0},x^{\prime 1},\dotsc,x^{\prime n})$
on $M^0$ parallelizing $\nabla $. Moreover, we have
$\nabla (dt)=0$ (which is equivalent to saying
$\Gamma _{00}^0=\Gamma _{0i}^0=\Gamma _{ij}^0=0$)
and hence for all $0\leq i\leq n$,
$X\in \mathfrak{X}(M^0)$,
\[
0=\nabla _X(dt)
\left(
\frac{\partial }{\partial x^{\prime i}}
\right)
=X\left(
dt
\left(
\frac{\partial }{\partial x^{\prime i}}
\right)
\right)
-dt
\left(
\nabla _X\frac{\partial }{\partial x^{\prime i}}
\right) ,
\]
thus proving that the function
$\partial t/\partial x^{\prime i}$ is a constant;
accordingly we can further assume $x^{\prime 0}=t$,
which ends the proof of (c).
\end{proof}
\begin{remark}
\label{remark_connection}
There is a natural and bijective correspondence
between homogeneous quadratic SODE
$\sigma $ (i.e., $(\partial /\partial t)^v=0$
in Table \ref{vH}) independent of $t$ and
symmetric linear connections on $M$.
Actually, given a symmetric linear connection
$\tilde{\nabla }$ on $M$ we can define a section
$\sigma _{\tilde{\nabla }}\colon M^1\to M^2$
as follows: If $\xi =j^1_{t_0}\gamma $,
then there exists a unique geodesic $\tilde{\gamma }$
for $\tilde{\nabla }$ such that,
i) $\tilde{\gamma }(t_0)=\gamma (t_0)$,
ii) $\tilde{\gamma }_\ast (d/dt|_{t_0})
=\gamma _\ast (d/dt|_{t_0})$. Then
$\sigma _{\tilde{\nabla }}(\xi )
=j^2_{t_0}\tilde{\gamma }$. On a coordinate system
$(x^i)$ from the equations of geodesics we deduce
the equations for $\sigma _{\tilde{\nabla }}$, namely,
\[
\ddot{x}^h\circ \sigma _{\tilde{\nabla }}
=-\tilde{\Gamma }_{ij}^h\dot{x}^i\dot{x}^j,
\]
where $\tilde{\Gamma }_{ij}^h$ are the Christoffel
symbols of $\tilde{\nabla }$. Conversely, if
$\ddot{x}^h\circ \sigma
=F_{ij}^h\dot{x}^i\dot{x}^j$
and $(x^{\prime i})$ is another coordinate
system overlapping $(x^i)$, then
\[
F_{jk}^{\prime h}
\frac{\partial x^{\prime j}}
{\partial x^a}
\frac{\partial x^{\prime k}}
{\partial x^b}
=\frac{\partial ^2x^{\prime h}}
{\partial x^a\partial x^b}
+\frac{\partial x^{\prime h}}
{\partial x^i}F_{ab}^i,
\]
and this equation is readily seen
to be equivalent to the transformation
rule of Christoffel's symbols;
e.g., see \cite[III, Proposition 7.2]{KN}.
In this case, $\nabla ^\sigma $
is completely determined by
$\tilde{\nabla }_{\partial /\partial x^l}
\partial /\partial x^j
=-F^h_{ij}\partial /\partial x^h$,
by means of the following formulas:
\[
\begin{array}
[c]{lll}
(\Gamma ^\sigma)_{tt}^t=0,
& (\Gamma ^\sigma )_{x^it}^t=0,
& (\Gamma ^\sigma )_{\dot{x}^it}^t=0,\\
(\Gamma ^\sigma )_{tt}^{x^j}=0,
& (\Gamma ^\sigma )_{x^ht}^{x^j}
=\dot{x}^kF_{hk}^j,
& (\Gamma ^\sigma )_{\dot{x}^kt}^{x^j}
=-\delta_k^j,\\
(\Gamma ^\sigma )_{tt}^{\dot{x}^k}=0,
& (\Gamma ^\sigma )_{x^it}^{\dot{x}^k}
=F_{hr}^kF_{is}^h\dot{x}^s\dot{x}^r,
& (\Gamma ^\sigma )_{\dot{x}^it}^{\dot{x}^k}
=-F_{ri}^k\dot{x}^r,\\
(\Gamma ^\sigma )_{tx^i}^t
=0,
& (\Gamma ^\sigma )_{x^jx^i}^{t}
=0,
& (\Gamma ^\sigma )_{\dot{x}^jx^i}^t
=0,\\
(\Gamma ^\sigma )_{tx^i}^{x^j}=0,
& (\Gamma ^\sigma )_{x^kx^i}^{x^j}
=-F_{ik}^j,
& (\Gamma ^\sigma )_{\dot{x}^kx^i}^{x^j}
=0,\\
(\Gamma ^\sigma )_{tx^i}^{\dot{x}^j
}=0,
& (\Gamma ^\sigma )_{x^kx^i}^{\dot{x}^j}
=-\left(
\tfrac{\partial F_{ir}^j}{\partial x^k}
+F_{ik}^hF_{hr}^j-F_{kh}^jF_{ir}^h
\right)
\dot{x}^r,
& (\Gamma ^\sigma )_{\dot{x}^kx^i}^{\dot{x}^j}
=-F_{ik}^j,\\
(\Gamma ^\sigma )_{t\dot{x}^i}^t=0,
& (\Gamma ^\sigma )_{x^j\dot{x}^i}^t
=0,
& (\Gamma ^\sigma )_{\dot{x}^j\dot{x}^i}^t
=0,\\
(\Gamma ^\sigma )_{t\dot{x}^i}^{x^k}
=0,
& (\Gamma ^\sigma )_{x^j\dot{x}^i}^{x^k}
=0,
& (\Gamma ^\sigma )_{\dot{x}^j\dot{x}^i}^{x^k}
=0,\\
(\Gamma ^\sigma )_{t\dot{x}^i}^{\dot{x}^j}
=0,
& (\Gamma ^\sigma )_{x^k\dot{x}^i}^{\dot{x}^j}
=-F_{ik}^j,
& (\Gamma ^\sigma )_{\dot{x}^k\dot{x}^i}^{\dot{x}^j}
=0,
\end{array}
\]
as follows from the formulas in \eqref{Tabla}
for this particular case.  By using these formulas
and the identification
$M^1\cong \mathbb{R}\times TM$, $j^1_{t_0}\gamma
\mapsto (t_0,\gamma _\ast (d/dt)_{t_0})$,
from \cite[II, formulas (7.8)]{YanoIshihara},
one realizes that the component of $\nabla ^\sigma $
in the tangent bundle coincides with the horizontal
lift $\tilde{\nabla }^H$ of $\tilde{\nabla }$.

Finally, the components of the torsion and curvature
tensor fields of $\nabla ^\sigma $ are expressed
in terms of the components of the curvature tensor
field of the connection $\tilde{\nabla }$ by means
of the following formulas:
$T_{ij}^k=\tilde{R}_{rji}^k\dot{x}^r$,
$P_j^h=\tilde{R}_{sjr}^h\dot{x}^r\dot{x}^s$,
$A_{kj}^h=\tilde{R}_{krj}^h\dot{x}^r$,
$B_{ijk}^h=\tilde{R}_{kij}^h$.
\end{remark}

Let $V$ be $\mathbb{R}^{2n+1}$ with basis
$(v_i)_{i=1}^{2n+1}$ and dual basis $(v^i)_{i=1}^{2n+1}$.
In \cite[3.4]{AM} a generalized Chern-Weil homomorphism
$\left(
S(\mathfrak{g}^\ast )\otimes \bigwedge V^\ast
\right) ^G
\to \Omega (M)$ has been defined for every $G$-structure
on $M$. In our case, $G\cong Gl(n,\mathbb{R})$,
$\mathfrak{g}\cong \mathfrak{gl}(n,\mathbb{R})$, and,
as calculation shows,
$\left( S
\left(
\mathfrak{g}^\ast
\right)
\otimes \bigwedge
\left(
V^\ast
\right)
\right) ^G
=S\left(
\mathfrak{g}^\ast
\right) ^G\oplus (S
\left(
\mathfrak{g}^\ast
\right) ^G
\otimes v^1)$. In fact, every
$t\in S^a(\mathfrak{g}^\ast )\otimes \bigwedge ^b(V^\ast )$
can be written as,
\begin{align*}
t & =\sum _{2\leq i_2<\ldots <i_b\leq 2n+1}s_{a,I}
\otimes v^1\wedge v^{i_2}\wedge \ldots \wedge v^{i_b}\\
& +\sum _{2\leq j_1<\ldots <j_b\leq 2n+1}s_{a,J}\otimes
v^{j_1}\wedge v^{j_2}\wedge \ldots \wedge v^{j_b},
\end{align*}
where $I=(i_2,\dotsc,i_b)\in \mathbb{N}^{b-1}$,
$J=(j_1,\dotsc,j_b)\in \mathbb{N}^b$,
$s_{a,I},s_{a,J}\in S^a(\mathfrak{g}^\ast )$.

If
$A\in G=\iota (Gl(n,\mathbb{R}))
\subset Gl(2n+1,\mathbb{R})$
is the matrix in Proposition \ref{Prop2}
corresponding to $\Lambda =\lambda I_n$,
$\lambda \in \mathbb{R}^\ast $,
i.e., $A\cdot v^1=v^1$, $A\cdot v^i=\lambda v^i$,
$2\leq i\leq 2n+1$, then the invariance equation
$A\cdot t-t=0$ is equivalent to saying,
\begin{align*}
0  & =\sum _{2\leq i_2<\ldots <i_b\leq 2n+1}
\left(
\lambda -\lambda ^b
\right)
s_{a,I}
\otimes v^1\wedge v^{i_2}\wedge \ldots \wedge v^{i_b}\\
& +\sum _{2\leq j_1<\ldots <j_b\leq 2n+1}
\left(
1-\lambda ^b
\right)
s_{a,J}
\otimes v^{j_1}\wedge v^{j_2}\wedge \ldots \wedge v^{j_b},
\end{align*}
and the result readily follows.

Accordingly, if $\theta $ is the soldering form
on $F(M^1)$, $\Omega ^\sigma $ is the curvature form
of the Chern connection attached to $\sigma $,
and $\theta ^\prime =\theta |_{P^\sigma }$,
$\Omega ^{\prime \sigma }=\Omega ^\sigma |_{P^\sigma }$
(cf.\ \cite[II, Proposition 6.1-(b)]{KN})
are their restrictions to $P^\sigma $ respectively,
then for every Weil polynomial
$f\in S(\mathfrak{g}^\ast )^G$ we obtain
$f\left(
\Omega ^{\prime \sigma }
\right)
\wedge \theta ^{\prime 1}
=d\left(
tf(\Omega ^{\prime \sigma })
\right) $.

In summary, the Chern connection attached
to an arbitrary SODE $\sigma $ on $M$
determines the Chern classes on $M$
in the standard Chern-Weil homomorphism
(under the natural isomorphism
$H^\bullet (M^1;\mathbb{R})
\cong H^\bullet (M;\mathbb{R})$),
whereas the characteristic forms of odd degree
attached to $P^\sigma $ are all exact.
This also shows that the sufficient
condition on the connection $\omega $
of being symmetric in \cite[3.4, Theorem]{AM}
is not necessary.

\section{Holonomy of $\nabla ^\sigma $}
\label{section_holonomy}

\subsection{General holonomy}

\begin{proposition}
\label{Holonomy}
Assume $M$ and $\sigma$ are of class $C^\omega $.
If the equation
\begin{multline*}
0=2R^\sigma
\left(
X,Y
\right)
(U)+R^\sigma
\left(
\left(
H^\sigma \circ \varepsilon ^{-1}
\right)
Z,T
\right)
(U) \\
+R^\sigma
\left(
\left(
H^\sigma \circ\varepsilon ^{-1}
\right)  T,Z
\right)
(U) ,
\end{multline*}
$X,Y\in T_\xi ^-M^1$, $Z,T,U\in T_\xi ^+M^1$,
implies $X=Y=Z=T=U=0$, then the holonomy algebra
of $\nabla ^\sigma $ is
$\mathfrak{gl}(n,\mathbb{R})$.
\end{proposition}

\begin{proof}
According to the hypothesis in the statement,
the holonomy algebra can be computed
from the infinitesimal holonomy algebra
(see \cite[II, Theorem10.8]{KN}). Moreover,
from \cite[III, Theorem 9.2]{KN} we know
that such algebra is spanned by the endomorphisms
\[
\left(
\left(
\nabla ^\sigma
\right) ^k
R^\sigma
\right)
\left(
X,Y;V_1;\ldots;V_k
\right) ,
\quad
\forall X,Y,V_1,\dotsc,V_k
\in T_\xi M^1,
\;\forall k\in\mathbb{N}.
\]
For $k=0$, from the formulas
\eqref{R_sigma} we obtain
\begin{align}
R^\sigma (X_i^\sigma ,X_j^\sigma )
& =B_{ijk}^h
\left(
\omega ^k\otimes X_h^\sigma
+\varpi ^k
\otimes \dfrac{\partial }{\partial \dot{x}^h}
\right) ,
\quad
i<j,
\label{Curvature_B}\\
R^\sigma
\left(
X_i^\sigma ,\dfrac{\partial }{\partial \dot{x}^j}
\right)
& =R_{ijk}^h
\left(
\omega ^k\otimes X_h^\sigma +\varpi ^k
\otimes \dfrac{\partial }{\partial \dot{x}^h}
\right) ,
\quad
i\leq j.\label{Curvature_R}
\end{align}
Hence
\begin{align*}
\left.
R^\sigma
\left(
X_i^\sigma ,X_j^\sigma
\right)
\right\vert _{T^+(M^1)}
& =B_{ijk}^h\varpi ^k\otimes
\frac{\partial }{\partial \dot{x}^h},
\quad
i<j,\\
\left.
R^\sigma
\left(
X_i^\sigma ,
\dfrac{\partial }{\partial \dot{x}^j}
\right)
\right\vert _{T^+(M^1)}
& =R_{ijk}^h\varpi^k\otimes
\frac{\partial }{\partial \dot{x}^h}
\quad
i\leq j.
\end{align*}

If the matrices $
\left(
B_{ijk}^h(\xi)
\right) _{h,k=1}^n$, $i<j$,
$\left(
R_{ijk}^h(\xi )
\right) _{h,k=1}^n$, $i\leq j$ span
$\mathfrak{gl}(n,\mathbb{R})$,
we can conclude. Moreover, if
\[
\Upsilon ^\sigma \colon \bigwedge ^2T^-(M^1)
\oplus S^2T^+(M^1)\to \mathrm{End}T^+(M^1)
\]
is the homomorphism given by,
\begin{multline*}
\Upsilon ^\sigma
\left(  X\wedge Y,Z\odot T
\right)
(U)
=R^\sigma
\left(
X,Y
\right)
(U) \\
+\tfrac{1}{2}
\left\{
R^\sigma
\left(
\left(
H^\sigma \circ \varepsilon ^{-1}
\right)
Z,T
\right)
(U)+R^\sigma
\left(
\left(
H^\sigma \circ \varepsilon ^{-1}
\right)
T,Z
\right)
(U)
\right\}  ,
\end{multline*}
then the condition in the statement
is readily seen to be equivalent
to saying $\Upsilon ^\sigma $
is an isomorphism, as its matrix
on the bases
\[
\left(
X_i^\sigma \wedge X_j^\sigma ,
\dfrac{\partial }{\partial \dot{x}^h}
\odot
\dfrac{\partial }{\partial \dot{x}^k}
\right) ,
\;i<j,h\leq k;
\quad
\left(
\varpi ^a\otimes
\frac{\partial }{\partial \dot{x}^b}
\right) ,
\;a,b=1,\dotsc,n,
\]
of $\bigwedge ^2T^-(M^1)\oplus S^2T^+(M^1)$,
$\mathrm{End}T^+(M^1)=T^+(M^1)^\ast \otimes T^+(M^1)$,
respectively, is
$\left(
\left(
B_{ijb}^a
\right) _{a,b=1}^n,
\left(
R_{hkb}^a
\right) _{a,b=1}^n
\right) $.
\end{proof}

\subsection{Special holonomy}

Next, we determine the conditions under which
the holonomy algebra of $\nabla ^\sigma $
is contained in $\mathfrak{sl}(n,\mathbb{R})$.
If $M$ and $\sigma $ still are of class
$C^\omega $, then according to
(\cite[Lemma 1, p.\ 152]{KN}), the holonomy
algebra is spanned by the endomorphisms
\begin{equation}
\label{endomorphisms}
\nabla _{V_{l}}^\sigma \cdots \nabla _{V_1}^\sigma
\left(
R^\sigma
\left(
X,Y
\right)
\right) ,
\quad
\forall X,Y,V_1,\dotsc,V_l\in \mathfrak{X}(M^1),
\;\forall l\in\mathbb{N}.
\end{equation}

\begin{lemma}
\label{lemma_S}For every system of vector fields
$X,Y,V_1,\dotsc,V_l\in \mathfrak{X}(M^1)$,
$l\in\mathbb{N}$, the endomorphism
$\nabla _{V_l}^\sigma \cdots \nabla _{V_1}^\sigma
\left(
R^\sigma
\left(
X,Y
\right)
\right) $
can locally be written as follows:
\[
S_k^h\left(
\omega ^k\otimes X_h^\sigma
+\varpi ^k
\otimes \dfrac{\partial }{\partial \dot{x}^h}
\right) ,
\quad
S_k^h\in C^\infty (M^1).
\]
\end{lemma}

\begin{proof}
For $l=0$ from the formulas \eqref{R_sigma}
we obtain
\[
R^\sigma (X^\sigma ,X_i^\sigma )=A_{ki}^h
\left(
\omega ^k\otimes X_h^\sigma
+\varpi ^k
\otimes \dfrac{\partial }{\partial \dot{x}^h}
\right) ,
\]
which, together with the formulas
\eqref{Curvature_B} and \eqref{Curvature_R},
prove the statement in this case.
For $l\geq1$ the proof is by induction. If
\[
\left(
\nabla ^\sigma
\right) _{V_{l-1}}^\sigma
\cdots \nabla _{V_1}^\sigma
\left(
R^\sigma (X,Y)
\right)
=S_k^h
\left(
\omega ^k\otimes X_h^\sigma
+\varpi ^k\otimes
\dfrac{\partial }{\partial \dot{x}^h}
\right) ,
\]
then,
\begin{align*}
\nabla ^\sigma _{X^\sigma }
\left(
\nabla _{V_{l-1}}^\sigma
\cdots \nabla _{V_1}^\sigma
\left(
R^\sigma (X,Y)
\right)
\right)
& =S_{0,k}^h
\left(
\omega ^k\otimes X_h^\sigma
+\varpi ^k
\otimes \dfrac{\partial }{\partial \dot{x}^h}
\right) ,\\
\nabla ^\sigma _{X_j^\sigma }
\left(
\nabla _{V_{l-1}}^\sigma
\cdots \nabla _{V_1}^\sigma
\left(
R^\sigma (X,Y)
\right)
\right)
& =S_{j,k}^h
\left(
\omega ^k\otimes X_h^\sigma
+\varpi^k
\otimes \dfrac{\partial }{\partial \dot{x}^h}
\right) ,\\
\nabla _{\frac{\partial }{\partial \dot{x}^j}}^\sigma
\left(
\nabla _{V_{l-1}}^{\sigma
}\cdots\nabla _{V_1}^\sigma
\left(
R^\sigma (X,Y)
\right)
\right)
& =\frac{\partial S_k^h}{\partial \dot{x}^j}
\left(
\omega ^k\otimes X_h^\sigma
+\varpi ^k\otimes
\dfrac{\partial }{\partial \dot{x}^h}
\right) ,
\end{align*}
where
\begin{align*}
S_{0,k}^h & =X^\sigma
\left(
S_k^h
\right)
+\tfrac{1}{2}S_r^h
\dfrac{\partial F^r}{\partial \dot{x}^k}
-\tfrac{1}{2}S_k^r
\dfrac{\partial F^h}{\partial \dot{x}^r},\\
S_{j,k}^h
& =X_j^\sigma
\left(
S_k^h
\right)
+\tfrac{1}{2}S_r^h
\dfrac{\partial^2F^r}
{\partial \dot{x}^j\partial \dot{x}^k}
-\tfrac{1}{2}S_k^r
\dfrac{\partial^2F^h}
{\partial \dot{x}^j\partial \dot{x}^r}.
\end{align*}

\end{proof}

By induction on $l$ the endomorphisms
\eqref{endomorphisms} are proved to be
traceless if and only if there exists
a function $F\in C^\infty (M^0)$ such that,
\begin{equation}
\label{holonomia_special}
\sum _{h=1}^n
\frac{\partial F^h}{\partial \dot{x}^h}
=\frac{\partial F}{\partial t}
+\frac{\partial F}{\partial x^i}\dot{x}^i.
\end{equation}
For $l=0$ taking the formula \eqref{R's}
into account, we obtain
\[
\operatorname*{tr}
\left(  R_{ijk}^h
\right) _{h,k=1}^n
=\dfrac{\partial ^2}
{\partial \dot{x}^i\partial \dot{x}^j}
\left(
\sum _{h=1}^n\frac{\partial F^h}
{\partial \dot{x}^h}
\right) ,
\quad
1\leq i\leq j\leq n.
\]
Hence the matrices $(R_{ijk}^h)_{h,k=1}^n$
are traceless, if and only if,
\[
\sum _{h=1}^n\partial F^h/\partial \dot{x}^h
=F_0+F_i\dot{x}^i,
\quad
F_0,F_i\in C^\infty (M^0).
\]
Furthermore, taking the formula above
and the identity
\begin{align*}
2B_{ijk}^h
& =-\dfrac{\partial ^3F^h}
{\partial x^i\partial \dot{x}^j\partial \dot{x}^k}
-\tfrac{1}{2}
\frac{\partial F^l}{\partial \dot{x}^i}
\dfrac{\partial ^3F^h}
{\partial \dot{x}^j\partial \dot{x}^k\partial \dot{x}^l}
+\dfrac{\partial^3F^h}
{\partial x^j\partial \dot{x}^i\partial \dot{x}^k}\\
& +\tfrac{1}{2}
\frac{\partial F^{l}}{\partial \dot{x}^j}
\dfrac{\partial ^3F^h}
{\partial \dot{x}^i\partial \dot{x}^k\partial \dot{x}^l}
+\tfrac{1}{2}
\dfrac{\partial ^2F^r}
{\partial \dot{x}^j\partial \dot{x}^k}
\dfrac{\partial ^2F^h}
{\partial \dot{x}^i\partial \dot{x}^r}
-\tfrac{1}{2}\dfrac{\partial ^2F^r}
{\partial \dot{x}^i\partial \dot{x}^k}
\dfrac{\partial^2F^h}
{\partial \dot{x}^j\partial \dot{x}^r}\nonumber
\end{align*}
into account, we conclude that the matrices
$(B_{ijk}^h)_{h,k=1}^n$ are traceless if and only if,
$\partial F_j/\partial x^i=\partial F_i/\partial x^j$
for $1\leq i<j\leq n$. Finally, from the identity
\begin{align*}
2A_{kj}^h & =-\dfrac{\partial ^3F^h}
{\partial t\partial \dot{x}^k\partial \dot{x}^j}
-\dot{x}^r
\dfrac{\partial ^3F^h}
{\partial x^r\partial \dot{x}^k\partial \dot{x}^j}
-F^r\dfrac{\partial ^3F^h}
{\partial \dot{x}^r\partial \dot{x}^k\partial \dot{x}^j}
+\dfrac{\partial ^2F^h}{\partial x^j\partial \dot{x}^k}\\
& +\tfrac{1}{2}\dfrac{\partial F^h}
{\partial \dot{x}^r}\dfrac{\partial ^2F^r}
{\partial \dot{x}^k\partial \dot{x}^j}
-\tfrac{1}{2}\dfrac{\partial F^r}
{\partial \dot{x}^k}\dfrac{\partial^2F^h}
{\partial \dot{x}^r\partial \dot{x}^j},\nonumber
\end{align*}
we deduce that the matrices $(A_{kj}^h)_{h,k=1}^n$
are traceless for $1\leq j\leq n$ if and only if,
$\partial F_j/\partial t=\partial F_0/\partial x^j$.

For $l\geq 1$, by applying Lemma \ref{lemma_S}
and the induction hypothesis, we obtain
\[
\nabla _{V_{l-1}}^\sigma \cdots \nabla _{V_1}^\sigma
\left(
R^\sigma (X,Y)
\right)
=S_k^h\left( \omega ^k\otimes X_h^\sigma
+\varpi ^k\otimes
\dfrac{\partial }{\partial \dot{x}^h}
\right) ,
\;S_h^h=0.
\]
Hence, taking the formulas for $S_{0,k}^h$ and
$S_{j,k}^h$ at the end of the proof of Lemma
\ref{lemma_S} into account we obtain
\[
\operatorname*{tr}
\left(
S_{0,k}^h
\right) _{h,k=1}^n
=X^\sigma
\left(
\operatorname*{tr}
\left(
S_k^h
\right) _{h,k=1}^n
\right)
+\tfrac{1}{2}
\left(
S_r^h
\dfrac{\partial F^r}{\partial \dot{x}^h}
-S_h^r
\dfrac{\partial F^h}{\partial \dot{x}^r}
\right)
=0,
\]
\[
\operatorname*{tr}
\left(
S_{j,k}^h
\right) _{h,k=1}^n
=X_j^{\sigma
}\left(
\operatorname*{tr}
\left(
S_k^h
\right) _{h,k=1}^n
\right)
+\tfrac{1}{2}
\left(
S_r^h
\dfrac{\partial^2F^r}
{\partial \dot{x}^j\partial \dot{x}^h}
-S_h^r
\dfrac{\partial^2F^h}
{\partial \dot{x}^j\partial \dot{x}^r}
\right)
=0,
\]
\[
\operatorname*{tr}
\left(
\frac{\partial S_k^h}{\partial \dot{x}^j}
\right) _{h,k=1}^n
=\frac{\partial }{\partial \dot{x}^j}
\left(
\operatorname*{tr}
\left(
S_k^h
\right) _{h,k=1}^n
\right)
=0.
\]
As a calculation shows, the equation
\eqref{holonomia_special} means
that the volume form
\[
\exp (-F)dt\wedge \omega ^1\wedge \ldots \wedge
\omega ^n\wedge \varpi ^1\wedge \ldots\wedge
\varpi ^n
\]
is parallel with respect to $\nabla ^\sigma $.

\subsection{Orthogonal holonomy}

\begin{proposition}
\label{matrixU}
The holonomy group of $\nabla ^\sigma $ is contained
in $SO(n)$ if and only if for every $\xi\in M^1$ there
exist a coordinate neighbourhood $(N^0;t,x^i)$
of $p^{10}(\xi)$ in $M^0$ and a positive definite
symmetric matrix
$U=(u_j^i)_{i,j=1}^{n}$, $u_j^i\in C^\infty (N^0)$,
such that the following equation holds:
\begin{equation}
\label{PDE}
\frac{\partial U}{\partial t}
+\dot{x}^i\frac{\partial U}{\partial x^i}
+UW+(UW)^t=0,
\end{equation}
where
$W=(w_j^i)$,
$w_j^i
=\tfrac{1}{2}
\dfrac{\partial F^i}{\partial \dot{x}^j}$.
\end{proposition}

\begin{proof}
The holonomy group of $\nabla ^\sigma $
is contained in $SO(n)$ if
and only if there exists a Riemannian
metric $g^1$ on $M^1$ such that,

\begin{enumerate}
\item[(i)]
$g^1$ is parallel with respect
to $\nabla ^\sigma $.

\item[(ii)]
For every $\xi\in M^1$ there exist
an open neighbourhood $N^1$ and a section
of $P^\sigma $ defined over $N^1$,
$(X^\sigma ,\bar{X}_j^\sigma
=\Lambda _j^iX_i^\sigma ,
\bar{X}_j
=\Lambda_j^i\partial /\partial \dot{x}^i)_{j=1}^n
$,
$\Lambda _j^i\in C^\infty (N^1)$,
which is an orthonormal linear frame
with respect to $g^1$;
e.g., see \cite[II, \S 7, Lemma 2; III,
Proposition 1.5; IV, Proposition 2.1]{KN}.
\end{enumerate}

We impose the condition $\nabla ^\sigma g^1=0$,
by using the basis
$(X^\sigma ,X_i^\sigma ,
\partial /\partial \dot{x}^i)_{i=1}^n$.
From the identities
{\small
\begin{equation}
\begin{array}
[c]{l}
g^1\left(
X^\sigma ,X_k^\sigma
\right)
=g^1\left(
X^\sigma ,
\left(
\Lambda ^{-1}
\right) _k^a
\Lambda _a^rX_r^\sigma
\right)
=\left(
\Lambda ^{-1}
\right) _k^a g^1
\left(
X^\sigma ,\bar{X}_a^\sigma
\right)
=0,\\
g^1\left(
X^\sigma ,\partial/\partial \dot{x}^k
\right)
=\left(
\Lambda ^{-1}
\right) _k^ag^1
\left(
X^\sigma ,
\Lambda_a^r\partial/\partial \dot{x}^r
\right)
=\left(
\Lambda^{-1}
\right) _k^ag^1
\left(
X^\sigma ,\bar{X}_a
\right)
=0,\\
g^1\left(
X_h^\sigma ,X_k^\sigma
\right)
=\left(
\Lambda^{-1}
\right) _h^bg^1
\left(
\bar{X}_b^\sigma ,\bar{X}_a^\sigma
\right)
\left(
\Lambda^{-1}
\right)  _k^a=
\left(
\Lambda^{-1}
\right)  _h^b
\delta _{ba}
\left(
\Lambda^{-1}
\right)  _k^a,\\
g^1\left(
X_h^\sigma ,\partial /\partial \dot{x}^k
\right)
=\left(
\Lambda^{-1}
\right) _b^hg^1
\left(
\bar{X}_b^\sigma ,\bar{X}_a
\right)
\left(
\Lambda ^{-1}
\right) _k^a
=0,\\
g^1\left(
\partial /\partial \dot{x}^h,
\partial /\partial \dot{x}^k
\right)
=\left(
\Lambda ^{-1}
\right) _b^hg^1
\left(  \bar{X}_b,\bar{X}_a
\right)
\left(
\Lambda^{-1}
\right) _k^a
=\left(
\Lambda ^{-1}
\right) _h^b\delta _{ba}
\left(
\Lambda^{-1}
\right) _k^a,
\end{array}
\label{g}
\end{equation}
}
and the formulas in \eqref{Tabla},
we conclude the following equations
hold identically:
\begin{align*}
X^\sigma
\left(
g^1(X^\sigma ,X^\sigma )
\right)
& =g^1\left(
\nabla _{X^\sigma }^\sigma X^\sigma ,X^\sigma
\right)
+g^1\left(
X^\sigma ,\nabla _{X^\sigma }^\sigma X^\sigma
\right) ,\\
X^\sigma
\left(
g^1(X^\sigma ,X_i^\sigma )
\right)
& =g^1
\left(
\nabla _{X^\sigma }^\sigma X^\sigma ,X_i^\sigma
\right)
+g^1\left(
X^\sigma ,
\nabla _{X^\sigma }^\sigma X_i^\sigma
\right) ,\\
X^\sigma
\left(
g^1\left(
X^\sigma ,
\partial /\partial \dot{x}^j
\right)
\right)
& =g^1
\left(
\nabla _{X^\sigma }^\sigma X^\sigma ,
\partial /\partial \dot{x}^j
\right)
+g^1\left(
X^\sigma ,\nabla _{X^\sigma }^\sigma
\partial/\partial \dot{x}^j
\right) ,\\
X^\sigma
\left(
g^1\left(
X_i^\sigma ,\partial /\partial \dot{x}^j
\right)
\right)
& =g^1
\left(
\nabla _{X^\sigma }^\sigma X_i^\sigma ,
\partial /\partial \dot{x}^j
\right)
+g^1\left(
X_i^\sigma ,\nabla _{X^\sigma }^\sigma
\partial /\partial \dot{x}^j
\right) ,
\end{align*}
\begin{align*}
X_k^\sigma
\left(
g^1(X^\sigma ,X^\sigma )
\right)
& =g^1
\left(
\nabla _{X_k^\sigma }^\sigma X^\sigma ,X^\sigma
\right)
+g^1\left(
X^\sigma ,\nabla _{X_k^\sigma }^\sigma X^\sigma
\right) ,\\
X_k^\sigma
\left(
g^1(X^\sigma ,X_i^\sigma )
\right)
& =g^1
\left(
\nabla _{X_k^\sigma }^\sigma X^\sigma ,X_i^\sigma
\right)
+g^1\left(
X^\sigma ,\nabla _{X_k^\sigma }^\sigma X_i^\sigma
\right)  ,\\
X_k^\sigma
\left(
g^1\left(
X^\sigma ,\partial/\partial \dot{x}^j
\right)
\right)
& =g^1
\left(  \nabla _{X_k^\sigma }^\sigma X^\sigma ,
\partial /\partial \dot{x}^j
\right)
+g^1\left(
X^\sigma ,\nabla _{X_k^\sigma }^\sigma
\partial /\partial \dot{x}^j
\right) ,\\
X_k^\sigma
\left(
g^1\left(
X_i^\sigma ,\partial/\partial \dot{x}^j
\right)
\right)
& =g^1
\left(
\nabla _{X_k^\sigma }^\sigma X_i^\sigma ,
\partial /\partial \dot{x}^j
\right)
+g^1\left(
X_i^\sigma ,\nabla _{X_k^\sigma }^\sigma
\partial /\partial \dot{x}^j
\right) ,
\end{align*}
\begin{align*}
\frac{\partial }{\partial \dot{x}^k}
\left(
g^1(X^\sigma ,X^\sigma )
\right)
& =g^1
\left(
\nabla _{\partial/\partial \dot{x}^k}^\sigma
X^\sigma ,X^\sigma
\right)
+g^1\left(
X^\sigma ,
\nabla _{\partial /\partial \dot{x}^k}^\sigma
X^\sigma
\right) ,\\
\frac{\partial }{\partial \dot{x}^k}
\left(
g^1(X^\sigma ,X_i^\sigma )
\right)
& =g^1
\left(
\nabla _{\partial /\partial \dot{x}^k}^\sigma
X^\sigma ,X_i^\sigma
\right)
+g^1\left(
X^\sigma ,\nabla _{\partial
/\partial \dot{x}^k}^\sigma X_i^\sigma
\right)  ,\\
\frac{\partial }{\partial \dot{x}^k}
\left(
g^1\left(
X^\sigma ,\partial /\partial \dot{x}^j
\right)
\right)
& =g^1\left(
\nabla _{\partial /\partial \dot{x}^k}^\sigma
X^\sigma ,\partial/\partial \dot{x}^j
\right)
+g^1\left(
X^\sigma ,
\nabla _{\partial /\partial \dot{x}^k}^\sigma
\partial/\partial \dot{x}^j
\right)  ,\\
\frac{\partial }{\partial \dot{x}^k}
\left(
g^1\left(
X_i^\sigma ,\partial /\partial \dot{x}^j
\right)
\right)
& =g^1
\left(
\nabla _{\partial /\partial \dot{x}^k}^\sigma
X_i^\sigma ,\partial /\partial \dot{x}^j
\right)
+g^1\left(
X_i^\sigma ,
\nabla _{\partial /\partial \dot{x}^k}^\sigma
\partial /\partial \dot{x}^j
\right) ,
\end{align*}
and the rest of conditions
for $\nabla ^\sigma g^1=0$ leads
us to the following equations
for $i,j,k=1,\dotsc,n$:
\[
\begin{array}
[c]{l}
X^\sigma
\left(
g^1(X_i^\sigma ,X_j^\sigma )
\right)
=-\tfrac{1}{2}
\left(
g^1
\left(
X_j^\sigma ,X_k^\sigma
\right)
\dfrac{\partial F^k}{\partial \dot{x}^i}
+g^1\left(
X_i^\sigma ,X_k^\sigma
\right)
\dfrac{\partial F^k}{\partial \dot{x}^j}
\right)  ,\\
X_k^\sigma
\left(
g^1(X_i^\sigma ,X_j^\sigma )
\right)
=-\tfrac{1}{2}
\left(
g^1\left(
X_j^\sigma ,X_h^\sigma
\right)
\dfrac{\partial ^2F^h}
{\partial \dot{x}^i\partial \dot{x}^k}
+g^1\left(
X_i^\sigma ,X_h^\sigma
\right)
\dfrac{\partial ^2F^h}
{\partial \dot{x}^j\partial \dot{x}^k}
\right)  ,\\
\frac{\partial }
{\partial \dot{x}^k}
\left(
g^1(X_i^\sigma ,X_j^\sigma )
\right) =0,
\end{array}
\]
\[
\begin{array}
[c]{l}
X^\sigma
\left(
g^1\left(
\dfrac{\partial }{\partial \dot{x}^i},
\dfrac{\partial }{\partial \dot{x}^j}
\right)
\right)
=-\tfrac{1}{2}
\left(
g^1\left(
\dfrac{\partial }{\partial \dot{x}^j},
\dfrac{\partial }{\partial \dot{x}^k}
\right)
\dfrac{\partial F^k}{\partial \dot{x}^i}
+g^1\left(
\dfrac{\partial }{\partial \dot{x}^i},
\dfrac{\partial }{\partial \dot{x}^k}
\right)
\dfrac{\partial F^k}{\partial \dot{x}^j}
\right) ,\\
X_k^\sigma
\left(
g^1\left(
\dfrac{\partial }{\partial \dot{x}^i},
\dfrac{\partial }{\partial \dot{x}^j}
\right)
\right)
=-\tfrac{1}{2}
\left(
g^1\left(
\dfrac{\partial }{\partial \dot{x}^j},
\dfrac{\partial }{\partial \dot{x}^h}
\right)
\dfrac{\partial ^2F^h}
{\partial \dot{x}^i\partial \dot{x}^k}
+g^1\left(
\dfrac{\partial }{\partial \dot{x}^i},
\dfrac{\partial }{\partial \dot{x}^h}
\right)
\dfrac{\partial^2F^h}{\partial \dot{x}^j
\partial \dot{x}^k}
\right) ,\\
\frac{\partial }{\partial \dot{x}^k}
\left(
g^1\left(
\dfrac{\partial }{\partial \dot{x}^i},
\dfrac{\partial }{\partial \dot{x}^j}
\right)
\right) =0.
\end{array}
\]
As the first group of three equations above
is equivalent to the second group, taking
\eqref{g} into account, these six equations
reduce to the following:
\[
0=\delta _{ba}
\left\{
2X^\sigma
\left(
\left(
\Lambda ^{-1}
\right) _i^b
\left(
\Lambda ^{-1}
\right) _j^a
\right)
+\left(
\left(
\Lambda ^{-1}
\right) _j^b
\dfrac{\partial F^k}{\partial \dot{x}^i}
+\left(
\Lambda ^{-1}
\right) _i^b
\dfrac{\partial F^k}{\partial \dot{x}^j}
\right)
\left(
\Lambda ^{-1}\right) _k^a
\right\} ,
\]
\begin{align*}
0  & =\delta_{ba}
\left\{
2X_k^\sigma
\left(
\left(
\Lambda ^{-1}
\right) _i^b
\left(
\Lambda ^{-1}
\right)  _j^a
\right)
\right. \\
& \left.
+\left(
\left(
\Lambda^{-1}
\right) _j^b
\dfrac{\partial ^2F^h}
{\partial \dot{x}^i\partial \dot{x}^k}
+\left(
\Lambda ^{-1}
\right) _i^b
\dfrac{\partial ^2F^h}
{\partial \dot{x}^j\partial \dot{x}^k}
\right)
\left(
\Lambda ^{-1}
\right) _h^a\right\} ,
\end{align*}
\[
\begin{array}
[c]{l}
\delta _{ba}\dfrac{\partial }
{\partial \dot{x}^k}
\left(
\left(
\Lambda ^{-1}
\right) _i^b
\left(
\Lambda {-1}
\right) _j^a
\right)
=0.
\end{array}
\]

For $a\neq b$ all these equations vanish
identically and for $a=b$, by writing
$U=\left(
\Lambda ^t
\right) ^{-1}\Lambda ^{-1}$,
such equations are
\begin{align}
X^\sigma (U)+UW+(UW)^t
& =0,\label{ecuacion1}\\
X_k^\sigma (U)+UV_k
+\left(
UV_k
\right) ^t
& =0,\label{ecuacion2}
\end{align}
where $V_k=(v_{ik}^h)$,
$v_{ik}^h=\tfrac{1}{2}
\dfrac{\partial^2F^h}
{\partial \dot{x}^i\partial \dot{x}^k}$,
together with the following:
\begin{equation}
\label{ecuacion3}
\dfrac{\partial U}{\partial \dot{x}^k}=0.
\end{equation}
The equations \eqref{ecuacion3} are equivalent
to saying the entries $u_j^i$ of $U$ belong
to $C^\infty (M^0)$. Hence, the equations
\eqref{ecuacion1}, \eqref{ecuacion2} can also
be rewritten respectively as
\begin{align*}
\frac{\partial U}{\partial t}
+\dot{x}^i
\frac{\partial U}{\partial x^i}
+UW+(UW)^t
& =0,\\
\frac{\partial U}{\partial x^k}
+UV_k
+\left(
UV_k
\right) ^t
& =0.
\end{align*}
Finally, we claim that the second equation
is a consequence of the first one
and \eqref{ecuacion3}, as follows taking
derivatives with respect to $\dot{x}^k$
in the first equation above.
\end{proof}

\begin{remark}
As a simple---but rather long---computations
shows, the integrability conditions
for the system \eqref{ecuacion1},
\eqref{ecuacion2}, and \eqref{ecuacion3} are
\begin{equation}
\begin{array}
[c]{ll}
0=UR_{ij}
+\left(
R_{ij}\right) ^tU,
& 1\leq i\leq j\leq n,\\
0=UB_{ij}
+\left(
B_{ij}
\right) ^tU,
& 1\leq i<j\leq n,\\
0=UA_j
+\left(
A_j
\right) ^tU,
& 1\leq j\leq n,
\end{array}
\label{UABR}
\end{equation}
where the square matrices
$A_j$, $B_{ij}$, $R_{ij}$
are given by $A_j=(A_{kj}^h)$,
$B_{ij}=(B_{ijk}^h)$,
$R_{ij}=(R_{ijk}^h)$,
and the functions $A_{kj}^h$,
$B_{ijk}^h$, $R_{ijk}^h$ are defined
by the formulas \eqref{A's},
\eqref{B's}, \eqref{R's}, respectively.
If $U$ is a positive definite symmetric
matrix, then
\[
\mathfrak{g}_U
=\left\{
X\in \mathfrak{gl}(n,\mathbb{R}):UX+X^tU
=0
\right\}
\]
is a Lie subalgebra of dimension
$\frac{1}{2}n(n-1)$. Hence, if a matrix
$U$ exists satisfying \eqref{UABR},
then the dimension of the subalgebra in
$\mathfrak{gl}(n,\mathbb{R})$ generated
by the $n(n+1)$ matrices
$\left\{
A_h
\right\} _{h=1}^{n}$,
$\left\{
B_{ij}
\right\} _{1\leq i<j\leq n}$,
$\left\{
R_{kl}
\right\} _{1\leq k\leq l\leq n}$
must be $\leq\frac{1}{2}n(n-1)$.
\end{remark}

\begin{remark}
\label{remark3}
The matrix $U$ depends only on the symmetric part
of the polar decomposition of $\Lambda $; namely,
if $\Lambda =SR$, where $R\in SO(n)$ and $S$ is a
positive definite symmetric matrix, then $U=S^{-2}$.
\end{remark}

\begin{example}
According to Remark \ref{remark_connection},
the Levi-Civita connection of a pseudo-Riemannian
metric $g=g_{ij}dx^i\otimes dx^j$ on $M$ induces
a homogeneous quadratic SODE
$\sigma $ independent of $t$, given by
$F^h=F_{ij}^h\dot{x}^i\dot{x}^j$, where
\[
F_{ij}^h
=-\tfrac{1}{2}g^{hk}
\left(
\frac{\partial g_{ki}}{\partial x^j}
+\frac{\partial g_{jk}}{\partial x^i}
-\frac{\partial g_{ji}}{\partial x^k}
\right) .
\]
As in Remark \ref{remark_connection}
the component of the Chern connection
in the tangent bundle $TM$ coincides
with the horizontal lift $\nabla ^H$
of the Levi-Civita connection $\nabla $
of $g$, which is known to be a metric
connection with respect to the metric
$g_{I\!I}$ on $TM$ defined in
\cite[p.\ 137]{YanoIshihara}; in the present
case, $g_{I\!I}$ takes the form
$g_{I\!I}=g_{ij}(\omega ^i\otimes \varpi ^j
+\varpi ^i\otimes \omega ^j)+g_{ij}\dot{x}^i
(dt\otimes \varpi ^j+\varpi ^j\otimes dt)$.
Using this fact, it is readily seen that
$\nabla ^\sigma $ parallelizes the metric
$h^1=dt\otimes dt+g_{I\!I}$ on $M^1$.

Moreover, the equations \eqref{PDE}
hold identically for $U=(g_{ij})_{i,j=1}^n$,
and using Remark \ref{remark3} without lost
of generality we can assume
$\Lambda =U^{-\frac{1}{2}}$. From the results
of Proposition \ref{matrixU} we thus conclude
that $\nabla ^\sigma $ parallelizes
also the metric $g^1=dt\otimes dt
+g_{ij}\omega ^i\otimes \omega ^j
+g_{ij}\varpi ^i\otimes \varpi ^j$.
If $g$ is a Riemannian metric, then $g^1$
is also Riemannian but $h^1$ is maximally
hyperbolic; in fact, its signature is $(n+1,n)$.
\end{example}

\section{Naturality of the Chern connection}
\label{section_naturality}

A diffeomorphism $\Phi \colon M^0\to M^0$
is said to be a $p$-vertical automorphism
of the submersion $p\colon M^0\to \mathbb{R}$
if it takes the form
$\Phi (t,x)
=(t,\phi (t,x))$,
$\forall(t,x)\in M^0$,
$\phi \in C^\infty (M^0,M)$.
The set of such transformations is a group
with respect to composition of maps,
denoted by $\mathrm{Aut}^v(p)$.
For each $r\geq 0$, every
$\Phi \in \mathrm{Aut}^v(p)$ induces
a diffeomorphism
$\Phi ^{(r)}\colon M^r\to M^r$
by setting
\begin{equation}
\label{Phi^r}
\Phi ^{(r)}
\left(
j_t^r\gamma
\right)
=j_t^r
\left(
\Phi \circ j^0\gamma
\right) ,
\quad
\forall \gamma \in C^\infty (\mathbb{R},M).
\end{equation}
If $f\colon N\to N$ is a diffeomorphism and
$X\in \mathfrak{X}(N)$,
then $f\cdot X\in \mathfrak{X}(N)$ is defined
by $(f\cdot X)_X=f_\ast (X_{f^{-1}(x)})$,
$\forall x\in N$.

The connection $\nabla ^\sigma $ enjoys
the important property
of being functorial with respect
to the $p$-vertical automorphisms;
more precisely,

\begin{theorem}
\label{Functoriality1}
For every SODE $\sigma$ on $M$
and every $\Phi \in \mathrm{Aut}^v(p)$,
let $\Phi \cdot \nabla ^\sigma $
be the linear connection defined by,
\[
\left(
\Phi \cdot \nabla ^\sigma
\right) _XY
=\Phi ^{(1)}\cdot \Bigl(
\left(
\nabla ^\sigma
\right) _{(\Phi ^{(1)})^{-1}\cdot X}
\Bigl(
(\Phi ^{(1)})^{-1}\cdot Y
\Bigr)
\Bigr) ,
\quad
\forall X,Y\in \mathfrak{X}(M^1),
\]
and let $\Phi \cdot \sigma $ be
the SODE defined as follows:
\[
\Phi \cdot \sigma =\Phi ^{(2)}
\circ \sigma \circ (\Phi ^{(1)})^{-1},
\]
\[
\begin{array}
[c]{lll}
M^1\smallskip
& \overset\sigma {\longrightarrow \smallskip}
& M^2\smallskip \\
\!\!\!\!\!\!\!\!\!\!\!\!\!\!\!\!\!\!\!\!\!\!
(\Phi ^{(1)})^{-1}\uparrow \smallskip
&  & \;\downarrow \Phi ^{(2)}
\smallskip\\
M^1
& \overset{\Phi \cdot \sigma }{\longrightarrow }
& M^2
\end{array}
\]
Then,
\[
\Phi \cdot \nabla ^\sigma
=\nabla ^{\Phi \cdot \sigma }.
\]
\end{theorem}

\begin{proof}
The definition of $\Phi \cdot \sigma $ makes sense,
as it is readily seen that the map
$\Phi \cdot \sigma \colon M^1\to M^2$ is a section
of $p^{21}$.

By applying the chain rule twice, we obtain
\begin{align}
\dot{x}^h\circ \Phi ^{(1)}
 & =\phi _t^h+\phi _i^h\dot{x}^i,
 \label{x_dot}\\
\ddot{x}^h \circ \Phi ^{(2)}
& =\phi _{tt}^h+2\phi _{ti}^h\dot{x}^i
+\phi _{ij}^h\dot{x}^i\dot{x}^j
+\phi _i^h\ddot{x}^i,
\label{x_dot_dot}
\end{align}
where $\phi ^h=x^h\circ \Phi $,
$\phi _t^h=\frac{\partial\phi ^h}{\partial t}$,
$\phi _i^h=\frac{\partial\phi ^h}{\partial x^i}$,
$\phi _{ti}^h=\frac{\partial^2\phi ^h}
{\partial t\partial x^i}$,
$\phi _{ij}^h=\frac{\partial^2\phi ^h}
{\partial x^i\partial x^j}$, etc. Hence
\begin{align}
\qquad
\Phi ^{(1)}
\cdot \frac{\partial }{\partial t}
& =\frac{\partial }{\partial t}
+\left(
\phi _t^h\circ \Phi ^{-1}
\right)
\frac{\partial }{\partial x^h}
+\left(
\phi _{tt}^h+\phi _{ti}^h\dot{x}^i
\right)
\circ (\Phi ^{(1)})^{-1}
\frac{\partial }{\partial \dot{x}^h},
\label{Phi1_t}\\
\qquad
\Phi ^{(1)}
\cdot \frac{\partial }{\partial x^i}
& =\left(
\phi _i^h\circ \Phi ^{-1}
\right)
\frac{\partial }{\partial x^h}
+\left(
\phi _{ti}^h+\phi _{ij}^h\dot{x}^j
\right)
\circ (\Phi ^{(1)})^{-1}
\frac{\partial }{\partial \dot{x}^h},
\label{Phi1_x}\\
\qquad
\Phi ^{(1)}\cdot
\frac{\partial }{\partial \dot{x}^i}
& =\left(
\phi _i^h\circ\Phi ^{-1}
\right)
\frac{\partial }{\partial \dot{x}^h}.
\label{Phi1_x_dot}
\end{align}

We first prove that for every
$\Phi \in \mathrm{Aut}^v(p)$
the following formula holds:
\begin{equation}
\label{PhiXsigma}
\Phi ^{(1)}\cdot X^\sigma
=X^{\Phi \cdot \sigma }.
\end{equation}

If
$F^h(\Phi \cdot \sigma )
=\ddot{x}^h\circ (\Phi \cdot \sigma )$,
then from the formula \eqref{x_dot_dot}
we obtain
\begin{equation}
\label{FhPhi_sigma}
F^h
\left(
\Phi \cdot \sigma
\right)
\circ \Phi ^{(1)}=\phi _{tt}^h
+2\phi _{ti}^h\dot{x}^i
+\phi _{ij}^h\dot{x}^i\dot{x}^j
+\phi _i^hF^i(\sigma ),
\end{equation}
and by using the formulas \eqref{x_dot},
\eqref{Phi1_t}, \eqref{Phi1_x},
and \eqref{Phi1_x_dot} we deduce
{\small
\begin{align*}
\Phi ^{(1)}\cdot X^\sigma
&  =\frac{\partial }{\partial t}
+\left(
\phi _t^h\circ \Phi ^{-1}
\right)
\frac{\partial }{\partial x^h}
+\left(
\phi _{tt}^h+\phi _{ti}^h\dot{x}^i
\right)
\circ (\Phi ^{(1)})^{-1}
\frac{\partial }{\partial \dot{x}^h}\\
&  +\Bigl(
\dot{x}^i\circ
\left(
\Phi ^{(1)}
\right) ^{-1}
\Bigr)
\left\{
\left(
\phi _i^h\circ \Phi ^{-1}
\right)
\frac{\partial }{\partial x^h}
+\left(
\phi _{ti}^h+\phi _{ij}^h\dot{x}^j
\right)
\circ
\left(
\Phi ^{(1)}
\right) ^{-1}
\frac{\partial }{\partial \dot{x}^h}
\right\} \\
&  +
\Bigl(
F^i(\sigma )\circ
\left(
\Phi ^{(1)}\right) ^{-1}
\Bigr)
\left(
\phi _i^h\circ \Phi ^{-1}
\right)
\frac{\partial }{\partial \dot{x}^h}\\
&  =X^{\Phi \cdot \sigma }.
\end{align*}
}
Similarly, we obtain
\begin{align}
\Phi ^{(1)}\cdot X_i^\sigma
& =\left(
\phi _i^h\circ \Phi ^{-1}
\right)
X_h^{\Phi \cdot \sigma },
\label{PhiXisigma}\\
\Phi ^{(1)}
\cdot \frac{\partial }{\partial \dot{x}^i}
& =\left(
\phi _i^h\circ\Phi ^{-1}
\right)
\frac{\partial }{\partial \dot{x}^h}.
\label{Phixdot}
\end{align}
According to \eqref{Tabla} we know the connection
$\nabla ^{\Phi \cdot \sigma }$ is given by,
{\scriptsize
\begin{equation}
\begin{array}
[c]{lll}
\nabla _{X^{\Phi \cdot \sigma }}^{\Phi \cdot \sigma }
X^{\Phi \cdot \sigma }
\! = \! 0,
& \!\!\!
\nabla _{X^{\Phi \cdot \sigma }}^{\Phi \cdot \sigma }
X_i^{\Phi \cdot \sigma }
\! = \! -\frac{1}{2}
\tfrac{\partial F^j(\Phi \cdot \sigma )}
{\partial \dot{x}^i}X_j^{\Phi \cdot \sigma },
& \!\!\!
\nabla _{X^{\Phi \cdot \sigma }}^{\Phi \cdot \sigma }
\tfrac{\partial }{\partial \dot{x}^i}
\!=\!-\frac{1}{2}
\tfrac{\partial F^j(\Phi \cdot \sigma )}
{\partial \dot{x}^i}
\tfrac{\partial }{\partial \dot{x}^j},
\medskip\\
\nabla _{X_i^{\Phi \cdot \sigma }}^{\Phi \cdot \sigma }
X^{\Phi \cdot \sigma }
\! = \! 0,
& \!\!\!
\nabla _{X_j^{\Phi \cdot \sigma}}^{\Phi \cdot \sigma }
X_i^{\Phi \cdot \sigma }
\! = \!
-\frac{1}{2}
\tfrac{\partial ^2F^k(\Phi \cdot \sigma )}
{\partial \dot{x}^i\partial \dot{x}^j}
X_k^{\Phi \cdot \sigma },
& \!\!\!
\nabla _{X_i^{\Phi \cdot \sigma }}^{\Phi \cdot \sigma}
\tfrac{\partial }{\partial \dot{x}^j}
\! = \!
-\frac{1}{2}
\tfrac{\partial^2F^k(\Phi \cdot \sigma )}
{\partial \dot{x}^i\partial \dot{x}^j}
\tfrac{\partial }
{\partial \dot{x}^k},\medskip\\
\nabla _{\frac{\partial }
{\partial \dot{x}^i}}^{\Phi \cdot \sigma }
X^{\Phi \cdot \sigma }
\! = \! 0,
& \!\!\!
\nabla _{\frac{\partial }
{\partial \dot{x}^i}}^{\Phi \cdot \sigma }
X_j^{\Phi \cdot\sigma}
\! = \! 0,
& \!\!\!
\nabla _{\frac{\partial }
{\partial \dot{x}^i}}^{\Phi \cdot \sigma }
\tfrac{\partial }{\partial \dot{x}^j}
\! = \! 0.
\end{array}
\label{TablaPhi}
\end{equation}
}

Hence, we need only to prove that the formulas
in \eqref{TablaPhi} also hold when
$\nabla ^{\Phi \cdot \sigma }$ is replaced by
$\Phi \cdot \nabla ^\sigma $. We have
\begin{align*}
\left(
\Phi \cdot \nabla ^\sigma
\right) _{Y}X^{\Phi \cdot \sigma } &
=\Phi ^{(1)}
\cdot \Bigl(
\left(
\nabla ^\sigma
\right) _{(\Phi ^{(1)})^{-1}\cdot Y}
\Bigl(
(\Phi ^{(1)})^{-1}\cdot X^{\Phi \cdot \sigma}
\Bigr)
\Bigr) \\
& \overset{\text{\eqref{PhiXsigma}}}{=}\Phi ^{(1)}
\cdot \Bigl(
\nabla _{(\Phi ^{(1)})^{-1}\cdot Y}^\sigma
X^\sigma
\Bigr) \\
& =0,
\end{align*}
for every $Y\in \mathfrak{X}(M^1)$.
As $\Phi $ is a diffeomorphism, the matrix
$(\phi _i^h)_{h,i=1,\dotsc,n}$ is non-singular;
we set
$\Psi =(\psi_i^h)=(\phi _i^h)^{-1}$.
By replacing $\Phi $ (resp.\ $\sigma $)
by $\Phi ^{-1}$ (resp.\ $\Phi \cdot \sigma $)
in \eqref{PhiXisigma} we obtain
$(\Phi ^{(1)})^{-1}\cdot X_i^{\Phi \cdot \sigma }
=\psi _i^hX_h^\sigma $. Hence
\[
\begin{array}
[c]{rl}
\left(
\Phi \cdot \nabla ^\sigma
\right) _{X^{\Phi \cdot \sigma }}
X_i^{\Phi \cdot \sigma }
= & \Phi ^{(1)}\cdot \Bigl(
\left(
\nabla ^\sigma
\right) _{(\Phi ^{(1)})^{-1}
\cdot X^{\Phi \cdot \sigma }}
\left(
(\Phi ^{(1)})^{-1}\cdot X_i^{\Phi \cdot \sigma }
\right)
\Bigr)
\smallskip\\
= & \Phi ^{(1)}\cdot
\left(
\nabla _{X^\sigma }^\sigma (\psi_i^h
X_h^\sigma )
\right)
\smallskip\\
= & \Phi ^{(1)}\cdot
\left\{
\left(
X^\sigma (\psi _i^h)
-\tfrac{1}{2}\psi _i^j
\tfrac{\partial F^h(\sigma )}
{\partial \dot{x}^j}
\right)
X_h^\sigma
\right\}
\smallskip\\
= & -\Phi ^{(1)}
\cdot \left\{
\left(
\psi _r^hX^\sigma (\phi _j^r)
+\tfrac{1}{2}
\tfrac{\partial F^h(\sigma )}
{\partial \dot{x}^j}
\right)
\psi _i^jX_h^\sigma
\right\}
\smallskip\\
= &
-\left\{
\left(
\psi_r^hX^\sigma (\phi _j^r)
+\tfrac{1}{2}
\tfrac{\partial F^h(\sigma )}
{\partial \dot{x}^j}
\right)
\psi _i^j
\right\}
\circ(\Phi ^{(1)})^{-1}
\!
\left(
\Phi ^{(1)}\cdot X_h^\sigma
\right)
\smallskip\\
\text{{\small by virtue of \eqref{PhiXisigma}}}
= &
-\left(
X^\sigma (\phi _j^a)\psi_i^j
\right)
\circ
\left(
\Phi ^{(1)}
\right) ^{-1}
X_a^{\Phi \cdot \sigma }\smallskip\\
& -\tfrac{1}{2}
\left(
\phi _h^a\circ \Phi ^{-1}
\right)
\left(
\tfrac{\partial F^h(\sigma )}
{\partial \dot{x}^j}\psi_i^j
\right)
\circ\left(
\Phi ^{(1)}
\right) ^{-1}
X_a^{\Phi \cdot \sigma }\smallskip\\
= & -\tfrac{1}{2}
\frac{\partial F^a(\Phi \cdot \sigma )}
{\partial \dot{x}^i}
X_a^{\Phi \cdot \sigma },
\end{array}
\]
as
\begin{equation}
\label{partialF}
\frac{\partial F^h(\Phi \cdot \sigma )}
{\partial \dot{x}^i}
=2\left(
X^\sigma (\phi _b^h)\psi _i^b
\right)
\circ \Phi ^{-1}
+\Bigl(
\phi _a^h\frac{\partial F^a(\sigma )}
{\partial \dot{x}^b}\psi _i^b
\Bigr)
\circ(\Phi ^{(1)})^{-1},
\end{equation}
which follows taking derivatives
with respect to $\dot{x}^i$
in the formula \eqref{FhPhi_sigma}
and taking \eqref{x_dot} into account.
Similarly,

\[
\begin{array}
[c]{rl}
\left(
\Phi \cdot \nabla ^\sigma
\right) _{X_j^{\Phi \cdot \sigma }}
X_i^{\Phi \cdot \sigma}
= & \!\! \Phi ^{(1)}\cdot
\Bigl(
\left(
\nabla ^\sigma
\right) _{\left( \Phi ^{(1)}\right) ^{-1}
\cdot X_j^{\Phi \cdot \sigma}}
\left(
(\Phi ^{(1)})^{-1}
\cdot X_i^{\Phi \cdot \sigma }
\right)
\Bigr)
\smallskip\\
= & \!\!
\Phi ^{(1)}
\cdot \left(
\nabla _{\psi _j^kX_k^\sigma }^\sigma
(\psi _i^hX_h^\sigma )
\right)
\smallskip\\
= & \!\!\Phi ^{(1)}
\cdot \left\{
\psi_j^k
\left(
X_k^\sigma (\psi _i^h)
-\tfrac{1}{2}\psi_i^a
\tfrac{\partial ^2F^h(\sigma )}
{\partial \dot{x}^a\partial \dot{x}^k}
\right)
X_h^\sigma
\right\}
\smallskip\\
= & \!\! \Phi ^{(1)}
\cdot \left\{
\psi _j^k
\left(
-\psi _r^hX_k^\sigma
\left(
\phi _a^r
\right)
-\tfrac{1}{2}
\tfrac{\partial ^2F^h(\sigma )}
{\partial \dot{x}^a\partial \dot{x}^k}
\right)
\psi _i^aX_h^\sigma
\right\}
\smallskip\\
= & \!\!
-\left\{
\psi _j^k
\left(
\psi_r^hX_k^\sigma
\left(
\phi _a^r
\right)
+\tfrac{1}{2}
\tfrac{\partial ^2F^h(\sigma )}
{\partial \dot{x}^a\partial \dot{x}^k}
\right)
\psi _i^a
\right\}
\!
\circ \! (\Phi ^{(1)})^{-1}
\left(
\Phi ^{(1)}\! \cdot \! X_h^\sigma
\right)
\smallskip\\
= & \!\!
-\left\{
\psi _j^k\left(
\tfrac{1}{2}\psi _r^h\phi _{ak}^r
+\tfrac{1}{2}
\tfrac{\partial ^2F^h(\sigma)}
{\partial \dot{x}^a\partial \dot{x}^k}
\right)
\psi _i^a\!
\right\}
\! \circ \! (\Phi ^{(1)})^{-1}
\!
\left(
\phi _h^s\! \circ \! \Phi ^{-1}
\right)
\! X_s^{\Phi \cdot \sigma }
\smallskip\\
= & \!\!
-\tfrac{1}{2}
\left\{
\phi _{ak}^{s}\psi _j^k\psi _i^a
+\tfrac{\partial ^2F^h(\sigma )}
{\partial \dot{x}^a\partial \dot{x}^k}
\psi_j^k\psi_i^a\phi _h^s
\right\}
\circ (\Phi ^{(1)})^{-1}
X_s^{\Phi \cdot \sigma }
\smallskip \\
= & \!\!
-\tfrac{1}{2}
\tfrac{\partial ^2F^k(\Phi \cdot \sigma )}
{\partial \dot{x}^i\partial \dot{x}^j}
X_k^{\Phi \cdot \sigma },
\end{array}
\]
as
\begin{equation}
\label{Partial2F}
\frac{\partial ^2F^h(\Phi \cdot \sigma )}
{\partial \dot{x}^i\partial \dot{x}^j}
=\Bigl(
\phi _{ab}^h\psi _i^a\psi _j^b+\phi _a^h
\frac{\partial ^2F^a(\sigma )}
{\partial \dot{x}^b\partial \dot{x}^c}
\psi _j^c\psi_i^b
\Bigr)
\circ(\Phi ^{(1)})^{-1},
\end{equation}
and also
\begin{align*}
\left(
\Phi \cdot \nabla ^\sigma
\right) _{\frac{\partial }{\partial \dot{x}^i}}
X_j^{\Phi \cdot\sigma }
& =\Phi ^{(1)}\cdot
\Bigl(
\left(
\nabla ^\sigma
\right)  _{(\Phi ^{(1)})^{-1}\cdot
\frac{\partial }{\partial \dot{x}^i}}
\Bigl(
(\Phi ^{(1)})^{-1}\cdot X_j^{\Phi \cdot \sigma}
\Bigr)
\Bigr) \\
& =\Phi ^{(1)}\cdot
\Bigl(
\nabla _{\psi_i^k\frac{\partial }
{\partial \dot{x}^k}}^\sigma
(\psi _j^hX_h^\sigma )
\Bigr) \\
& =\Phi ^{(1)}\cdot
\Bigl(
\psi _i^k\psi_j^h
\nabla _{\frac{\partial }
{\partial \dot{x}^k}}^\sigma
X_h^\sigma
\Bigr) \\
& =0.
\end{align*}
Taking \eqref{partialF} into account,
 we obtain
\[
\begin{array}
[c]{rl}
\left(
\Phi \cdot\nabla ^\sigma
\right) _{X^{\Phi \cdot \sigma }}
\frac{\partial }{\partial \dot{x}^i}
= & \!
\Phi ^{(1)}\cdot
\Bigl(
\left(
\nabla ^\sigma
\right) _{(\Phi ^{(1)})^{-1}
\cdot X^{\Phi \cdot \sigma }}
\left(
(\Phi ^{(1)})^{-1}
\cdot\frac{\partial }{\partial \dot{x}^i}
\right)
\Bigr)
\smallskip\\
= & \! \Phi ^{(1)}
\cdot \left(
\nabla _{X^\sigma }^\sigma
\left( \psi_i^h
\frac{\partial }{\partial \dot{x}^h}
\right)
\right)
\smallskip\\
= & \! \Phi ^{(1)}
\cdot\left\{
\left(
X^\sigma
\left(
\psi_i^h
\right)
-\tfrac{1}{2}\psi_i^j
\tfrac{\partial F^h(\sigma )}
{\partial \dot{x}^j}
\right)
\frac{\partial }{\partial \dot{x}^h}
\right\}
\smallskip\\
= & \!-\Phi ^{(1)}
\cdot \left\{
\left(
\psi _r^hX^\sigma
\left(
\phi _j^r
\right)
+\tfrac{1}{2}
\tfrac{\partial F^h(\sigma)}
{\partial \dot{x}^j}
\right)
\psi_i^j\frac{\partial }
{\partial \dot{x}^h}
\right\}
\smallskip\\
= & \!
-\left\{
\left(
\psi _r^hX^\sigma
\left(
\phi _j^r
\right)
+\tfrac{1}{2}
\tfrac{\partial F^h(\sigma )}
{\partial \dot{x}^j}
\right)
\psi _i^j
\right\}
\circ (\Phi ^{(1)})^{-1}
\left(
\Phi ^{(1)}
\cdot \frac{\partial }{\partial \dot{x}^h}
\right) \smallskip\\
\text{{\small by virtue of \eqref{Phixdot}}}
= & \!
-\left\{
\left(
X^\sigma
\left(
\phi _j^a
\right)
\psi_i^j
\right)
\circ (\Phi ^{(1)})^{-1}\right\}
\frac{\partial }
{\partial \dot{x}^a}
\smallskip\\
& \!
-\tfrac{1}{2}
\left\{
\left(
\phi _h^a\psi_i^j
\tfrac{\partial F^h(\sigma )}
{\partial \dot{x}^j}
\right)
\circ (\Phi ^{(1)})^{-1}
\right\}
\frac{\partial }{\partial \dot{x}^a}
\smallskip \\
= & \! -\tfrac{1}{2}
\frac{\partial F^a(\Phi \cdot \sigma )}
{\partial \dot{x}^i}
\frac{\partial }{\partial \dot{x}^a}.
\end{array}
\]
Similarly, taking \eqref{Partial2F}
into account, we obtain
\[
\begin{array}
[c]{rl}
\left(
\Phi \cdot \nabla ^\sigma
\right) _{X_j^{\Phi \cdot \sigma }}
\frac{\partial }{\partial \dot{x}^i}
= & \Phi ^{(1)}\cdot
\Bigl(
\left(
\nabla ^\sigma
\right) _{(\Phi ^{(1)})^{-1}\cdot X_j^{\Phi \cdot \sigma}}
\left(
(\Phi ^{(1)})^{-1}
\cdot \frac{\partial }{\partial \dot{x}^i}
\right)
\Bigr)
\smallskip\\
= & \Phi ^{(1)}\cdot
\Bigl(
\nabla _{\psi_j^kX_k^\sigma }^\sigma
\left(  \psi_i^h\frac
{\partial }{\partial \dot{x}^h}
\right)
\Bigr) \smallskip\\
= & \Phi ^{(1)}\cdot
\left\{
\psi_j^k
\left(
\frac{\partial\psi_i^h}{\partial \dot{x}^k}
-\tfrac{1}{2}\psi _i^a
\tfrac{\partial^2F^h(\sigma )}
{\partial \dot{x}^a\partial \dot{x}^k}
\right)
\frac{\partial }{\partial \dot{x}^h}
\right\}
\smallskip\\
= & \Phi ^{(1)}\cdot
\left\{
\psi _j^k
\left(
-\psi _r^h\frac{\partial\phi _a^r}
{\partial \dot{x}^k}
-\tfrac{1}{2}
\tfrac{\partial ^2F^h(\sigma)}
{\partial \dot{x}^a\partial \dot{x}^k}
\right)
\psi _i^a
\frac{\partial }{\partial \dot{x}^h}
\right\}
\smallskip\\
= & -\left\{
\psi _j^k
\left(
\psi _r^h\frac{\partial\phi _a^r}
{\partial \dot{x}^k}
+\tfrac{1}{2}
\tfrac{\partial^2F^h(\sigma )}
{\partial \dot{x}^a\partial \dot{x}^k}
\right)
\psi _i^a
\right\}
\circ (\Phi ^{(1)})^{-1}
\left(
\Phi ^{(1)}
\cdot \frac{\partial }{\partial \dot{x}^h}
\right)
\smallskip\\
= & -\tfrac{1}{2}
\tfrac{\partial^2F^k(\Phi \cdot \sigma )}
{\partial \dot{x}^i\partial \dot{x}^j}
\frac{\partial }{\partial \dot{x}^k}.
\end{array}
\]

Finally,
\begin{align*}
\left(
\Phi \cdot \nabla ^\sigma
\right) _{\frac{\partial }{\partial \dot{x}^i}}
\frac{\partial }{\partial \dot{x}^j}
& =\Phi ^{(1)}\cdot \Bigl(
\nabla _{\psi _i^k
\frac{\partial }{\partial \dot{x}^k}}^\sigma
\Bigl(
\psi _j^h\frac{\partial }{\partial \dot{x}^h}
\Bigr)
\Bigr)
\\
& =\Phi ^{(1)}\cdot \Bigl(
\psi _i^k\psi _j^h
\nabla _{\frac{\partial }{\partial \dot{x}^k}}^\sigma
\frac{\partial }{\partial \dot{x}^h}
\Bigr)
\\
& =0.
\end{align*}
\end{proof}

\begin{corollary}
\label{CorollaryFunctoriality1}
If $\Phi \in \mathrm{Aut}^v(p)$
and $X,Y,Z\in \mathfrak{X}(M^1)$
are vector fields $\Phi ^{(1)}$-related to
$X^\prime ,Y^\prime ,Z^\prime
\in \mathfrak{X}(M^1)$ respectively, then
$T^\sigma (X,Y)$ (resp.\ $R^\sigma (X,Y)Z$)
is $\Phi ^{(1)}$-related to
$T^{\Phi \cdot \sigma }(X^\prime ,Y^\prime )$
(resp.\ $R^{\Phi \cdot \sigma }
(X^\prime ,Y^\prime )Z^\prime $).
\end{corollary}

\begin{proof}
It follows from
\cite[VI, Proposition 1.2, (2), (3)]{KN}.
\end{proof}

\section{Differential invariants}
\label{section_invariants}

Let $(p^{21})^r\colon J^r(p^{21})\to M^1$
be the $r$-jet bundle of the submersion
$p^{21}\colon M^2\to M^1$.
According to \eqref{Phi^r}, every
$\Phi \in \mathrm{Aut}^v(p)$ induces in particular
diffeomorphisms $\Phi ^{(r)}\colon M^r\to M^r$,
$r=1,2$, such that $p^{21}\circ \Phi ^{(2)}
=\Phi ^{(1)}\circ p^{21}$. Hence, for every
$r\geq 0$, the pair $\Phi ^{(2)},\Phi ^{(1)}$
induces a transformation $(\Phi ^{(2)})^{(r)}
\colon J^r(p^{21})\to J^r(p^{21})$ given by
$(\Phi ^{(2)})^{(r)}(j_\xi ^r\sigma)
=j_{\Phi ^{(1)}(\xi )}^r(\Phi ^{(2)}
\circ \sigma \circ (\Phi ^{(1)})^{-1})$.
Let $\mathcal{U}\subseteq J^r(p^{21})$
be an open subset invariant under
all these transformations. A smooth function
$\mathcal{I}\colon \mathcal{U}\to \mathbb{R}$
is said to be a differential invariant of order
$r$ with respect to the group $\mathrm{Aut}^v(p)$
if
$\mathcal{I}\circ (\Phi ^{(2)})^{(r)}=\mathcal{I}$
for all $\Phi \in\mathrm{Aut}^v(p)$. If we set
$I(\sigma ,\xi )=\mathcal{I}(j_\xi ^r\sigma )$,
$\xi \in M^1$, for a given SODE
$\sigma $ on $M$, then the invariance condition
above reads as follows: $I\left( \Phi \cdot \sigma ,
\Phi ^{(1)}(\xi )\right) =I(\sigma ,\xi )$,
$\forall \xi \in M^1$,
$\forall \Phi \in \mathrm{Aut}^v(p)$, thus leading
one to the naive definition of an invariant,
as being a function depending on the components
of $\sigma $ and its partial derivatives up to
a certain order, which remains unchanged
under arbitrary changes of coordinates.

If $\Phi _t\in \mathrm{Aut}^v(p)$ is the flow
of a $p$-vertical vector field
$X\in \mathfrak{X}(M^0)$, then
$\Phi _t^{(2)}$ is the flow of a $p^2$-vertical
vector field $X^{(2)}\in \mathfrak{X}(M^2)$ and
$(\Phi _t^{(2)})^{(r)}$ is the flow of a vector
field $(X^{(2)})^{(r)}$ on $J^r(p^{21})$.
Every differential invariant of order $r$
is a first integral of the distribution
$\mathcal{D}^{(r)}$ on $J^r(p^{21})$ spanned
by all the jet prolongations $(X^{(2)})^{(r)}$
of $p$-vertical vector fields.

We claim that the only first integrals
of the distributions $\mathcal{D}^{(0)}$ and
$\mathcal{D}^{(1)}$ are
$(p^2)^\ast C^\infty (\mathbb{R})$ and
$((p^{21})^1)^\ast (p^1)^\ast C^\infty (\mathbb{R})$,
respectively. In fact, from the general formulas
of jet prolongation of vector fields (e.g.,
see \cite{Kumpera}, \cite{MM}), we obtain
\begin{align}
X & =u^i\frac{\partial }{\partial x^i},
\quad u^i\in C^\infty(M^0), \nonumber\\
X^{(2)} & =u^i\frac{\partial }{\partial x^i}
+v^i\frac{\partial }{\partial \dot{x}^i}
+w^i\frac{\partial }{\partial \ddot{x}^i},
\nonumber\\
v^i & =\frac{\partial u^i}{\partial t}
+\frac{\partial u^i}{\partial x^h}
\dot{x}^h,\label{v}\\
w^i & =\frac{\partial ^2u^i}{\partial t^2}
+2\frac{\partial ^2u^i}{\partial t\partial x^h}
\dot{x}^h
+\frac{\partial ^2u^i}{\partial x^h\partial x^k}
\dot{x}^h\dot{x}^k
+\frac{\partial u^i}{\partial x^h}\ddot{x}^h.
\label{w}
\end{align}
As the values of $u^i$, $\partial u^i/\partial t$,
$\partial u^i/\partial x^h$, and
$\partial ^2u^i/\partial t^2$ can arbitrarily
be taken at a given point $j_t^2\gamma \in M^2$,
we conclude that the distribution on $M^2$ generated
by all the vector fields $X^{(2)}$ span the
subbundle of $p^2$-vertical tangent vectors.
Hence, the only differential invariants
of order $0$ are the functions in
$(p^2)^\ast C^\infty (\mathbb{R})$.

By again computing the jet prolongation,
we obtain
\[
(X^{(2)})^{(1)}
=u^i\tfrac{\partial }{\partial x^i}
+v^i\tfrac{\partial }{\partial \dot{x}^i}
+w^i\tfrac{\partial }{\partial \ddot{x}^i}
+w_t^i\tfrac{\partial }{\partial \ddot{x}_t^i}
+w_a^i\tfrac{\partial }{\partial \ddot{x}_a^i}
+w_{\dot{a}}^i
\tfrac{\partial }{\partial \ddot{x}_{\dot{a}}^i},
\]
\begin{align}
w_t^i & =\tfrac{\partial ^3u^i}
{\partial t^3}
+2\tfrac{\partial ^3u^i}
{\partial t^2\partial x^h}
\dot{x}^h
+\tfrac{\partial ^3u^i}
{\partial t\partial x^h\partial x^k}
\dot{x}^h\dot{x}^k
+\tfrac{\partial^2u^i}
{\partial t\partial x^h}
\ddot{x}^h
\nonumber \\
& +\tfrac{\partial u^i}
{\partial x^h}
\ddot{x}_t^h
-\tfrac{\partial u^a}
{\partial t}\ddot{x}_a^i
-\tfrac{\partial^2u^a}
{\partial t^2}
\ddot{x}_{\dot{a}}^i
-\tfrac{\partial^2u^a}
{\partial t\partial x^h}
\dot{x}^h\ddot{x}_{\dot{a}}^i,
\nonumber\\
w_a^i & =\tfrac{\partial ^3u^i}
{\partial t^2\partial x^a}
+2\tfrac{\partial ^3u^i}
{\partial t\partial x^a\partial x^h}
\dot{x}^h
+\tfrac{\partial ^3u^i}
{\partial x^a\partial x^h\partial x^k}
\dot{x}^h\dot{x}^k
+\tfrac{\partial^2u^i}
{\partial x^a\partial x^h}
\ddot{x}^h\label{w_a}\\
& +\tfrac{\partial u^i}
{\partial x^h}
\ddot{x}_a^h
-\tfrac{\partial u^b}
{\partial x^a}\ddot{x}_b^i
-\tfrac{\partial^2u^b}
{\partial t\partial x^a}
\ddot{x}_{\dot{b}}^i
-\tfrac{\partial^2u^b}
{\partial x^a\partial x^h}
\dot{x}^h\ddot{x}_{\dot{b}}^i,
\nonumber\\
w_{\dot{a}}^i
& =2\tfrac{\partial^2u^i}
{\partial t\partial x^a}
+2\tfrac{\partial^2u^i}
{\partial x^a\partial x^h}
\dot{x}^h
+\tfrac{\partial u^i}{\partial x^h}
\ddot{x}_{\dot{a}}^h
-\tfrac{\partial
u^r}{\partial x^a}\ddot{x}_{\dot{r}}^i.
\label{w_apunto}
\end{align}
By collecting the derivatives
of the functions $u^i$ in the expression
above for $(X^{(2)})^{(1)}$ we conclude
\begin{multline*}
(X^{(2)})^{(1)}
=u^r\chi_{00}^r
+\tfrac{\partial u^r}{\partial t}
\chi_t^r
+\tfrac{\partial u^r}{\partial x^a}
\chi_a^r
+\tfrac{\partial^2u^r}{\partial t^2}
\chi_{tt}^r
+\tfrac{\partial^2u^r}
{\partial t\partial x^a}
\chi_{ta}^r
+\sum _{a\leq b}
\tfrac{\partial ^2u^r}
{\partial x^a\partial x^b}
\chi_{a\leq b}^r\\
+\tfrac{\partial ^3u^r}
{\partial t^3}
\chi_{ttt}^r
+\tfrac{\partial ^3u^r}
{\partial t^2\partial x^a}
\chi_{tta}^r+\sum _{a\leq b}
\tfrac{\partial ^3u^r}
{\partial t\partial x^a\partial x^b}
\chi_{t,a\leq b}^r
+\sum _{a\leq b\leq c}
\tfrac{\partial ^3u^r}
{\partial x^a\partial
x^b\partial x^c}
\chi_{a\leq b\leq c}^r,
\end{multline*}
where
\[
\begin{array}
[c]{l}
\chi_{00}^r
=\frac{\partial }{\partial x^r},\\
\chi_t^r
=\frac{\partial }{\partial\dot{x}^r}
-\ddot{x}_r^i
\frac{\partial }{\partial\ddot{x}_t^i},\\
\chi_a^r=\dot{x}^a
\frac{\partial }{\partial\dot{x}^r}
+\ddot{x}^a
\frac{\partial }{\partial\ddot{x}^r}
+\ddot{x}_t^a
\frac{\partial }{\partial\ddot{x}_t^r}
+\ddot{x}_{h}^a
\frac{\partial }{\partial\ddot{x}_{h}^r}
-\ddot{x}_r^i
\frac{\partial }{\partial\ddot{x}_a^i}
-\ddot{x}_{\dot{r}}^i
\frac{\partial }
{\partial\ddot{x}_{\dot{a}}^i}
+\ddot{x}_{\dot{b}}^a
\frac{\partial }
{\partial\ddot{x}_{\dot{b}}^r},\\
\chi_{tt}^r
=\frac{\partial }{\partial\ddot{x}^r}
-\ddot{x}_{\dot{r}}^i
\frac{\partial }{\partial\ddot{x}_t^i},\\
\chi_{ta}^r=\dot{x}^a
\left(
2\frac{\partial }{\partial\ddot{x}^r}
-\ddot{x}_{\dot{r}}^i
\frac{\partial }{\partial\ddot{x}_t^i}
\right)
+\ddot{x}^a
\frac{\partial }{\partial\ddot{x}_t^r}
-\ddot{x}_{\dot{r}}^i
\frac{\partial }{\partial\ddot{x}_a^i}
+2\frac{\partial }{\partial
\ddot{x}_{\dot{a}}^r},
\end{array}
\]
\begin{align*}
\chi_{a\leq b}^r
& =\tfrac{1}{1+\delta_{ab}}
\left\{
2\dot{x}^a\dot{x}^b
\tfrac{\partial }{\partial\ddot{x}^r}
+\ddot{x}^b
\tfrac{\partial }{\partial\ddot{x}_a^r}
+\ddot{x}^a
\tfrac{\partial }{\partial\ddot{x}_b^r}
\dot{x}^b\ddot{x}_{\dot{r}}^i
\tfrac{\partial }{\partial\ddot{x}_a^i}
-\dot{x}^a\ddot{x}_{\dot{r}}^i
\tfrac{\partial }{\partial \ddot{x}_b^i}
\right.  \\
& \left.
+2\dot{x}^b
\tfrac{\partial }
{\partial\ddot{x}_{\dot{a}}^r}
+2\dot{x}^a
\tfrac{\partial }
{\partial\ddot{x}_{\dot{b}}^r}
\right\} ,
\end{align*}
\[
\begin{array}
[c]{l}
\chi_{ttt}^r
=\frac{\partial }{\partial\ddot{x}_t^r},\\
\chi_{tta}^r
=2\dot{x}^a
\frac{\partial }{\partial\ddot{x}_t^r}
+\frac{\partial }{\partial\ddot{x}_a^r},\\
\chi_{t,a\leq b}^r
=\tfrac{2}{1+\delta_{ab}}
\left\{
\dot{x}^a\dot{x}^b
\frac{\partial }{\partial\ddot{x}_t^r}
+\dot{x}^b
\frac{\partial }{\partial\ddot{x}_a^r}
+\dot{x}^a
\frac{\partial }{\partial\ddot{x}_b^r}
\right\} ,\\
\chi_{a\leq b\leq c}^r
=\tfrac{2}{(1+\delta _{ab}+\delta _{bc})!}
\left\{
\dot{x}^b\dot{x}^c
\frac{\partial }{\partial\ddot{x}_a^r}
+\dot{x}^a\dot{x}^c
\frac{\partial }{\partial\ddot{x}_b^r}
+\dot{x}^a\dot{x}^b
\frac{\partial }{\partial\ddot{x}_{c}^r}
\right\} ,
\end{array}
\]
and $(t,x^i,\dot{x}^i,\ddot{x}^i,\ddot{x}_t^i,
\ddot{x}_a^i,\ddot{x}_{\dot{a}}^i)$ is the induced
coordinate system on $J^1(p^{21})$, namely
\[
\ddot{x}_t^i
\left(  j_\xi ^1\sigma
\right)
=\frac{\partial F^i}{\partial t}(\xi ),
\quad
\ddot{x}_a^i
\left(
j_\xi ^1\sigma
\right)
=\frac{\partial F^i}{\partial x^a}(\xi ),
\quad
\ddot{x}_{\dot{a}}^i
\left(
j_\xi ^1\sigma
\right)
=\frac{\partial F^i}{\partial\dot{x}^a}(\xi ).
\]
Accordingly, $\chi_{00}^r$, $\chi_t^r$,
$\chi_a^r$, $\chi_{tt}^r$, $\chi_{ta}^r$,
$\chi_{a\leq b}^r$, $\chi_{ttt}^r$, $
\chi_{tta}^r$, $\chi_{ta\leq b}^r$,
$\chi_{a\leq b\leq c}^r$ constitute
a system of generators for the distribution
$\mathcal{D}^{(1)}$. From the expressions
above for $\chi_{00}^r$, $\chi_{ttt}^r$,
$\chi_t^r$, $\chi_{tt}^r$, $\chi_{tta}^r$,
$\chi_{ta}^r$ we obtain
\[
\begin{array}
[c]{lll}
\frac{\partial }{\partial x^r}
=\chi_{00}^r,
& \frac{\partial }{\partial \dot{x}^r}
=\chi_t^r+\ddot{x}_r^i\chi_{ttt}^i,
& \frac{\partial }{\partial\ddot{x}_t^r}
=\chi_{ttt}^r,\\
\frac{\partial }{\partial\ddot{x}^r}
=\chi_{tt}^r
+\ddot{x}_{\dot{r}}^i\chi_{ttt}^i,
& \frac{\partial }{\partial\ddot{x}_a^r}
=\chi_{tta}^r-2\dot{x}^a\chi_{ttt}^r,
&
\end{array}
\]
\[
\begin{array}
[c]{l}
2\frac{\partial }{\partial\ddot{x}_{\dot{a}}^r}
=\chi_{ta}^r-2\dot{x}^a\chi_{tt}^r
-3\dot{x}^a\ddot{x}_{\dot{r}}^i\chi_{ttt}^i
-\ddot{x}^a\chi_{ttt}^r
+\ddot{x}_{\dot{r}}^i\chi_{tta}^i.
\end{array}
\]
As in the previous case, the first-order
differential invariants are none other than
the functions in
$((p^{21})^1)^\ast (p^1)^\ast C^\infty (\mathbb{R})$.

\begin{lemma}
\label{rank}The rank of $\mathcal{D}^{(2)}$ is $11$
if $\dim M=1$ and $\tfrac{1}{2}n(3n^2+11n+10)$
if $\dim M=n>1$.
\end{lemma}

\begin{proof}
By computing the second jet prolongation,
the following formulas are obtained:
\begin{align*}
(X^{(2)})^{(2)}
& =u^i\tfrac{\partial }{\partial x^i}
+v^i
\tfrac{\partial }{\partial\dot{x}^i}
+w^i
\tfrac{\partial }{\partial\ddot{x}^i}
+w_t^i
\tfrac{\partial }{\partial\ddot{x}_t^i}
+w_a^i
\tfrac{\partial }{\partial\ddot{x}_a^i}
+w_{\dot{a}}^i
\tfrac{\partial }{\partial\ddot{x}_{\dot{a}}^i}
+w_{tt}^i
\tfrac{\partial }{\partial\ddot{x}_{tt}^i}\\
& +w_{ta}^i
\tfrac{\partial }{\partial\ddot{x}_{ta}^i}
+w_{t\dot{a}}^i
\tfrac{\partial }{\partial\ddot{x}_{t\dot{a}}^i}
+w_{a\leq b}^i
\tfrac{\partial }{\partial\ddot{x}_{a\leq b}^i}
+w_{a\dot{b}}^i
\tfrac{\partial }{\partial\ddot{x}_{a\dot{b}}^i}
+w_{\dot{a}\leq\dot{b}}^i
\tfrac{\partial }
{\partial\ddot{x}_{\dot{a}\leq\dot{b}}^i},
\end{align*}
\begin{align}
w_{tt}^i
& =\tfrac{\partial ^4u^i}{\partial t^4}
+2\tfrac{\partial ^{4}u^i}
{\partial t^3\partial x^h}
\dot{x}^h
+\tfrac{\partial ^4u^i}
{\partial t^2\partial x^h\partial x^k}
\dot{x}^h\dot{x}^k
-\tfrac{\partial ^3u^r}{\partial t^3}
\ddot{x}_{\dot{r}}^i
+\tfrac{\partial ^3u^i}
{\partial t^2\partial x^h}
\ddot{x}^h
-\tfrac{\partial ^3u^r}
{\partial t^2\partial x^h}
\dot{x}^h\ddot{x}_{\dot{r}}^i
\nonumber\\
& -\tfrac{\partial^2u^r}{\partial t^2}
\ddot{x}_r^i
-2\tfrac{\partial^2u^r}
{\partial t^2}
\ddot{x}_{t\dot{r}}^i
+2\tfrac{\partial^2u^i}
{\partial t\partial x^h}
\ddot{x}_t^h
-2\tfrac{\partial^2u^r}
{\partial t\partial x^h}
\dot{x}^h\ddot{x}_{t\dot{r}}^i
-2\tfrac{\partial u^r}{\partial t}
\ddot{x}_{tr}^i
+\tfrac{\partial u^i}{\partial x^h}
\ddot{x}_{tt}^h,
\nonumber
\end{align}
\begin{align}
w_{ta}^i
&  =\tfrac{\partial ^4u^i}
{\partial t^3\partial x^a}
+2\tfrac{\partial ^4u^i}
{\partial t^2\partial x^a\partial x^h}
\dot{x}^h
+\tfrac{\partial ^4u^i}
{\partial t\partial x^h\partial x^k\partial x^a}
\dot{x}^h\dot{x}^k
-\tfrac{\partial ^3u^r}
{\partial t^2\partial x^a}
\ddot{x}_{\dot{r}}^i
+\tfrac{\partial ^3u^i}
{\partial t\partial x^a\partial x^h}
\ddot{x}^h\nonumber\\
& -\tfrac{\partial ^3u^r}
{\partial t\partial x^h\partial x^a}
\dot{x}^h\ddot{x}_{\dot{r}}^i
-\tfrac{\partial^2u^r}{\partial t^2}
\ddot{x}_{a\dot{r}}^i
-\tfrac{\partial^2u^r}
{\partial t\partial x^a}\ddot{x}_r^i
+\tfrac{\partial^2u^i}
{\partial t\partial x^h}
\ddot{x}_a^h
-\tfrac{\partial^2u^r}
{\partial t\partial x^a}
\ddot{x}_{t\dot{r}}^i
\nonumber\\
& -\tfrac{\partial^2u^r}
{\partial t\partial x^h}
\dot{x}^h\ddot{x}_{a\dot{r}}^i
+\tfrac{\partial^2u^i}
{\partial x^h\partial x^a}
\ddot{x}_t^h
-\tfrac{\partial^2u^r}
{\partial x^a\partial x^h}
\dot{x}^h\ddot{x}_{t\dot{r}}^i
-\tfrac{\partial u^r}{\partial t}
\ddot{x}_{ar}^i
-\tfrac{\partial u^r}{\partial x^a}
\ddot{x}_{tr}^i
+\tfrac{\partial u^i}{\partial x^h}
\ddot{x}_{ta}^h,\nonumber
\end{align}
\begin{align}
w_{t\dot{a}}^i
& =2\tfrac{\partial ^3u^i}
{\partial t^2\partial x^a}
+2\tfrac{\partial ^3u^i}
{\partial t\partial x^a\partial x^h}
\dot{x}^h
-\tfrac{\partial^2u^r}{\partial t^2}
\ddot{x}_{\dot{a}\dot{r}}^i
+\tfrac{\partial^2u^i}
{\partial t\partial x^h}
\ddot{x}_{\dot{a}}^h
-\tfrac{\partial^2u^r}
{\partial t\partial x^a}
\ddot{x}_{\dot{r}}^i
\label{w_t_apunto}\\
& -\tfrac{\partial^2u^r}
{\partial t\partial x^h}\dot{x}^h
\ddot{x}_{\dot{a}\dot{r}}^i
-\tfrac{\partial u^r}{\partial t}
\ddot{x}_{\dot{a}r}^i
+\tfrac{\partial u^i}{\partial x^h}
\ddot{x}_{t\dot{a}}^h
-\tfrac{\partial u^r}{\partial x^a}
\ddot{x}_{t\dot{r}}^i,
\nonumber
\end{align}
\begin{align}
w_{a\leq b}^i
& =\tfrac{\partial ^4u^i}
{\partial t^2\partial x^a\partial x^b}
+2\tfrac{\partial ^4u^i}
{\partial t\partial x^a\partial x^h\partial x^b}
\dot{x}^h
+\tfrac{\partial ^4u^i}
{\partial x^a\partial x^b\partial x^h\partial x^k}
\dot{x}^h\dot{x}^k
\tfrac{\partial ^3u^r}
{\partial t\partial x^a\partial x^b}
\ddot{x}_{\dot{r}}^i
\nonumber \\
& +\tfrac{\partial ^3u^i}
{\partial x^a\partial x^b\partial x^h}
\ddot{x}^h-\tfrac{\partial ^3u^r}
{\partial x^a\partial x^b\partial x^h}
\dot{x}^h\ddot{x}_{\dot{r}}^i
-\tfrac{\partial^2u^r}
{\partial t\partial x^a}
\ddot{x}_{\dot{r}b}^i
-\tfrac{\partial^2u^r}
{\partial t\partial x^b}
\ddot{x}_{a\dot{r}}^i
\nonumber\\
& +\tfrac{\partial^2u^i}
{\partial x^h\partial x^b}
\ddot{x}_a^h
+\tfrac{\partial^2u^i}
{\partial x^a\partial x^h}
\ddot{x}_b^h
-\tfrac{\partial^2u^r}
{\partial x^a\partial x^b}
\ddot{x}_r^i
-\tfrac{\partial^2u^r}
{\partial x^a\partial x^h}
\dot{x}^h\ddot{x}_{\dot{r}b}^i
\nonumber\\
& -\tfrac{\partial^2u^r}
{\partial x^b\partial x^h}
\dot{x}^h\ddot{x}_{a\dot{r}}^i
-\tfrac{\partial u^r}
{\partial x^b}
\ddot{x}_{ar}^i
-\tfrac{\partial u^r}{\partial x^a}
\ddot{x}_{rb}^i
+\tfrac{\partial u^i}{\partial x^h}
\ddot{x}_{a\leq b}^h,
\nonumber
\end{align}
\begin{align}
w_{a\dot{b}}^i
& =2\tfrac{\partial ^3u^i}
{\partial t\partial x^a\partial x^b}
+2\tfrac{\partial ^3u^i}
{\partial x^a\partial x^b\partial x^k}
\dot{x}^k
-\tfrac{\partial^2u^r}
{\partial t\partial x^a}
\ddot{x}_{\dot{r}\dot{b}}^i
-\tfrac{\partial^2u^r}
{\partial x^a\partial x^b}
\ddot{x}_{\dot{r}}^i
+\tfrac{\partial^2u^i}
{\partial x^a\partial x^h}
\ddot{x}_{\dot{b}}^h\label{w_a_bpunto}\\
& -\tfrac{\partial^2u^r}
{\partial x^a\partial x^h}
\dot{x}^h\ddot{x}_{\dot{r}\dot{b}}^i
+\tfrac{\partial u^i}{\partial x^h}
\ddot{x}_{a\dot{b}}^h
-\tfrac{\partial u^r}{\partial x^a}
\ddot{x}_{r\dot{b}}^i
-\tfrac{\partial u^r}{\partial x^b}
\ddot{x}_{a\dot{r}}^i,
\nonumber
\end{align}
\begin{equation}
w_{\dot{a}\leq\dot{b}}^i
=2\tfrac{\partial^2u^i}
{\partial x^a\partial x^b}
+\tfrac{\partial u^i}{\partial x^h}
\ddot{x}_{\dot{a}\leq\dot{b}}^h
-\tfrac{\partial u^r}{\partial x^a}
\ddot{x}_{\dot{b}\dot{r}}^i
-\tfrac{\partial u^r}{\partial x^b}
\ddot{x}_{\dot{a}\dot{r}}^i,
\label{w_apunto_bpunto}
\end{equation}
where we assume $\ddot{x}_{ba}^i
=\ddot{x}_{a\leq b}^i$ for $a\leq b$
and $(t$, $x^i$, $\dot{x}^i$,
$\ddot{x}^i$, $\ddot{x}_t^i$,
$\ddot{x}_a^i$, $\ddot{x}_{\dot{a}}^i$,
$\ddot{x}_{tt}^i$, $\ddot{x}_{ta}^i$,
$\ddot{x}_{t\dot{a}}^i$,
$\ddot{x}_{a\leq b}^i$,
$\ddot{x}_{a\dot{b}}^i$,
$\ddot{x}_{\dot{a}\leq\dot{b}}^i)$
is the induced coordinate system
on $J^2(p^{21})$. Hence
\begin{multline*}
(X^{(2)})^{(2)}
=u^r\varkappa_{00}^r
+\tfrac{\partial u^r}{\partial t}
\varkappa_t^r
+\tfrac{\partial u^r}{\partial x^a}
\varkappa_a^r
+\tfrac{\partial^2u^r}{\partial t^2}
\varkappa_{tt}^r
+\tfrac{\partial^2u^r}
{\partial t\partial x^a}
\varkappa_{ta}^r
+\sum _{a\leq b}
\tfrac{\partial^2u^r}
{\partial x^a\partial x^b}
\varkappa_{a\leq b}^r\\
+\tfrac{\partial ^3u^r}{\partial t^3}
\varkappa_{ttt}^r
+\tfrac{\partial ^3u^r}
{\partial t^2\partial x^a}
\varkappa_{tta}^r
+\sum _{a\leq b}\tfrac{\partial ^3u^r}
{\partial t\partial x^a\partial x^b}
\varkappa_{t,a\leq b}^r\\
+\sum _{a\leq b\leq c}
\tfrac{\partial ^3u^r}
{\partial x^a\partial x^b\partial x^c}
\varkappa_{a\leq b\leq c}^r
+\tfrac{\partial ^4u^r}{\partial t^4}
\varkappa_{tttt}^r
+\tfrac{\partial ^4u^r}
{\partial t^3\partial x^a}
\varkappa _{ttta}^r
+\sum _{a\leq b}
\tfrac{\partial ^4u^r}
{\partial t^2\partial x^a\partial x^b}
\varkappa _{tt,a\leq b}^r\\
+\sum _{a\leq b\leq c}
\tfrac{\partial ^4u^r}
{\partial t\partial x^a\partial x^b\partial x^c}
\varkappa _{t,a\leq b\leq c}^r
+\sum _{a\leq b\leq c\leq d}
\tfrac{\partial ^4u^r}
{\partial x^a\partial x^b\partial x^c\partial x^d}
\varkappa_{a\leq b\leq c\leq d}^r,
\end{multline*}
where
\[
\begin{array}
[c]{rl}
\varkappa_{00}^r
= & \!\!\!
\frac{\partial }{\partial x^r},\\
\varkappa_t^r
= & \!\!\!
\frac{\partial }{\partial\dot{x}^r}
-\ddot{x}_r^i
\frac{\partial }{\partial\ddot{x}_t^i}
-2\ddot{x}_{tr}^i
\frac{\partial }{\partial\ddot{x}_{tt}^i}
-\ddot{x}_{ar}^i
\frac{\partial }{\partial\ddot{x}_{ta}^i}
-\ddot{x}_{r\dot{a}}^i
\frac{\partial }
{\partial\ddot{x}_{t\dot{a}}^i},
\end{array}
\]
\begin{multline*}
\varkappa_a^r
=\dot{x}^a
\tfrac{\partial }{\partial\dot{x}^r}
+\ddot{x}^a
\tfrac{\partial }{\partial\ddot{x}^r}
+\ddot{x}_t^a
\tfrac{\partial }{\partial\ddot{x}_t^r}
+\ddot{x}_b^a
\tfrac{\partial }{\partial\ddot{x}_b^r}
+\ddot{x}_{\dot{b}}^a
\tfrac{\partial }
{\partial\ddot{x}_{\dot{b}}^r}
-\ddot{x}_r^i
\tfrac{\partial }{\partial\ddot{x}_a^i}
-\ddot{x}_{\dot{r}}^i
\tfrac{\partial }
{\partial\ddot{x}_{\dot{a}}^i}
+\ddot{x}_{tt}^a
\tfrac{\partial }
{\partial\ddot{x}_{tt}^r}\\
-\ddot{x}_{tr}^i
\tfrac{\partial }
{\partial\ddot{x}_{ta}^i}
+\ddot{x}_{tb}^a
\tfrac{\partial }
{\partial\ddot{x}_{tb}^r}
+\ddot{x}_{t\dot{b}}^a
\tfrac{\partial }
{\partial\ddot{x}_{t\dot{b}}^r}
-\ddot{x}_{t\dot{r}}^i
\tfrac{\partial }
{\partial\ddot{x}_{t\dot{a}}^i}
-\sum _{c\leq a}\ddot{x}_{cr}^i
\tfrac{\partial }
{\partial\ddot{x}_{c\leq a}^i}
-\sum _{a\leq b}\ddot{x}_{rb}^i
\tfrac{\partial }
{\partial \ddot{x}_{a\leq b}^i}\\
+\sum _{c\leq b}\ddot{x}_{cb}^a
\tfrac{\partial }
{\partial \ddot{x}_{c\leq b}^r}
+\ddot{x}_{c\dot{b}}^a
\tfrac{\partial }
{\partial \ddot{x}_{c\dot{b}}^r}
-\ddot{x}_{r\dot{b}}^i
\tfrac{\partial }
{\partial \ddot{x}_{a\dot{b}}^i}
-\ddot{x}_{c\dot{r}}^i
\tfrac{\partial }
{\partial \ddot{x}_{c\dot{a}}^i}
+\sum _{c\leq b}
\ddot{x}_{\dot{c}\leq\dot{b}}^a
\tfrac{\partial }
{\partial\ddot{x}_{\dot{c}\leq \dot{b}}^r}\\
-\sum _{a\leq b}
\ddot{x}_{\dot{b}\dot{r}}^i
\tfrac{\partial }
{\partial\ddot{x}_{\dot{a}\leq \dot{b}}^i}
-\sum _{b\leq a}
\ddot{x}_{\dot{b}\dot{r}}^i
\tfrac{\partial }
{\partial\ddot{x}_{\dot{b}\leq \dot{a}}^i},
\end{multline*}
\[
\begin{array}
[c]{l}
\varkappa _{tt}^r
=\frac{\partial }{\partial\ddot{x}^r}
-\ddot{x}_{\dot{r}}^i
\frac{\partial }{\partial\ddot{x}_t^i}
-\ddot{x}_r^i
\frac{\partial }{\partial\ddot{x}_{tt}^i}
-2\ddot{x}_{t\dot{r}}^i
\frac{\partial }{\partial\ddot{x}_{tt}^i}
-\ddot{x}_{a\dot{r}}^i
\frac{\partial }{\partial\ddot{x}_{ta}^i}
-\ddot{x}_{\dot{r}\dot{a}}^i
\frac{\partial }
{\partial \ddot{x}_{t\dot{a}}^i},
\end{array}
\]
\begin{multline*}
\varkappa_{ta}^r=\dot{x}^a
\left(
2\tfrac{\partial }{\partial\ddot{x}^r}
-\ddot{x}_{\dot{r}}^i
\tfrac{\partial }{\partial\ddot{x}_t^i}
\right)
+\ddot{x}^a
\tfrac{\partial }{\partial\ddot{x}_t^r}
-\ddot{x}_{\dot{r}}^i
\tfrac{\partial }{\partial\ddot{x}_a^i}
+2\tfrac{\partial }
{\partial \ddot{x}_{\dot{a}}^r}
+2\ddot{x}_t^a
\tfrac{\partial }{\partial\ddot
{x}_{tt}^r}-2\dot{x}^a\ddot{x}_{t\dot{r}}^i
\tfrac{\partial }{\partial\ddot{x}_{tt}^i}\\
-\ddot{x}_r^i
\tfrac{\partial }{\partial\ddot{x}_{ta}^i}
+\ddot{x}_b^a
\tfrac{\partial }{\partial\ddot{x}_{tb}^r}
-\ddot{x}_{t\dot{r}}^i
\tfrac{\partial }{\partial\ddot{x}_{ta}^i}
-\dot{x}^a\ddot{x}_{b\dot{r}}^i
\tfrac{\partial }{\partial\ddot{x}_{tb}^i}
+\ddot{x}_{\dot{b}}^a
\tfrac{\partial }{\partial\ddot{x}_{t\dot{b}}^r}
-\ddot{x}_{\dot{r}}^i
\tfrac{\partial }{\partial\ddot{x}_{t\dot{a}}^i}\\
-\dot{x}^a\ddot{x}_{\dot{r}\dot{b}}^i
\tfrac{\partial }{\partial\ddot{x}_{t\dot{b}}^i}
-\ddot{x}_{\dot{r}b}^i
\tfrac{\partial }{\partial\ddot{x}_{a\leq b}^i}
-\ddot{x}_{b\dot{r}}^i
\tfrac{\partial }{\partial\ddot{x}_{b\leq a}^i}
-\ddot{x}_{\dot{r}\dot{b}}^i
\tfrac{\partial }{\partial\ddot{x}_{a\dot{b}}^i},
\end{multline*}
\begin{align*}
\varkappa_{a\leq b}^r
& =\tfrac{1}{1+\delta_{ab}}
\left\{
2\dot{x}^a\dot{x}^b
\tfrac{\partial }{\partial\ddot{x}^r}
+\ddot{x}^b
\tfrac{\partial }{\partial\ddot{x}_a^r}
+\ddot{x}^a
\tfrac{\partial }{\partial\ddot{x}_b^r}
-\dot{x}^b\ddot{x}_{\dot{r}}^k
\tfrac{\partial }{\partial\ddot{x}_a^k}
-\dot{x}^a\ddot{x}_{\dot{r}}^k
\tfrac{\partial }{\partial \ddot{x}_b^k}
+2\dot{x}^b
\tfrac{\partial }{\partial\ddot{x}_{\dot{a}}^r}
\right.  \\
& +2\dot{x}^a
\tfrac{\partial }{\partial\ddot{x}_{\dot{b}}^r}
+\ddot{x}_t^b
\tfrac{\partial }{\partial\ddot{x}_{ta}^r}
+\ddot{x}_t^a
\tfrac{\partial }{\partial\ddot{x}_{tb}^r}
-\dot{x}^b\ddot{x}_{t\dot{r}}^i
\tfrac{\partial }{\partial\ddot{x}_{ta}^i}
-\dot{x}^a\ddot{x}_{t\dot{r}}^i
\tfrac{\partial }{\partial\ddot{x}_{tb}^i}
+\ddot{x}_{c}^a
\tfrac{\partial }{\partial\ddot{x}_{c\leq b}^r}\\
& +\ddot{x}_{c}^b
\tfrac{\partial }{\partial\ddot{x}_{c\leq a}^r}
+\ddot{x}_{c}^b
\tfrac{\partial }{\partial\ddot{x}_{a\leq c}^r}
+\ddot{x}_{c}^a
\tfrac{\partial }{\partial\ddot{x}_{b\leq c}^r}
-\dot{x}^b\ddot{x}_{\dot{r}c}^i
\tfrac{\partial }{\partial\ddot{x}_{a\leq c}^i}
-\dot{x}^a\ddot{x}_{\dot{r}c}^i
\tfrac{\partial }{\partial\ddot{x}_{b\leq c}^i}\\
& -\dot{x}^a\ddot{x}_{c\dot{r}}^i
\tfrac{\partial }{\partial\ddot{x}_{c\leq b}^i}
-\dot{x}^b\ddot{x}_{c\dot{r}}^i
\tfrac{\partial }{\partial\ddot{x}_{c\leq a}^i}
-\ddot{x}_{\dot{r}}^i
\tfrac{\partial }{\partial\ddot{x}_{a\dot{b}}^i}
-\ddot{x}_{\dot{r}}^i
\tfrac{\partial }{\partial\ddot
{x}_{b\dot{a}}^i}+\ddot{x}_{\dot{c}}^b
\tfrac{\partial }{\partial\ddot{x}_{a\dot{c}}^r}\\
& \left.
+\ddot{x}_{\dot{c}}^a
\tfrac{\partial }{\partial\ddot{x}_{b\dot{c}}^r}
-\dot{x}^b\ddot{x}_{\dot{r}\dot{c}}^i
\tfrac{\partial }{\partial \ddot{x}_{a\dot{c}}^i}
-\dot{x}^a\ddot{x}_{\dot{r}\dot{c}}^i
\tfrac{\partial }{\partial\ddot{x}_{b\dot{c}}^i}
\right\}
-\ddot{x}_r^i
\tfrac{\partial }{\partial\ddot{x}_{a\leq b}^i}
+2\tfrac{\partial }
{\partial\ddot{x}_{\dot{a}\leq\dot{b}}^r},
\end{align*}
\[
\begin{array}
[c]{l}
\varkappa_{ttt}^r
=\frac{\partial }
{\partial\ddot{x}_t^r}
-\ddot{x}_{\dot{r}}^i
\frac{\partial }
{\partial\ddot{x}_{tt}^i},\\
\varkappa_{tta}^r=2\dot{x}^a
\frac{\partial }
{\partial\ddot{x}_t^r}
+\frac{\partial }
{\partial\ddot{x}_a^r}
+\ddot{x}^a
\frac{\partial }
{\partial\ddot{x}_{tt}^r}
-\dot{x}^a\ddot{x}_{\dot{r}}^i
\frac{\partial }
{\partial\ddot{x}_{tt}^i}
-\ddot{x}_{\dot{r}}^i
\frac{\partial }
{\partial\ddot{x}_{ta}^i}
+2\frac{\partial }
{\partial\ddot{x}_{t\dot{a}}^r},
\end{array}
\]
\begin{align*}
\varkappa _{t,a\leq b}^r
& =\tfrac{1}{1+\delta _{ab}}
\left\{
2\dot{x}^a\dot{x}^b
\tfrac{\partial }
{\partial\ddot{x}_t^r}
+2\dot{x}^b
\tfrac{\partial }
{\partial\ddot{x}_a^r}
+2\dot{x}^a
\tfrac{\partial }
{\partial\ddot{x}_b^r}
+\ddot{x}^b
\tfrac{\partial }
{\partial\ddot{x}_{ta}^r}
+\ddot{x}^a
\tfrac{\partial }
{\partial\ddot{x}_{tb}^r}
-\dot{x}^b\ddot{x}_{\dot{r}}^i
\tfrac{\partial }
{\partial\ddot{x}_{ta}^i}
\right.
\\
& \left.
-\dot{x}^a\ddot{x}_{\dot{r}}^i
\tfrac{\partial }
{\partial\ddot{x}_{tb}^i}
+2\dot{x}^b
\tfrac{\partial }
{\partial\ddot{x}_{t\dot{a}}^r}
+2\dot{x}^a
\tfrac{\partial }
{\partial\ddot{x}_{t\dot{b}}^r}
+2\tfrac{\partial }
{\partial\ddot{x}_{a\dot{b}}^r}
+2\tfrac{\partial }
{\partial\ddot{x}_{b\dot{a}}^r}
\right\}
-\sum _{a\leq b}
\ddot{x}_{\dot{r}}^i
\tfrac{\partial }
{\partial\ddot{x}_{a\leq b}^i},
\end{align*}
\begin{align*}
\varkappa_{tttt}^r
& =\tfrac{\partial }
{\partial\ddot{x}_{tt}^r},\\
\varkappa_{ttta}^r
& =2\dot{x}^a\tfrac{\partial }
{\partial \ddot{x}_{tt}^r}
+\tfrac{\partial }
{\partial \ddot{x}_{ta}^r},\\
\varkappa _{tt,a\leq b}^r
& =\tfrac{1}{1+\delta _{ab}}
\left(
2\dot{x}^a\dot{x}^b
\tfrac{\partial }
{\partial\ddot{x}_{tt}^r}
+2\dot{x}^b
\tfrac{\partial }
{\partial\ddot{x}_{ta}^r}
+2\dot{x}^a
\tfrac{\partial }
{\partial\ddot{x}_{tb}^r}
\right)
+\tfrac{\partial }
{\partial\ddot{x}_{a\leq b}^r},
\end{align*}
\begin{align*}
\sum _{a\leq b\leq c}
\tfrac{\partial ^3u^r}
{\partial x^a\partial x^b\partial x^c}
\varkappa_{a\leq b\leq c}^r
& =\sum _{a\leq b}
\tfrac{\partial ^3u^r}
{\partial x^a\partial x^b\partial x^h}
\left(
\ddot{x}^h
\tfrac{\partial }
{\partial\ddot{x}_{a\leq b}^r}
-\dot{x}^h\ddot{x}_{\dot{r}}^i
\tfrac{\partial }
{\partial\ddot{x}_{a\leq b}^i}
\right)
\\
& +\tfrac{\partial ^3u^r}
{\partial x^a\partial x^b\partial x^c}
\left(
\dot{x}^b\dot{x}^c
\tfrac{\partial }
{\partial\ddot{x}_a^r}
+2\dot{x}^c
\tfrac{\partial }
{\partial\ddot{x}_{a\dot{b}}^r}
\right) ,
\end{align*}
\[
\sum _{a\leq b\leq c}
\tfrac{\partial ^4u^r}
{\partial t\partial x^a\partial x^b\partial x^c}
\varkappa_{t,a\leq b\leq c}^r
=\tfrac{\partial ^4u^r}
{\partial t\partial x^a\partial x^b\partial x^c}
\dot{x}^c\dot{x}^b
\tfrac{\partial }{\partial\ddot{x}_{ta}^r}
+2\sum _{a\leq b}
\tfrac{\partial ^4u^r}
{\partial t\partial x^a\partial x^b\partial x^c}
\dot{x}^c
\tfrac{\partial }
{\partial\ddot{x}_{a\leq b}^r},
\]
\[
\sum _{a\leq b\leq c\leq d}
\tfrac{\partial ^4u^r}
{\partial x^a\partial x^b\partial x^c\partial x^d}
\varkappa_{a\leq b\leq c\leq d}^r
=\sum _{a\leq b}
\tfrac{\partial ^4u^r}
{\partial x^a\partial x^b\partial x^c\partial x^d}
\dot{x}^c\dot{x}^d
\tfrac{\partial }{\partial\ddot{x}_{a\leq b}^r}.
\]
From the expressions for $\varkappa _{tttt}^r$,
$\varkappa _{ttta}^r$, $\varkappa _{tt,a\leq b}^r$
above we deduce
\[
\begin{array}
[c]{rl}
\frac{\partial }{\partial\ddot{x}_{ta}^r}
= & \!\!\!\varkappa _{ttta}^r
-2\dot{x}^a\varkappa _{tttt}^r,\\
\frac{\partial }{\partial\ddot{x}_{a\leq b}^r}
= & \!\!\!\varkappa _{tt,a\leq b}^r
-\frac{2}{1+\delta _{ab}}
\left(
\dot{x}^b\varkappa _{ttta}^r
+\dot{x}^a\varkappa_{tttb}^r
-3\dot{x}^a\dot{x}^b\varkappa _{tttt}^r
\right) .
\end{array}
\]
Hence the vector fields $\varkappa_{t,a\leq b\leq c}^r$
and $\varkappa _{a\leq b\leq c\leq d}^r$ can be written
as linear combinations of $\varkappa_{tttt}^r$,
$\varkappa_{ttta}^r$, and $\varkappa_{tt,a\leq b}^r$;
namely,
\[
\begin{array}
[c]{rl}
\varkappa _{t,a\leq b\leq c}^r
\equiv 0 & \!\!\!
\operatorname{mod}
\left\langle
\varkappa _{tttt}^r,
\varkappa _{ttta}^r,
\varkappa _{tttb}^r,
\varkappa _{tttc}^r,
\varkappa _{tt,a\leq b}^r,
\varkappa _{tt,a\leq c}^r,
\varkappa _{tt,b\leq c}^r
\right\rangle ,\\
\varkappa _{a\leq b\leq c\leq d}^r\
equiv 0 & \!\!\!
\operatorname{mod}
\left\langle
\varkappa _{tttt}^r,
\varkappa _{ttta}^r,
\varkappa _{tttb} ^r,
\varkappa _{tttc}^r,
\varkappa _{tttd}^r,
\varkappa _{tt,a\leq b}^r,
\varkappa _{tt,a\leq c}^r,
\varkappa _{tt,a\leq d}^r,
\right.  \\
& \qquad
\qquad
\qquad
\qquad
\qquad
\qquad
\qquad
\left.
\varkappa _{tt,b\leq c}^r,
\varkappa _{tt,b\leq d}^r,
\varkappa _{tt,c\leq d}^r
\right\rangle ,
\end{array}
\]
and, as a computation shows, the vector fields
$\varkappa _{a\leq b\leq c}^r$ can be written
as linear combinations of $\varkappa _{tta}^r$,
$\varkappa _{t,a\leq b}^r$, $\varkappa _{tttt}^r$,
$\varkappa _{ttta}^r$,
and $\varkappa _{tt,a\leq b}^r$; namely,
\[
\begin{array}
[c]{rl}
\varkappa _{a\leq b\leq c}^r
\equiv 0 & \!\!\!
\operatorname{mod}
\left\langle
\varkappa _{ttt}^r,
\varkappa _{tta}^r,
\varkappa _{ttb}^r,
\varkappa _{ttc}^r,
\varkappa _{t,a\leq b}^r,
\varkappa _{t,a\leq c}^r,
\varkappa _{t,b\leq c}^r,
\varkappa _{tttt}^i,
\varkappa _{ttta}^i,
\varkappa _{tttb}^i,
\right.  \\
& \qquad
\qquad
\qquad
\qquad
\qquad
\qquad
\qquad
\left.
\varkappa _{tttc}^i,
\varkappa _{tt,a\leq b}^i,
\varkappa _{tt,a\leq c}^i,
\varkappa _{tt,b\leq c}^i
\right\rangle .
\end{array}
\]

Moreover, from the previous formulas,
we obtain
\[
\!\!\!
\begin{array}
[c]{rl}
\varkappa _t^r
\!=\! & \!\!\!
\frac{\partial }{\partial\dot{x}^r}
\!-\!\ddot{x}_{r\dot{a}}^i
\frac{\partial }
{\partial\ddot{x}_{t\dot{a}}^i}
\!-\!\ddot{x}_r^i(\varkappa_{ttt}^i
\!+\!\ddot{x}_{\dot{i}}^k\varkappa_{tttt}^k)
\!-\!2\ddot{x}_{tr}^i\varkappa_{tttt}^i
\!-\!\ddot{x}_{ar}^i
\left(
\varkappa_{ttta}^i
\!-\!2\dot{x}^a\varkappa_{tttt}^i
\right)
\!,\smallskip \\
\varkappa_{tt}^r
\!=\! & \!\!\!
\tfrac{\partial }{\partial\ddot{x}^r}
\!-\!\ddot{x}_{\dot{r}\dot{b}}^i
\tfrac{\partial }
{\partial\ddot{x}_{t\dot{b}}^i}
\!-\!\ddot{x}_{\dot{r}}^i(\varkappa_{ttt}^i
\!+\!\ddot{x}_{\dot{i}}^k\varkappa_{tttt}^k)
\!-\!\left(
\ddot{x}_r^i\!+\!2\ddot{x}_{t\dot{r}}^i
\right)
\varkappa_{tttt}^i
\!-\!\ddot{x}_{b\dot{r}}^i
\left(
\varkappa_{tttb}^i
\!-\!2\dot{x}^b\varkappa_{tttt}^i
\right)
\!,
\end{array}
\]
\begin{align*}
\varkappa_a^r
\! & =\!
{\small \dot{x}}^a
\tfrac{\partial }{\partial\dot{x}^r}
{\small +\ddot{x}}^a
\tfrac{\partial }{\partial\ddot{x}^r}
+{\small \ddot{x}}_b^a
\tfrac{\partial }{\partial\ddot{x}_b^r}
+{\small \ddot{x}}_{\dot{b}}^a
\tfrac{\partial }
{\partial\ddot{x}_{\dot{b}}^r}
+{\small \ddot{x}}_{t\dot{b}}^a
\tfrac{\partial }
{\partial\ddot{x}_{t\dot{b}}^r}
+\ddot{x}_{c\dot{b}}^a
\tfrac{\partial }
{\partial\ddot{x}_{c\dot{b}}^r}
\!+\!\sum _{c\leq b}
\ddot{x}_{\dot{c}\leq\dot{b}}^a
\tfrac{\partial }
{\partial\ddot{x}_{\dot{c}\leq\dot{b}}^r}\\
& -\!\ddot{x}_r^i
\tfrac{\partial }{\partial\ddot{x}_a^i}
-\ddot{x}_{\dot{r}}^i
\tfrac{\partial }
{\partial\ddot{x}_{\dot{a}}^i}
-\ddot
{x}_{t\dot{r}}^i
\tfrac{\partial }
{\partial\ddot{x}_{t\dot{a}}^i}
-\ddot{x}_{r\dot{b}}^i
\tfrac{\partial }
{\partial\ddot{x}_{a\dot{b}}^i}
-\ddot{x}_{c\dot{r}}^i
\tfrac{\partial }
{\partial\ddot{x}_{c\dot{a}}^i}
\!-\!\sum _{a\leq b}
\ddot{x}_{\dot{b}\dot{r}}^i
\tfrac{\partial }
{\partial \ddot{x}_{\dot{a}\leq\dot{b}}^i}\\
& -\!\sum _{b\leq a}\ddot{x}_{\dot{b}\dot{r}}^i
\tfrac{\partial }
{\partial\ddot{x}_{\dot{b}\leq\dot{a}}^i}
\!+\!\sum _{c\leq b}\ddot{x}_{cb}^a
\left(
\!\varkappa_{tt,c\leq b}^r
\!-\!\tfrac{2}{1+\delta_{cb}}
\left(
\dot{x}^b\varkappa_{tttc}^r
\!+\!\dot{x}^c\varkappa_{tttb}^r
\!-\!3\dot{x}^c\dot{x}^b\varkappa _{tttt}^r
\right)
\!
\right) \\
& -\! \sum _{c\leq a}\ddot{x}_{cr}^i
\left(
\varkappa_{tt,c\leq a}^i
\!-\!\tfrac{2}{1+\delta_{ca}}
\left(
\dot{x}^a\varkappa_{tttc}^i
\!+\!\dot{x}^c\varkappa_{ttta}^i
\!-\!3\dot{x}^c\dot{x}^a\varkappa _{tttt}^i
\right)
\!
\right) \\
& -\sum _{a\leq b}\ddot{x}_{rb}^i
\left(
\!\varkappa _{tt,a\leq b}^i
\!-\!\tfrac{2}{1+\delta _{ab}}
\left(
\dot{x}^a\varkappa_{tttb}^i
\!+\!\dot{x}^b\varkappa _{ttta}^i
\!-\!3\dot{x}^b\dot{x}^a
\varkappa _{tttt}^i
\right)
\!
\right) \\
& +\!\ddot{x}_t^a
\left(
\varkappa _{ttt}^r
\!+\!\ddot{x}_{\dot{r}}^k\varkappa _{tttt}^k
\right)
\!+\!\ddot{x}_{tt}^a\varkappa_{tttt}^r\!
-\!\ddot{x}_{tr}^i
\left(
\varkappa_{ttta}^i
\!-\!2\dot{x}^a\varkappa_{tttt}^i
\right) \\
& +\!\ddot{x}_{tb}^a
\left(
\varkappa _{tttb}^r\!-\!2\dot{x}^b
\varkappa _{tttt}^r
\right) ,
\end{align*}
\begin{align*}
\varkappa_{ta}^r\!
& =\!\dot{x}^a\!
\left(
\varkappa _{tt}^r
\!+\!\ddot{x}_r^i\varkappa _{tttt}^i
\right)
\!+\!\ddot{x}^a
\left(
\varkappa _{ttt}^r
\!+\!\ddot{x}_{\dot{r}}^i\varkappa _{tttt}^i
\right)
\!+\!2\ddot{x}_t^a\varkappa_{tttt}^r\!-\!\ddot{x}_r^i
\left(
\varkappa_{ttta}^i\!-\!2\dot{x}^a\varkappa_{tttt}^i
\right) \\
&  +\!\ddot{x}_b^a\!
\left(
\varkappa_{tttb}^r
\!-\!2\dot{x}^b\varkappa_{tttt}^r
\right)
\!+\!\dot{x}^a
\tfrac{\partial }{\partial\ddot{x}^r}
\!-\!\ddot{x}_{\dot{r}}^i
\tfrac{\partial }{\partial\ddot{x}_a^i}
\!+\!2\tfrac{\partial }
{\partial\ddot{x}_{\dot{a}}^r}
\!-\!\ddot{x}_{\dot{r}\dot{c}}^i
\tfrac{\partial }
{\partial\ddot{x}_{a\dot{c}}^i}
\!-\!\ddot{x}_{\dot{r}}^i
\tfrac{\partial }
{\partial\ddot{x}_{t\dot{a}}^i}
\!+\!\ddot{x}_{\dot{b}}^a
\tfrac{\partial }
{\partial\ddot{x}_{t\dot{b}}^r}\\
& -\!\ddot{x}_{t\dot{r}}^i\!
\left(
\varkappa_{ttta}^i
\!-\!2\dot{x}^a\varkappa_{tttt}^i
\right)  \!-\!\ddot{x}_{\dot{r}c}^i
\left(
\!\varkappa_{tt,a\leq c}^i
\!-\!\tfrac{2}{1+\delta_{ac}}
\left(
dot{x}^c\varkappa_{ttta}^i
\!+\!\dot{x}^a\varkappa_{tttc}^i
\!-\!3\dot{x}^a\dot{x}^c\varkappa_{tttt}^i
\right)
\!
\right) \\
&  -\! \ddot{x}_{c\dot{r}}^i\!
\left(
\! \varkappa_{tt,c\leq a}^i
\! -\!\tfrac{2}{1+\delta_{ac}}
\left(
\dot{x}^c\varkappa_{ttta}^i
\!+\!\dot{x}^a\varkappa_{tttc}^i
\!-\!3\dot{x}^a\dot{x}^c\varkappa_{tttt}^i
\right)
\!
\right)  ,
\end{align*}
\begin{align*}
\varkappa _{a\leq b}^r\!
&  =\!\tfrac{1}{1+\delta _{ab}}
\left[
\ddot{x}^b
\tfrac{\partial }{\partial \ddot{x}_a^r}
\!+\!\ddot{x}^a
\tfrac{\partial }{\partial\ddot{x}_b^r}
\!+\!\ddot{x}_{\dot{r}}^i
\left(
\dot{x}^b
\tfrac{\partial }
{\partial\ddot{x}_{t\dot{a}}^i}
\!+\!\dot{x}^a
\tfrac{\partial }
{\partial \ddot{x}_{t\dot{b}}^i}
\!-\!\tfrac{\partial }
{\partial \ddot{x}_{a\dot{b}}^i}
\!-\!\tfrac{\partial }
{\partial \ddot{x}_{b\dot{a}}^i}
\right)
\right. \\
& +\ddot{x}_{\dot{c}}^a
\left(
\!\tfrac{\partial }
{\partial \ddot{x}_{b\dot{c}}^r}
\!-\!\dot{x}^b
\tfrac{\partial }
{\partial \ddot{x}_{t\dot{c}}^r}
\!
\right)
\!+\!\ddot{x}_{\dot{c}}^b
\left(
\!\tfrac{\partial }
{\partial \ddot{x}_{a\dot{c}}^r}
\!-\!\dot{x}^a
\tfrac{\partial }{\partial \ddot{x}_{t\dot{c}}^r}
\!
\right)
\!-\!\dot{x}^b
\left\{
\dot{x}^a
\left(
\varkappa _{tt}^r\!
+\!\ddot{x}_r^i\varkappa _{tttt}^i
\right)
\right. \\
&  \left.
+\ddot{x}^a
\left(
\varkappa _{ttt}^r\!
+\!\ddot{x}_{\dot{r}}^i\varkappa _{tttt}^i
\right)
\!+\!2\ddot{x}_t^a\varkappa _{tttt}^r
\!-\!\ddot{x}_r^i
\left(
\varkappa _{ttta}^i
\!-\!2\dot{x}^a\varkappa _{tttt}^i
\right)
\right\} \\
& +\ddot{x}_{c}^a
\left(
\varkappa _{tttc}^r
\!-\! 2\dot{x}^c\varkappa _{tttt}^r
\right)
\!+\!\dot{x}^b\varkappa _{ta}^r
\!+\!\dot{x}^a\varkappa _{tb}^r\\
& \!-\dot{x}^a
\left\{
\dot{x}^b
\left(
\varkappa _{tt}^r
\!+\!\ddot{x}_r^i\varkappa _{tttt}^i
\right)
\!+\!\ddot{x}^b
\left(
\varkappa _{ttt}^r
\!+\!\ddot{x}_{\dot{r}}^i\varkappa _{tttt}^i
\right)
\!+\!2\ddot{x}_t^b\varkappa _{tttt}^r
\right. \\
& \left.
-\ddot{x}_r^i
\left(
\varkappa _{tttb}^i
\!-\!2\dot{x}^b\varkappa _{tttt}^i
\right)
\!+\!\ddot{x}_c^b
\left(
\varkappa _{tttc}^r
\!-\!2\dot{x}^c
\varkappa _{tttt}^r
\right)
\right\} \\
&  \! +\ddot{x}_t^b
\left(
\varkappa _{ttta}^r
\!-\!2\dot{x}^a
\varkappa _{tttt}^r
\right)
\!+\!\ddot{x}_t^a
\left(
\varkappa _{tttb}^r
\!-\!2\dot{x}^b
\varkappa _{tttt}^r
\right) \\
& +\ddot{x}_c^a
\left(
\!\varkappa _{tt,c\leq b\!}^r
\!-\!\tfrac{2}{1+\delta _{cb}}
\left(
\dot{x}^b\varkappa _{tttc}^r
\!+\!\dot{x}^c\varkappa_{tttb}^r
\!-\!3\dot{x}^c\dot{x}^b
\varkappa _{tttt}^r
\right)
\!
\right) \\
& +\ddot{x}_{c}^b
\left(
\!\varkappa _{tt,c\leq a}^r
-\tfrac{2}{1+\delta _{ca}}
\left(
\dot{x}^a\varkappa _{tttc}^r
+\dot{x}^c\varkappa _{ttta}^r
-3\dot{x}^c\dot{x}^a\varkappa _{tttt}^r
\right)
\!
\right) \\
&  +\ddot{x}_c^b
\left(
\!\varkappa _{tt,a\leq c}^r
-\tfrac{2}{1+\delta _{ca}}
\left(
\dot{x}^a\varkappa _{tttc}^r
+\dot{x}^c\varkappa _{ttta}^r
-3\dot{x}^c\dot{x}^a
\varkappa _{tttt}^r
\right)
\!\right) \\
&  \left.
+\ddot{x}_c^a
\left(
\!\varkappa _{tt,b\leq c}^r
\!-\!\tfrac{2}{1+\delta _{cb}}
\left(
\dot{x}^b\varkappa _{tttc}^r
\!+\!\dot{x}^c\varkappa _{tttb}^r
\!-\!3\dot{x}^c\dot{x}^b
\varkappa _{tttt}^r
\right)
\!
\right)
\!
\right] \\
&  -\ddot{x}_r^i
\left(
\! \varkappa _{tt,a\leq b}^i
\!-\!\tfrac{1}{1+\delta _{ab}}
\left(
2\dot{x}^b\varkappa _{ttta}^i
\!+\!2\dot{x}^a\varkappa _{tttb}^i
\!-\!6\dot{x}^a\dot{x}^b
\varkappa _{tttt}^i
\right)
\!
\right)
\!+\!2\tfrac{\partial }
{\partial \ddot{x}_{\dot{a}\leq\dot{b}}^r},
\end{align*}
\[
\begin{array}
[c]{l}
\varkappa_{ttt}^r
=\frac{\partial }
{\partial \ddot{x}_t^r}
-\ddot{x}_{\dot{r}}^i\varkappa _{tttt}^i,
\smallskip\\
\varkappa_{tta}^r
=2\dot{x}^a\varkappa _{ttt}^r
+\ddot{x}^a\varkappa_{tttt}^r
+3\dot{x}^a\ddot{x}_{\dot{r}}^i\varkappa _{tttt}^i
-\ddot{x}_{\dot{r}}^i\varkappa_{ttta}^i
+\frac{\partial }
{\partial \ddot{x}_a^r}
+2\frac{\partial }
{\partial\ddot{x}_{t\dot{a}}^r},
\end{array}
\]
\begin{align*}
\varkappa_{t,a\leq b}^r
& =\tfrac{1}{1+\delta _{ab}}
\left\{
\dot{x}^b
\left(
\varkappa _{tta}^r
-\dot{x}^a\varkappa _{ttt}^r
-\ddot{x}^a\varkappa _{tttt}^r
\right)
+\dot{x}^a
\left(
\varkappa _{ttb}^r
-\dot{x}^b\varkappa _{ttt}^r
-\ddot{x}^b\varkappa _{tttt}^r
\right)
\right. \\
& +\ddot{x}^b
\left(
\varkappa _{ttta}^r
-2\dot{x}^a\varkappa _{tttt}^r
\right)
+\ddot{x}^a
\left(
\varkappa_{tttb}^r
-2\dot{x}^b\varkappa_{tttt}^r
\right)
+\dot{x}^b
\tfrac{\partial }
{\partial\ddot{x}_a^r}
+\dot{x}^a
\tfrac{\partial }
{\partial\ddot{x}_b^r}\\
&
\left.
+2\tfrac{\partial }
{\partial\ddot{x}_{a\dot{b}}^r}
+2\tfrac{\partial }
{\partial\ddot{x}_{b\dot{a}}^r}
\right\} \\
&  -\sum _{a\leq b}\ddot{x}_{\dot{r}}^i
\left(
\varkappa _{tt,a\leq b}^i
-\tfrac{2}{1+\delta_{ab}}
\left(
\dot{x}^b\varkappa _{ttta}^i
+\dot{x}^a\varkappa _{tttb}^i
-3\dot{x}^a\dot{x}^b\varkappa_{tttt}^i
\right)
\right)  ,
\end{align*}
\[
\begin{array}
[c]{rl}
\varkappa _{tttt}^r
= & \!\!\!
\frac{\partial }{\partial\ddot{x}_{tt}^r},
\smallskip\\
\varkappa _{ttta}^r
= & \!\!\!2\dot{x}^a\varkappa _{tttt}^r
+\frac{\partial }{\partial\ddot{x}_{ta}^r},
\smallskip\\
\varkappa _{tt,a\leq b}^r
= & \!\!\!\tfrac{2}{1+\delta _{ab}}
\left(
\dot{x}^b\varkappa_{ttta}^r
+\dot{x}^a\varkappa_{tttb}^r
-3\dot{x}^b\varkappa_{tttt}^r
\right)
+\frac{\partial }
{\partial\ddot{x}_{a\leq b}^r}.
\end{array}
\]

Hence for $n\geq 2$,
\[
\mathcal{D}^{(2)}
=\left\langle
\varkappa _{00}^r,
\varkappa _t^r,
\varkappa _a^r,
\varkappa _{tt}^r,
\varkappa _{ta}^r,
\varkappa _{a\leq b}^r,
\varkappa _{ttt}^r,
\varkappa _{tta}^r,
\varkappa _{t,a\leq b}^r,
\varkappa _{tttt}^r,
\varkappa _{ttta}^r,
\varkappa _{tt,a\leq b}^r
\right\rangle
\]
and for $n\!=\!1$,
$\mathcal{D}^{(2)}
=\left\langle \varkappa _{00},
\varkappa _t,
\varkappa _{tt},
\varkappa _{t1},
\varkappa _{11},
\varkappa _{ttt},
\varkappa _{tt1},
\varkappa _{t,11},
\varkappa _{tttt},
\varkappa _{ttt1},
\varkappa_{tt,11}
\right\rangle $.
\end{proof}

According to the formulas \eqref{K_sigma},
\eqref{T's}, \eqref{P's} the tensor field
$K^\sigma $ takes values in the vector
bundle $(p^{10})^\ast \bigwedge ^2T^\ast M^0
\otimes V(p^{10})$ and we can define
the curvature mapping
$\mathcal{K}\colon J^2(p^{21})
\to (p^{10})^\ast \bigwedge ^2T^\ast M^0
\otimes V(p^{10})$ by setting
$\mathcal{K}(j_\xi ^2\sigma )
=\left( K^\sigma \right) _\xi $.

\begin{theorem}
Every second-order differential invariant
$\mathcal{I}\colon J^2(p^{21})\to \mathbb{R}$
with respect to $\mathrm{Aut}^v(p)$ factors
uniquely through the curvature mapping
as follows:
$\mathcal{I}=\mathcal{\bar{I}}\circ \mathcal{K}$,
where
$\mathcal{\bar{I}}\colon(p^{10})^\ast
\bigwedge ^2T^\ast M^0\otimes V(p^{10})\to \mathbb{R}$
is an invariant smooth function under
the natural action of $\mathrm{Aut}^v(p)$, namely,
\begin{equation}
\Phi \cdot \eta
=\left(
\bigwedge ^2((\Phi^{(1)})^{-1})^\ast
\otimes(\Phi ^{(1)})_\ast
\right)
(\eta ),
\quad
\forall
\eta \in (p^{10})^\ast
\bigwedge ^2T^\ast M^0\otimes V(p^{10}).
\label{action}
\end{equation}
\end{theorem}

\begin{proof}
The statement is an immediate consequence
of the following properties:

\begin{enumerate}
\item
The curvature mapping is a surjective submersion.

\item
The fibre $\mathcal{K}^{-1}(\eta )$ for every
$\eta \in \bigwedge ^2 T_{(t_0,x_0)}^\ast M^0
\otimes V_\xi (p^{10})$ is an affine subbundle
over $J^1(p^{21})$; in particular, the fibres
of $J^2(p^{21})$ are connected.

\item
If we define
\[
\mathcal{\tilde{D}}_{j_\xi ^2\sigma }^{(2)}
=\left\{
\begin{array}
[c]{ll}
\left\{
(X^{(2)})_{j_\xi ^2\sigma }^{(2)}
\in \mathcal{D}_{j_\xi ^2\sigma }^{(2)}
:j_{j_{t_0}^0\gamma }^1X=0
\right\}  ,
& \text{if }n=\dim M\geq 2\\
\left\{
(X^{(2)})_{j_\xi ^2\sigma }^{(2)}
\in \mathcal{D}_{j_\xi ^2\sigma}^{(2)}:
X_\xi ^{(1)}=0
\right\}  ,
& \text{if }n=\dim M=1
\end{array}
\right.
\]
then
$\ker(\mathcal{K}_\ast )_{j_\xi ^2\sigma }
=\mathcal{\tilde{D}}_{j_\xi ^2\sigma }^{(2)}$,
$\xi =j_{t_0}^1\gamma $.

\item
The curvature mapping is
$\mathrm{Aut}^v(p)$-equivariant
with respect to the natural actions, i.e.,
$\Phi \cdot \mathcal{K}(j_\xi ^2\sigma )
=\mathcal{K}(\Phi \cdot j_\xi ^2\sigma )$,
where the action on the left-hand side
is defined in \eqref{action}
and that on the right-hand side is given
as in the beginning of this section, i.e.,
$\Phi \cdot j_\xi ^2\sigma
=j_{\Phi ^{(1)}(\xi )}^2(\Phi ^{(2)}
\circ \sigma \circ (\Phi ^{(1)})^{-1})$.
\end{enumerate}

Coordinates are introduced in
$(p^{10})^\ast \bigwedge ^2T^\ast M^0\otimes
V(p^{10})$ by setting
\[
\eta =\left\{
y_i^j(\eta )
\left(
dt\wedge \omega ^i
\right)  _{(t_0,x_0)}
+\sum _{h<i}y_{hi}^j(\eta )
\left(
\omega^h\wedge\omega^i
\right) _{(t_0,x_0)}
\right\}
\otimes
\left(
\frac{\partial }{\partial \dot{x}^j
}\right) _\xi ,
\]
for every
$\eta \in \bigwedge ^2T_{(t_0,x_0)}^\ast
M^0\otimes V_\xi (p^{10})$.
The first and second properties directly
follow from the equations
of the curvature mapping, i.e.,
\begin{align*}
t\circ \mathcal{K}
& =t,
\;
x^i\circ \mathcal{K}=x^i,
\;
\dot{x}^i\circ \mathcal{K}
=\dot{x}^i,\\
y_a^i\circ \mathcal{K}
& =-\tfrac{1}{2}
\ddot{x}_{t\dot{a}}^i
-\tfrac{1}{2}
\dot{x}^h\ddot{x}_{h\dot{a}}^i
-\tfrac{1}{2}
\ddot{x}^h\ddot{x}_{\dot{h}\dot{a}}^i
+\ddot{x}_a^i
+\tfrac{1}{4}
\ddot{x}_{\dot{a}}^k\ddot{x}_{\dot{k}}^i,\\
y_{ab}^k\circ \mathcal{K}
& =-\tfrac{1}{2}
\ddot{x}_{a\dot{b}}^k
+\tfrac{1}{2}
\ddot{x}_{b\dot{a}}^k
-\tfrac{1}{4}
\ddot{x}_{\dot{a}}^h\ddot{x}_{\dot{h}\dot{b}}^k
+\tfrac{1}{4}
\ddot{x}_{\dot{b}}^h\ddot{x}_{\dot{h}\dot{a}}^k,
\; a<b.
\end{align*}

Moreover, from the formulas
\eqref{v}--\eqref{w_apunto_bpunto}
we deduce
\begin{align*}
(X^{(2)})^{(2)}
\left(
y_j^i\circ \mathcal{K}
\right)
& =\frac{\partial u^i}{\partial x^r}
\left(
y_j^r\circ \mathcal{K}
\right)
-\frac{\partial u^r}{\partial x^j}
\left(
y_r^i\circ \mathcal{K}
\right) ,\\
(X^{(2)})^{(2)}
\left(
y_{ij}^k\circ\mathcal{K}
\right)
& =\frac{\partial u^r}{\partial x^i}
\left(
y_{rj}^k\circ \mathcal{K}
\right)
+\frac{\partial u^k}{\partial x^r}
\left(
y_{ji}^r\circ\mathcal{K}
\right)
+\frac{\partial u^r}{\partial x^j}
\left(
y_{ir}^k\circ\mathcal{K}
\right)  .
\end{align*}
By evaluating these two formulas at $j_\xi ^2\sigma $,
we conclude $\mathcal{\tilde{D}}_{j_\xi ^2\sigma }^{(2)}
\subseteq \ker (\mathcal{K}_{\ast })_{j_\xi ^2\sigma}$
and from the Lemma \ref{rank} and the first item above
we have
\begin{align*}
\dim  \ker (\mathcal{K}_\ast )_{j_\xi ^2\sigma}
& =\dim  J^2(p^{21})-\dim
\left(
(p^{10})^\ast \bigwedge ^2T^\ast M^0\otimes V(p^{10})
\right)
\\
& =\tfrac{3}{2}n(n+2)(n+1)\\
& =\dim \mathcal{\tilde{D}}_{j_\xi ^2\sigma }^{(2)}.
\end{align*}
Finally, the fourth item above follows by using
the formulas \eqref{x_dot}--\eqref{Phixdot}
and the fact that
$T^{\Phi \cdot \sigma }
(\Phi ^{(1)}\cdot X,\Phi ^{(1)}\cdot Y)
=\Phi^{(1)}\cdot T^\sigma (X,Y)$,
$\forall X,Y\in \mathfrak{X}(M^1)$,
as follows from Corollary \ref{CorollaryFunctoriality1}.
\end{proof}

\begin{remark}
As the distribution $\mathcal{D}^{(2)}$ is involutive,
the number of functionally independent second-order
differential invariants is $\frac{1}{2}n^2(n-1)+1$
if $n\geq 2$, and $2$ if $n=1$. By passing
to the quotient, the isomorphism \eqref{isomorphism}
induces another isomorphism
\[
\begin{array}
[c]{l}
\iota _1\colon T^-(M^1)
\overset{\cong}{\longrightarrow}
\left.
(p^{10})^\ast TM^0
\right/ (
p^{10})_\ast T^0(M^1),
\smallskip \\
\iota _1
\left(
X_i^\sigma
\right)
=\dfrac{\partial }{\partial x^i}
\operatorname{mod}(p^{10})_\ast T^0(M^1).
\end{array}
\]
Moreover, as $M^0=\mathbb{R}\times M$, there is
a natural embedding $(p^\prime )^\ast
TM\hookrightarrow TM^0$ and pulling it back via
$p^{10}$ we obtain another embedding
$(p^\prime \circ p^{10})^\ast
TM\hookrightarrow (p^{10})^\ast TM^0$.
By composing this latter embedding
and the quotient map
\[
(p^{10})^\ast TM^0\to
\left.
(p^{10})^\ast TM^0
\right/
(p^{10})_\ast T^0(M^1),
\]
an isomorphism
$\iota _2\colon (p^\prime \circ p^{10})^\ast TM
\overset{\cong}{\longrightarrow}\left.
(p^{10})^\ast TM^0\right/ (p^{10})_\ast T^0(M^1)$
is deduced. From the formula \eqref{K_sigma}
it follows $i_{X^\sigma }K^\sigma
=-P_j^h\omega ^j\otimes \partial /\partial \dot{x}^h$,
and we define an endomorphism $\tilde{K}^\sigma
\colon(p^\prime \circ p^{10})^\ast TM
\to (p^\prime \circ p^{10})^\ast TM$ by setting
\[
\begin{array}
[c]{l}
\tilde{K}^\sigma
=\varepsilon ^{-1}\circ
\left.
i_{X^\sigma }K^\sigma
\right\vert _{T^-(M^1)}
\circ (\iota _1)^{-1}
\circ \iota _2,
\smallskip \\
\tilde{K}^\sigma
\left(
\dfrac{\partial }{\partial x^j}
\right)
=-P_j^h
\dfrac{\partial }{\partial x^h}.
\end{array}
\]
The coefficients of the characterestic polynomial
of $\tilde{K}^\sigma $ determine $n$ second-order
invariants. This fact was remarked for the first
time in \cite{Kosambi2}. If $n=1$ or $n=2$ then
these invariants together with the function $t$
exhaust a basis of second-order invariants,
but this is no longer true for $n\geq 3$.

Finally, we should also like to remark that
$\mathcal{K}(j^2_\xi \sigma )$ depends only
on $j^1_\xi (\tilde{K}^\sigma )$,
as follows from the following identities
among the components of the torsion tensor
field of the Chern connection:
\[
3T_{kj}^i=\frac{\partial P_j^i}{\partial\dot{x}^k}
-\frac{\partial P_k^i}{\partial\dot{x}^j}.
\]
Therefore, the Kosambi tensor field
$\tilde{K}^\sigma $ encodes all the relevant
information for KCC theory.
\end{remark}


\begin{thebibliography}{99}
\bibitem{AM}
D. V. Alekseevsky, P. W. Michor,
\emph{Characteristic classes of }$G$\emph{-structures},
Differential Geom.\ Appl.\ \textbf{3} (1993), no.\ 4,
323--329.

\bibitem{Anastasiei}
M. Anastasiei,
\emph{A Historical Remark on the Connections
of Chern and Rund}, in ``Finsler Geometry'' (Seattle,
WA, 1995), pp.\ 171--176, Contemp.\ Math. \textbf{196},
Amer. Math. Soc., Providence, RI, 1996.

\bibitem{AB}
P. L. Antonelli, I. Bucataru,
\emph{KCC theory of a system of second order
differential equations} in ``Handbook of Finsler geometry'',
Volume 1, pp.\ 83--174, Kluwer Acad. Publ., Dordrecht, 2003.

\bibitem{BaoChernShen}
D. Bao, S.-S. Chern, Z. Shen,
\emph{An introduction to Riemann-Finsler geometry},
Graduate Texts in Mathematics, 200,
Springer-Verlag, New York, 2000.

\bibitem{Bleecker}
D. Bleecker, \emph{Gauge theory and variational principles},
Global Analysis Pure and Applied Series A, 1,
Addison-Wesley Publishing Co., Reading, Mass., 1981.

\bibitem{Byrnes1}
G. B. Byrnes,
\emph{A complete set of Bianchi identities
for tensor fields along the tangent bundle
projection}, J. Phys.\ A: Math.\ Gen.\
\textbf{27} (1994), 6617--6632.

\bibitem{Byrnes2}
G. B. Byrnes,
\emph{A linear connection for higher-order
ordinary differential equations}, J. Phys. A:
Math.\ Gen.\ \textbf{29} (1996), 1685--1694.

\bibitem{Cartan}
\'{E}lie Cartan, D. D. Kosambi,
\emph{Observations sur la m\'{e}moire pr\'{e}c\'{e}dent},
Math.\ Z. \textbf{37} (1933), no.\ 1, 619--622.

\bibitem{Chern}
S.-S. Chern, \emph{Sur la g\'{e}om\'{e}trie
d'un syst\`{e}me d'\'{e}quations diff\'{e}rentielles
du second ordre}, Bull.\ Sci.\ Math.\ \textbf{63} (1939),
206--212.

\bibitem{Chern2}
S.-S. Chern, \emph{The geometry of $G$-structures},
Bull.\ Amer.\ Math.\ Soc.\ \textbf{72} (1966), 167--219.

\bibitem{Chern2.5}
S.-S. Chern, \emph{Complex manifolds without potential theory},
Van Nostrand Mathematical Studies, No.\ 15, D. Van Nostrand Co.,
Inc., Princeton, N.J.-Toronto, Ont.-London, 1967.

\bibitem{Chern3}
S.-S. Chern, \emph{Riemannian geometry as a special case
of Finsler geometry}, in ``Finsler Geometry'' (Seattle,
WA, 1995), pp.\ 51--57, Contemp.\ Math. \textbf{196},
Amer. Math. Soc., Providence, RI, 1996.

\bibitem{CrampinMartinezSarlet}
M. Crampin, E. Mart\'{\i}nez, and W. Sarlet,
\emph{Linear connections for systems
of second-order ordinary differential equations}, Ann.\
Inst.\ H. Poincar\'e Phys.\ Th\'{e}or.\ \textbf{65}
(1996), no.\ 2, 223--249.

\bibitem{DGM}
J. Davidov, G. Grantcharov, O. Mu\u{s}karov,
\emph{Curvature properties of the Chern connection
of twistor spaces}, Rocky Mountain J. Math.\
\textbf{39} (2009), no.\ 1, 27--48.

\bibitem{Fujimoto}
A. Fujimoto, \emph{Theory of G-structures}.
English edition, translated from the original Japanese.
Publications of the Study Group of Geometry, Volume 1.
Study Group of Geometry, Department of Applied Mathematics,
College of Liberal Arts and Science, Okayama University,
Okayama, 1972.

\bibitem{Guillemin}
V. Guillemin,
\emph{The integrability problem for }$G$\emph{-structures},
Trans.\ Amer.\ Math.\ Soc.\ \textbf{116} (1965), no.\ 4,
544--560.

\bibitem{KN}
S.~Kobayashi, K.~Nomizu, \emph{Foundations
of Differential Geometry, Volume I}, John Wiley \& Sons,
Inc., N.Y., 1963.

\bibitem{Kosambi}
D.~D.~Kosambi, \emph{Parallelism and path-spaces},
Math.\ Z. \textbf{37} (1933), no.\ 1, 608--618.

\bibitem{Kosambi2} D.~D.~Kosambi,
\emph{Systems of differential equations
of the second order}, Quart.\ J. Math. Oxford
\textbf{6} (1935), 1--12.

\bibitem{Kumpera} A.~Kumpera, \emph{Invariants
diff\'{e}rentiels d'un pseudogroupe de Lie.\ I, II},
J. Differential Geom.\ \textbf{10} (1975), 289--345,
347--416.

\bibitem{LL}
J.~Lehmann-Lejeune,
\emph{Int\'egrabilit\'e des $G$-structures
d\'efinies par une $1$-forme $0$-d\'eformable
\`a valeurs dans le fibr\'e tangent}, Ann.\ Inst.
\ Fourier (Grenoble), \textbf{16} (1966), 329--387.

\bibitem{L}
Th.\ Leistner, \emph{Towards a classification
of Lorentzian holonomy groups},
arXiv:math/0305139v1 [math.DG].

\bibitem{Lich}
A. Lichnerowicz, \emph{Global theory of connections
and holonomy groups.\/} Translated from the French
and edited by Michael Cole. Noordhoff International
Publishing, Leiden, 1976.

\bibitem{MassaPagani}
E.~Massa, E.~Pagani, \emph{Jet bundle geometry,
dynamical connections, and the inverse problem
of Lagrangian mechanics}, Ann.\ Inst.\ H. Poincar\'{e}
Phys.\ Th\'eor.\ \textbf{61} (1994),
no.\ 1, 17--62.

\bibitem {MM}
J. Mu\~noz Masqu\'e, \emph{Formes de structure et
transformations infinit\'esimales de contact d'ordre
sup\'erieur}, C. R. Acad.\ Sc.\ Paris \textbf{298},
S\'{e}rie I, no.\ 8 (1984), 185--188.

\bibitem{Muzsnay}
Z. Muzsnay,
\emph{On the problem of linearizability of a $3$-web},
Nonlinear Anal.\ \textbf{68} (2008), no.\ 6, 1595--1602.

\bibitem{Nakai}
I. Nakai, \emph{Curvature of curvilinear $4$-webs
and pencils of one forms: variation on a theorem
of Poincar\'e, Mayrhofer and Reidemeister}, Comment.\
Math.\ Helv.\ \textbf{73} (1998), no.\ 2, 177--205.

\bibitem{YanoIshihara} K. Yano, Sh.\ Ishihara,
\emph{Tangent and cotangent bundles: differential
geometry}, Pure and Applied Mathematics, no.\ 16,
Marcel Dekker, Inc., New York, 1973.

\bibitem{Wu} H. Wu, \emph{Shiing-shen Chern:
1911--2004}, Bull.\ Amer.\ Math.\ Soc.\ (N.S.)
\textbf{46} (2009), no.\ 2, 327--338.
\end{thebibliography}
\end{document}